%% file: finalThesis2018.tex
\RequirePackage{amsmath}
\RequirePackage{fix-cm}
\pdfoutput=1
\documentclass{iitthesis}          
\usepackage{indentfirst} 
\usepackage{graphicx}
\usepackage{amssymb}
\usepackage{amsthm}

\usepackage{algorithmicx}
\usepackage{algorithm}

\usepackage{tabu}
\usepackage{cancel}

\usepackage{algpseudocode}
\usepackage{booktabs, mathtools}
\usepackage[dvipsnames]{xcolor}
\usepackage{url}
\usepackage[normalem]{ulem}
\usepackage{siunitx}
\usepackage[htt]{hyphenat}
\usepackage{hhline}
\usepackage{mathrsfs}
\usepackage{colortbl}
\algdef{SE}[DOWHILE]{Do}{doWhile}{\algorithmicdo}[1]{\algorithmicwhile\ #1}%

\DeclareMathOperator{\Order}{{\mathcal O}}

\newtheorem{defn}{Definition}

\newcommand{\bm}[1]{\boldsymbol{#1}}
\newcommand{\mSigma}{\mathsf{\Sigma}}
\newcommand{\mB}{\mathsf{B}}

\newcommand{\dif}[1]{\text{d}{#1}}
\newcommand{\D}[1]{\text{d}{#1}}
\newcommand{\trace}[1]{\textup{trace}{#1}}

\newcommand{\naturals}{\mathbb{N}}
\newcommand{\natzero}{\mathbb{N}_0}
\newcommand{\reals}{\mathbb{R}}
\newcommand{\integers}{\mathbb{Z}}

\newcommand{\Ex}{\mathbb{E}}

\newcommand{\rC}{\mathring{C}}
\newcommand{\rlambda}{\mathring{\lambda}}

\newcommand{\vbeta}{{\bm{\beta}}}
\newcommand{\vDelta}{{\boldsymbol{\Delta}}}
\newcommand{\veta}{{\bm{\eta}}}
\newcommand{\vlambda}{{\bm{\lambda}}}

\newcommand{\vtheta}{{\bm{\theta}}}
\newcommand{\vzeta}{{\bm{\zeta}}}

\newcommand{\va}{\bm{a}}
\newcommand{\vA}{\bm{A}}
\newcommand{\vb}{\bm{b}}
\newcommand{\vc}{\bm{c}}
\newcommand{\vC}{\bm{C}}

\newcommand{\vg}{\bm{g}}
\newcommand{\vh}{\bm{h}}
\newcommand{\vf}{\bm{f}}
\newcommand{\vk}{\bm{k}}
\newcommand{\vl}{\bm{l}}
\newcommand{\vm}{\bm{m}}

\newcommand{\vt}{\bm{t}}
\newcommand{\vv}{\bm{v}}
\newcommand{\vV}{\bm{V}}
\newcommand{\vw}{\bm{w}}

\newcommand{\vx}{\bm{x}}
\newcommand{\dx}{\dif{{x}}}
\newcommand{\dt}{\dif{{t}}}
\newcommand{\dvx}{\dif{\bm{x}}}

\newcommand{\dvt}{\dif{\bm{t}}}
\newcommand{\vrho}{\bm{\rho}}

\newcommand{\vy}{\bm{y}}
\newcommand{\vY}{\bm{Y}}

\newcommand{\vz}{\bm{z}}
\newcommand{\vZ}{\bm{Z}}
\newcommand{\dvz}{\dif{\bm{z}}}

\newcommand{\vPsi}{\boldsymbol{\Psi}}

\newcommand{\tvy}{\tilde{\bm{y}}}

\newcommand{\vone}{\bm{1}}

\newcommand{\mA}{\mathsf{A}}
\newcommand{\mC}{\mathsf{C}}
\newcommand{\mG}{\mathsf{G}}
\newcommand{\mH}{\mathsf{H}}
\newcommand{\mK}{\mathsf{K}}
\newcommand{\mP}{\mathsf{P}}
\newcommand{\rmC}{\mathring{\mathsf{C}}}
\newcommand{\mCtheta}{{\mathsf{C}_{\vtheta}}}
\newcommand{\mCthetaInv}{{\mathsf{C}^{-1}_{\vtheta}}}
\newcommand{\mCInv}{{\mathsf{C}^{-1}}}
\newcommand{\cov}{{\textup{cov}}}
\newcommand{\var}{{\textup{var}}}

\newcommand{\mL}{\mathsf{L}}

\newcommand{\mLambda}{\mathsf{\Lambda}}

\newcommand{\mV}{\mathsf{V}}
\newcommand{\mW}{\mathsf{W}}

\newcommand{\calN}{\mathcal{N}}
\newcommand{\me}{\mathrm{e}}

\newcommand{\hmu}{\widehat{\mu}}
\newcommand{\hsigma}{\widehat{\sigma}}

\newcommand{\MLE}{\textup{EB}}

\newcommand{\full}{\textup{full}}
\newcommand{\GCV}{\textup{GCV}}
\newcommand{\opt}{\textup{opt}}
\newcommand{\CI}{\textup{CI}}
\newcommand{\NICE}{\textup{nice}}
\newcommand{\PEAKY}{\textup{peaky}}
\newcommand{\NOISE}{\textup{noise}}
\newcommand{\TRUE}{\textup{smooth}}
\newcommand{\errtol}{\varepsilon}

\newcommand{\diag}{\text{diag}}
\newcommand{\err}{\textup{err}}
\newcommand{\code}[1]{\texttt{#1}}

\newcommand{\ia}{2^{l+1}q}
\newcommand{\ib}{2^{l}(2q+1)-1}
\newcommand{\ic}{2^{l}(2q+1)}
\newcommand{\id}{2^{l+1}(q+1)-1}

\newcommand{\ja}{2^{l+1}s}
\newcommand{\jb}{2^{l}(2s+1)-1}
\newcommand{\jc}{2^{l}(2s+1)}
\newcommand{\jd}{2^{l+1}(s+1)-1}

\def\abs#1{\ensuremath{\left \lvert #1 \right \rvert}}
\newcommand{\norm}[2][{}]{\ensuremath{\left \lVert #2 \right \rVert}_{#1}}

\DeclarePairedDelimiter{\ceil}{\lceil}{\rceil}

\providecommand{\argmin}{\operatorname*{argmin}}

\newcommand\figref{Figure~\ref}
\newcommand\secref{Section~\ref}

\graphicspath{{./figures/}{D:/Mega/MyWriteupBackup/Sep_2ndweek_1/}}

\newcommand{\JRNote}[1]{}

\newtheorem{theorem}{Theorem}[section]
\newtheorem{lemma}[theorem]{Lemma}

\makeatletter
\renewcommand*\env@matrix[1][\arraystretch]{%
	\edef\arraystretch{#1}%
	\hskip -\arraycolsep
	\let\@ifnextchar\new@ifnextchar
	\array{*\c@MaxMatrixCols c}}
\makeatother

\allowdisplaybreaks[4]
\begin{document}
\setlength\abovedisplayskip{0pt}
\setlength{\belowdisplayskip}{0pt}

\title{Fast Automatic Bayesian Cubature Using Matching Kernels and Designs
}
%



\author{Jagadeeswaran Rathinavel       
}

\degree{Doctor of Philosophy}
\dept{Applied Mathematics}


\date{December 2019}

\maketitle

\prelimpages         

\begin{acknowledgement}     
	\par  I want to thank my advisor Prof. Fred J Hickernell for his support and guidance in my completion of this thesis and throughout my studies here at IIT. 
	His support and motivation have given me the confidence to endure through the research. 
	
	I would like to also thank the GAIL project collaborators with
	whom I have worked to add my new algorithms to the GAIL MATLAB toolbox: Prof. Sou-Cheng Choi,
	Yuhan Ding, Lan Jiang, Xin Tong, and Kan Zhang. Especially, Prof. Sou-Cheng Choi's support and guidance as the project leader helped me to focus on my cubature algorithms.
	
	My special gratitude also goes to my thesis committee members, Prof. Jinqiao Duan,
	Prof. Fred J Hickernell, Prof. Shuwang Li, and Prof. Geoffrey Williamson. Above all, I want to thank them because they were flexible and willing to dedicate time to review my work and attend my comprehensive and defense examinations.
	
	I would like to thank Prof. Dirk Nuyens for suggestions, valuable tips and notes when we were researching higher order nets and kernels.
	
	I would like to thank the organizers of the SAMSI-Lloyds-Turing Workshop on Probabilistic Numerical Methods, where a part of preliminary version of this work was discussed.  I also thank Prof. Chris Oates and Prof. Sou-Cheng Choi for valuable comments.
	
	I would like to specifically thank my friend Samuel Davidson for reviewing and suggesting the improvements on the text.
	
	Last but not least, I would not be able to make it without the support of my family. I would like to thank my wife for her continuous support and sacrifice. I also would like to thank my parents for their endless support.


\end{acknowledgement}

\tableofcontents

\clearpage

\listoftables

\clearpage


\listoffigures

\clearpage

\begin{abstract}

Automatic cubatures approximate multidimensional integrals to user-specified error tolerances. 
In many real-world integration problems, the analytical solution is either unavailable or difficult to compute.
To overcome this, one can use numerical algorithms that approximately estimate the value of the integral. 

For high dimensional integrals, quasi-Monte Carlo (QMC) methods are very popular.
QMC methods are equal-weight quadrature rules where the quadrature points are chosen deterministically, unlike Monte Carlo (MC) methods where the points are chosen randomly.
The families of integration lattice nodes and digital nets are the most popular quadrature points used. 
These methods consider the integrand to be a deterministic function.
An alternative approach, called Bayesian cubature, postulates the integrand to be an instance of a Gaussian stochastic process.  

For high dimensional problems, it is difficult to adaptively change the sampling pattern. But one can automatically determine the sample size, $n$, given a fixed and reasonable sampling pattern. We take this approach using a Bayesian perspective.
We assume a Gaussian process parameterized by a constant mean and a covariance function defined by a scale parameter and a function specifying how the integrand values at two different points in the domain are related.
These parameters are estimated from integrand values or are given non-informative priors. This leads to a credible interval for the integral.  The sample size, $n$, is chosen to make the credible interval for the Bayesian posterior error no greater than the desired error tolerance. 

However, the process just outlined typically requires vector-matrix operations  with a computational cost of $O(n^3)$. Our innovation is to pair low discrepancy nodes with matching kernels, which lowers the computational cost to $O(n \log n)$. 
We begin the thesis by introducing the Bayesian approach to calculate the posterior cubature error and define our automatic Bayesian cubature (Chapter~\ref{sec:BC}). Although much of this material is known, it is used to develop the necessary foundations.
Some of the major contributions of this thesis include the following:
	1) The fast Bayesian transform is introduced. This generalizes the techniques that speedup Bayesian cubature when the kernel matches low discrepancy nodes.
	2) The fast Bayesian transform approach is demonstrated using two methods: a) rank-1 lattice sequences and shift-invariant kernels, and b) Sobol' sequences and Walsh kernels.
	These two methods are implemented as fast automatic Bayesian cubature algorithms in the Guaranteed Automatic Integration Library (GAIL).
	3) We develop additional numerical implementation techniques: 
	a) rewriting the covariance kernel to avoid cancellation error, 
	b) gradient descent for hyperparameter search, and
	c) non-integer kernel order selection.
The thesis concludes by applying our fast automatic Bayesian cubature algorithms to three sample integration problems. We show that our algorithms are faster than the basic Bayesian cubature and that they provide answers within the error tolerance in most cases. A significant portion of this thesis comprising an automatic Bayesian cubature algorithm using lattice sequences and shift-invariant kernels was published and discussed in \cite{JagHic09a, HicJag09a}.

The Bayesian cubatures that we develop are guaranteed for integrands belonging to cone of functions which reside in the middle of the sample space. The concept of a cone of functions is also explained briefly. 

\end{abstract}

\textpages     

\Chapter{INTRODUCTION}

\Section{Cubature}\label{intro}

Cubature is the problem of inferring a numerical value for a definite integral, 
$ \mu := \int_{\reals^d} g(\vx) \, \dif \vx$, where $\mu$ has no closed form analytic expression. Typically, $g$ is accessible through a black-box function routine. 
Cubature means numerical multivariate integration and is a key component of many problems in scientific computing, finance \cite{Gla03}, statistical modeling, imaging \cite{Keller2013}, uncertainty quantification, machine learning \cite{Goodfellow-et-al-2016}, etc.

The integral may often be expressed as
\begin{equation}
\label{eqn:defn_mu}
\mu:= \mu(f) := \Ex[f(\boldsymbol{X})] = \int_{[0,1]^d} f(\vx)\, \dvx, 
\end{equation}
where $f:[0,1]^d \to \reals$ is the integrand, and $\boldsymbol{X} \sim \mathcal{U}[0,1]^d$.  The process of transforming the original integral into the form of \eqref{eqn:defn_mu} is addressed in \cite{BecHae92b, Sid08a, Sid93, Lau96a, CriEtal07}.  The cubature may be an affine function of integrand values:
\begin{equation}
\label{eqn:defn_hmu}  
\hmu: = \hmu(f) := w_0 + \sum_{i=1}^{n} f(\vx_i) w_i, \quad \mathcal{P} :=\{\vx_i\}_{i=1}^n \subset [0,1]^d
\end{equation}
where the weights, $w_0$, and  $\vw = (w_i)_{i=1}^n \in \reals^n$, and the nodes, $\mathcal{P}$, are chosen to make the error, $\abs{\mu - \hmu}$, small. The integration domain $[0,1]^d$ is convenient for the low discrepancy node sets that we use.  The nodes are assumed to be deterministic.
The integral of function $f$ is the same over $[0, 1]^d$ or $(0, 1)^d$ or $[0, 1)^d$. So we use $[0, 1]^d$ or $[0, 1)^d$ depending on the application. Most often $[0, 1)^d$ is preferred especially for extensible node-sets because  it partitions easily into congruent subhypercubes.
This research focuses on multivariate numerical integrals where the computational cost is a bottleneck. 

\pagebreak

\Section{Stopping Criterion}

We construct a reliable stopping criterion that determines the number of integrand values required, $n$, to ensure that the error is no greater than a user-defined error tolerance denoted by $\varepsilon$, i.e., 
\begin{equation}
\label{eqn:err_crit} 
\abs{\mu - \hmu} \leq \errtol .
\end{equation}
Rather than relying on strong assumptions about the integrand, such as an upper bound on its variance or total variation, we construct a stopping criterion that is based on a credible interval arising from a Bayesian approach to the problem.  We build upon the work of Briol et al.~\cite{BriEtal18a}, Diaconis~\cite{Dia88a}, O'Hagan~\cite{OHa91a}, Ritter~\cite{Rit00a}, Rasmussen and Ghahramani~\cite{RasGha03a}, and others.  Our algorithm is an example of \emph{probabilistic numerics}.
To study numerical algorithms from a statistical point of view, where uncertainty is formally due to the presence of an unknown numerical error, is the goal of probabilistic numerics.

\JRNote{Briefly explain Probabilistic numeric ..}


Our primary contribution in this research is to demonstrate how the choice of a family of covariance kernels that match the low discrepancy sampling nodes facilitates fast computation of the cubature and the data-driven stopping criterion.  Our Bayesian cubature requires a computational cost of
\begin{equation} \label{eqn:OuralgoCost}
\Order\bigl(n \$(f) + N_{\opt}[n\$(C) + 
n \log(n)] \bigr),
\end{equation} 
where $\$(f)$ is the cost of one integrand value, $\$(C)$ is the cost of a single covariance kernel value,  $\Order(n \log(n))$ is the cost of a fast Bayesian transform, and $N_{\opt}$ is an upper bound on the number of optimization steps required to choose the hyperparameters. If function evaluation is expensive, e.g., the output of a computationally intensive simulation, or if $\$(f) = \Order(d)$ for large $d$, then $\$(f)$ might be similar in magnitude to $N_{\opt} \log(n)$ in practice.  Typically, $\$(C) = \Order(d)$.  Note that the $\Order(n \log(n))$ contribution is $d$ independent.

In contrast to our fast algorithm, the typical computational cost for Bayesian cubature is
\begin{equation} \label{eqn:TheiralgoCost}
\Order\bigl(n \$(f) + N_{\opt}[n^2\$(C) + n^3] \bigr),
\end{equation} 
which is explained in Section~\ref{sec:bayes_cubature_algo}. Note that apart from evaluating the integrand, the computational cost in \eqref{eqn:TheiralgoCost} is much larger than that in \eqref{eqn:OuralgoCost}.  


\Section{Low Discrepancy Points}

Low discrepancy points are characterized by how uniformly the points are distributed, which is measured by the \emph{discrepancy}. 
The goal is to have maximum uniform space filling. The discrepancy is defined as below.
Let $\mathcal{M}$ be the set of all intervals of the form $\prod_{\ell=1}^{d} [a_\ell, b_\ell) = \{ \vx \in \mathbb{R}^d : a_\ell \le x_\ell \le b_\ell, 0 \le a_\ell \le b_\ell \le 1 \}$. Then, the discrepancy of a point set $\mathcal{P}$ is,
\begin{align*}
D(\mathcal{P}) := \sup_{M \in \mathcal{M}} 
\abs{ \frac{\abs{M \cap \mathcal{P}}}{\abs{\mathcal{P}}} - \lambda_L(M) },
\end{align*}
where $\abs{\mathcal{P}}$ is the cardinality of the set $\mathcal{P}$, and $\lambda_L$ is the Lebesgue measure.
The \emph{low discrepancy points} satisfy $D(\mathcal{P}) = \Order((\log n)^d/n)$.
In this work we experiment with two most popular low discrepancy point sets, 1) lattice points, and 2) Sobol' points.

\Section{Prior Work}

Hickernell~\cite{Hic17a} compares different approaches to cubature error analysis depending on whether the rule is deterministic or random and whether the integrand is assumed to be deterministic or random.  Error analysis that assumes a deterministic integrand lying in a Banach space leads to an error bound that is typically impractical for deciding how large $n$ must be to satisfy \eqref{eqn:err_crit}.  The deterministic error bound includes a (semi-)norm of the integrand, which is often more complex to compute than the original integral.

Hickernell and Jim\'enez-Rugama~\cite{HicJim16a,JimHic16a} have developed stopping criteria for cubature rules based on low discrepancy nodes by tracking the decay of the discrete Fourier coefficients of the integrand.  The algorithms proposed here also rely on discrete Fourier coefficients, but in a different way.  We only discuss automatic Bayesian cubature for absolute error tolerances in this thesis. The recent work by Hickernell, Jim\'enez-Rugama, and Li~\cite{HicEtal17a} suggests how one might accommodate more general error criteria, such as relative error tolerances which has been adapted in the MATLAB implementation of our algorithms.

Chapter~\ref{sec:BC} explains the Bayesian approach to calculate the posterior cubature error and defines our automatic Bayesian cubature. Although much of this material is known, it is included for completeness.  We end Chapter~\ref{sec:BC}  by demonstrating why Bayesian cubature is typically computationally expensive.
Chapter~\ref{sec:fast_BC}  introduces the concept of covariance kernels that match the nodes and expedite the computations required by our automatic Bayesian cubature. 
Chapter~\ref{sec:shift_invariant_kernel} implements this concept for shift invariant kernels and rank-1 lattice nodes. It also develops approaches to build shift-invariant kernels of continuous valued kernel order rather than fixing the kernel order to integer values.
Chapter~\ref{sec:sobol_walsh} demonstrates another implementation of matching nodes and kernel using Sobol' points and Walsh kernels. It also shows that the fast Walsh Hadamard as the fast Bayesian transform for this case.
Chapter~\ref{sec:NumImpl} describes how to avoid cancellation error for kernels of product form. 
It also covers  some of the additional techniques used in the implementation of our Bayesian Cubature algorithms. 
Numerical examples are provided in Chapter~\ref{sec:NumExp} to demonstrate the performance and advantages of our new algorithms.  We conclude with a brief discussion and potential future work in Chapter~\ref{sec:conclusion-future-work}.

We use the terms integrand or function interchangeably to denote the function $f$ being considered for the numerical integration. Also, we use the terms, nodes, points, node-sets, designs, and data-sites interchangeably to denote the points $\mathcal{P}$ used in the cubature.

\Chapter{Bayesian Cubature}
\label{sec:BC} 

The Bayesian approach for numerical analysis was popularized by Diaconis \cite{Dia88a}. The earliest reference for such kind of approach dates back to Poincar\'e, where, the theory of interpolation was discussed.
Diaconis motivates the reader by interpreting the most well known numerical methods, 1) trapezoidal rule and 2) splines, from the statistical point of view with whatever is known about the integrand as prior information. 
For example, the trapezoidal rule can be interpreted as a Bayesian method with prior information being modeled as a Brownian motion in the sample space $\mathcal{C}[0,1)$, the space of continuous functions. 

This research is focused on the Bayesian approach for numerical integration that is known as Bayesian cubature as introduced by O'Hagan \cite{OHagen1991}.  
Bayesian cubature returns a probability distribution, that expresses belief about the true value of integral, $\mu(f)$.
This posterior probability distribution is based on a prior that depends on $f$, which is computed via Bayes' rule using the \emph{data} contained in the function evaluations \cite{BriEtal18a}. 
The distribution in general captures numerical uncertainty due to the fact that we have only used a finite number of function values to evaluate the integral.

\Section{Bayesian Posterior Error}
\label{sec:BayesPostErr}

We assume the integrand, $f$, is an instance of a stochastic Gaussian process, i.e., $f \sim \mathcal{GP}(m,s^2 C_\vtheta)$.  Specifically, $f$ is a real-valued random function with constant mean $m$ and covariance function $s^2C_\vtheta$, where $s$ is a positive scale factor, and $C_\vtheta: [0,1]^d \times [0,1]^d \to \mathbb{R} $ is a symmetric, positive-definite function and, parameterized by $\vtheta$:
\begin{multline} \label{FJH:eq:CondPosDef}
\mC^T = \mC,  \quad \va^T \mC \va > 0, \quad \text{where }  \mC = \left(  C_\vtheta(\vx_i,\vx_j)  \right)_{i,j=1}^n,\\
 \text{for all } \va \ne 0, \;
 n\in \mathbb{N}, \; \text{distinct} \; \vx_1, \ldots, \vx_n \in [0,1]^d.
\end{multline}
The covariance function, $C$, and the Gram matrix, $\mC$, depend implicitly on $\vtheta$, but the notation may omit this for simplicity's sake.
Procedures for estimating or integrating out the hyperparameters $m$, $s$, and $\vtheta$ are explained later in this section.

For a Gaussian process, all vectors of linear functionals of $f$ have a multivariate Gaussian distribution. 
For any deterministic sampling scheme with distinct nodes, $\{\vx_i\}_{i=1}^n$, and defining  $\vf  := \left( f(\vx_i)\right)_{i=1}^n$ as the multivariate Gaussian vector of function values, it follows from the definition of a Gaussian process that 
\begin{subequations} \label{eqn:fGaussDist}
\begin{align}
\vf  & \sim \calN(m \vone, s^2 \mC), \\
\mu & \sim \calN(m, s^2 c_0), 
\\
\text{where }
c_0 &:= \int_{[0,1]^{d} \times [0,1]^{d}} C_\vtheta(\vx,\vt) \, \dif{\vx} \, \dif{\vt}, \\
\cov(\vf, \mu) &= \left(  \int_{[0,1]^d} C(\vt,\vx_i) \, \D \vt \right)_{i=1}^n  =: \vc.
\end{align}
\end{subequations}
Here, $c_0$ and $\vc$ depend implicitly on $\vtheta$.  We assume the covariance function $C$ is simple enough that the integrals in these definitions can be computed analytically.
We need the following lemma to derive the posterior error of our cubature. 

\begin{lemma} \cite[(A.6), (A.11--13)]{RasWil06a} \label{thrm:condDist} If $\vY = (\vY_1, \vY_2)^T \sim \calN (\vm,\mC)$, where $\vY_1$ and $\vY_2$ are random vectors of arbitrary length, and 
	\begin{gather*}
	\vm = \begin{pmatrix} \vm_1 \\ \vm_2 \end{pmatrix} = \begin{pmatrix} \Ex(\vY_1) \\ \Ex(\vY_2) \end{pmatrix}, \\
	\mC = \begin{pmatrix}
	\mC_{11} & \mC_{21}^T \\ 	\mC_{21} & \mC_{22}
	\end{pmatrix} =
	\begin{pmatrix}
	\var(\vY_{1}) & \cov(\vY_{1}, \vY_2) \\ 	\cov(\vY_2,\vY_{1}) & \var(\vY_{2})
	\end{pmatrix} 
	\end{gather*}
	then 
	\begin{align*}
	\vY_1 \vert \vY_2 \; \sim \; \calN \bigl(\vm_1 + \mC_{21}^T \mC_{22}^{-1}(\vY_2 - \vm_2), \quad \mC_{11} - \mC_{21}^T \mC_{22}^{-1} \mC_{21} \bigr).
	\end{align*}
Moreover, the inverse of the matrix $\mC$ may be partitioned as
\begin{gather*}
\mC^{-1} = \begin{pmatrix} \mA_{11} & \mA_{21}^T \\ \mA_{21} & \mA_{22} \end{pmatrix}, \\
\mA_{11} = (\mC_{11} - \mC_{12} \mC_{22}^{-1} \mC_{21})^{-1}, \qquad 
\mA_{21} = -  \mC_{22}^{-1} \mC_{21} \mA_{11}, \\ 
\mA_{22} = \mC_{22}^{-1} + \mC_{22}^{-1} \mC_{21} \mA_{11} \mC_{21}^T \mC_{22}^{-1}.
\end{gather*}
\end{lemma}

It follows from Lemma \ref{thrm:condDist} that the \emph{conditional} distribution of the integral given observed function values, $\vf = \vy$ is also Gaussian:
\begin{align} \label{eqn:condInteg}
\mu | (\vf = \vy) \sim \calN \bigl(m (1 - \vc^T \mC^{-1} \vone)  + \vc^T \mC^{-1} \vy, \;\;
s^2(c_0  -\vc ^T \mC^{-1} \vc) \bigr).
\end{align}
The natural choice for  the cubature is the posterior mean of the integral, namely, 
\begin{equation}
\label{eqn:BayesCub}
\widehat{\mu}  \vert ( \vf = \vy)
= m(1 - \vone^T  \mC^{-1}\vc )
+ \vc^T \mC^{-1} \vy,
\end{equation}
which takes the form of \eqref{eqn:defn_hmu}.
Under this definition, the cubature error has zero mean and a variance depending on the choice of nodes:
\begin{equation*}
(\mu-\hmu) | (\vf = \vy)
 \sim  \calN 
\left(
0, \quad
s^2 (c_0 - \vc^T\mC^{-1}\vc) 
\right).
\end{equation*}
A credible interval for the integral is given by 
\begin{subequations} \label{eqn_prob_confidence_interval}
	\begin{gather}
	\mathbb{P}_f \left[
	|\mu-\hmu| \leq \err_{\CI}
	\right] = 99\%, \\
	\err_{\CI} = 2.58 s \sqrt{c_{0} - \vc^T\mC^{-1}\vc}.
	\end{gather}
\end{subequations}
Naturally, $2.58$ and $99\%$ can be replaced by other quantiles and credible levels.

\Section{Hyperparameter Estimation}
\label{sec:stopping_criteria}

The credible interval in \eqref{eqn_prob_confidence_interval} suggests how our automatic Bayesian cubature proceeds.  Integrand data is accumulated until the width of the credible interval, $\err_{\CI}$, is no greater than the error tolerance.  As $n$ increases, one expects $c_{0} - \vc^T\mC^{-1}\vc$ to decrease for well-chosen nodes, $\{\vx_i\}_{i=1}^n$. Please note that the credible interval depends on the parameters $m, s$, and $\vtheta$

Note that $\err_{\CI}$ has no explicit dependence on the integrand values, even though one would intuitively expect that a larger integrand should imply a larger $\err_{\CI}$.  This is because the hyperparameters, $m, s$, and $\vtheta$, have not yet been inferred from integrand data.  After inferring the hyperparameters, $\err_{\CI}$ \emph{does reflect the size} of the integrand values. The following next few sections describe three approaches to hyperparameter estimation.

\Section{Empirical Bayes}  \label{sec:MLE}

The first and a very straight forward approach is to estimate the parameters via maximum likelihood estimation.  The log-likelihood function of the parameters given the function data $\vy$ is:
\begin{align*}
l(s,m,\vtheta | \vy)
&= -\frac{1}{2} s^{-2} (\vy-m\vone)^T\mCInv(\vy-m\vone) 
\\ & \qquad
 - \frac{1}{2} \log(\det\, \mC) - \frac{n}{2} \log(s^2) + \text{constants.}
\end{align*}
Maximizing the log-likelihood first with respect to $m$, then with respect to $s$, and finally with respect to $\vtheta$ yields
\begin{align*}
m_\MLE &= \frac{\vone^T \mCInv \vy }{ \vone^T \mCInv \vone}, \\
\nonumber
s^2_{\MLE}  
&= \frac{1}{n} (\vy-m_{\MLE}\vone)^T\mCInv(\vy-m_{\MLE}\vone) 
\\
&= 
\frac{1}{n}
\vy^T 
\left[ 
\mCInv - 
\frac{ \mCInv \vone \vone^T \mCInv }{\vone^T\mCInv \vone}
\right] \vy, \\
\vtheta_\MLE
&= \argmin_{\vtheta} \biggl \{
\log\left(\vy^T 
\left[ \mCInv - 
\frac{ \mCInv \vone \vone^T \mCInv }{\vone^T\mCInv \vone}
\right] \vy 
\right)  +  \frac{1}{n} \log(\det(\mC))
\biggr \}.
\end{align*}
The empirical Bayes estimate of $\vtheta$ balances minimizing the covariance scale factor, $s^2_{\MLE}$, against minimizing  $\det(\mC)$. 

Under these estimates of the parameters, the cubature \eqref{eqn:BayesCub} and the credible interval \eqref{eqn_prob_confidence_interval} simplify to 
\begin{align} 
\nonumber
\hmu_\MLE  &:= 
\left(
\frac{ (1 - \vone^T  \mCInv\vc )  \vone }{ \vone^T \mCInv \vone}   +  \vc 
\right)^T  \mCInv \vy, \\
\nonumber
\err_\MLE^2 & : = \frac{2.58^2}{n}
 \vy^T \left[ \mCInv - 
\frac{ \mCInv \vone \vone^T \mCInv }{\vone^T\mCInv \vone}
\right] \vy  
(c_0 - \vc^T\mC^{-1}\vc ), \\
\label{eqn_prob_CI_MLE}
\mathbb{P}_f &\left[
|\mu-\hmu_\MLE| \leq \err_\MLE \right]  = 99\%.
\end{align}
Here $c_0$, $\vc$, and $\mC$ are assumed implicitly to be based on $\vtheta = \vtheta_\MLE$.   

\Subsection{Gradient descent to find optimal shape parameter} \label{grad_descent_MLE}
The equation specifying $\vtheta_\MLE$ as defined in \eqref{eqn:thetaMLE} does not say how the parameter search can be done. There exist empirical algorithms \cite{Bre73, For77} that one could use to accomplish the same.
	Since the objective function is known we could compute the gradient.
Using the gradient of $l(s,m,\vtheta | \vy)$, one can apply optimization techniques such as gradient descent to find the optimal value faster. Let us define the objective function for the same purpose by excluding the negative sign, which modifies the problem to become a minimization of
\begin{align*}
\mathcal{L}(\vtheta | \vy)
&:= \frac{1}{n} \log(\det\, \mC) +  \log\left( (\vy-m_\MLE\vone)^T\mCInv(\vy-m_\MLE\vone) \right) + \text{constants.}
\end{align*}
Taking derivative with respect to $\theta_\ell$, for $\ell=1,\cdots,d$
\begin{align*}
\frac{\partial}{\partial \theta_\ell} \mathcal{L}(\vtheta | \vy)
&=  \frac{1}{n} \frac{\partial}{\partial \theta_\ell} \log(\det\, \mC) + \frac{\partial}{\partial \theta_\ell} 
\log\left((\vy-m_\MLE\vone)^T\mCInv(\vy-m_\MLE\vone) \right)
\\
&= \frac 1n \trace{\left( \mCInv \frac{\partial \mC}{\partial \theta_\ell} \right)}
- 
\frac
{\left((\vy-m_\MLE\vone)^T\mCInv\right)^T 
	\left( \frac{\partial \mC}{\partial \theta_\ell} \right)
	((\vy-m_\MLE\vone)^T\mCInv)}
{(\vy-m_\MLE\vone)^T\mCInv(\vy-m_\MLE\vone)}
\end{align*}
where we used some of the results from \cite{Dong2017a}. 
This can be used with gradient descent as follows,
\begin{align}
\label{eqn:deep_descent}
\theta_\ell^{(j+1)} = \theta_\ell^{(j)} - \nu_\ell \frac{\partial}{\partial \theta_\ell} \mathcal{L}(\vtheta | \vy), \quad j=0,1,\cdots
\end{align}
where $\nu_\ell$ is the step size for the gradient descent.

\Section{Full Bayes}

Rather than using maximum likelihood to determine $m$ and $s$, one can treat them as hyper-parameters with a non-informative, conjugate prior, namely $\vrho_{m,s^2}(\xi, \lambda) \propto 1/\lambda$. Then the posterior density for the integral given the data using Bayes theorem is,
\begin{align*}
\MoveEqLeft[1]{\rho_{\mu}(z | \vf = \vy)} \\
& \propto \int_{0}^\infty \int_{-\infty}^\infty \rho_{\mu}(z | \vf = \vy, m = \xi, s^2 = \lambda)  \rho_{\vf}(\vy | \xi, \lambda ) \rho_{m, s^2}(\xi, \lambda) \, \D \xi \D \lambda \\
&\qquad \qquad \qquad \qquad \text{by the properties of conditional probability} \\
& \propto \int_{0}^\infty \int_{-\infty}^\infty \rho_{\mu}(z | \vf = \vy, m = \xi, s^2 = \lambda)  \rho_{\vf}(\vy | \xi, \lambda ) \rho_{m, s^2}(\xi, \lambda) \, \D \xi \D \lambda \\
&\qquad \qquad \qquad \qquad \text{by Bayes' Theorem} \\
& \propto \displaystyle \int_{0}^\infty  \frac{1}{\lambda^{(n+3)/2}} 
 \int_{-\infty}^\infty  \exp \biggl( -\frac{1}{2\lambda}\biggl\{
\frac{
[z - \xi (1 - \vc^T \mC^{-1} \vone)  -  \vc^T \mC^{-1} \vy]^2}
{c_0  -\vc ^T \mC^{-1} \vc}  \\
& \qquad \qquad \qquad  + (\vy - \xi \vone)^T \mC^{-1}(\vy - \xi \vone) \biggr \} \biggr) \, \D \xi \D \lambda \\
&\qquad \qquad
\text{by \eqref{eqn:fGaussDist}, \eqref{eqn:condInteg}} \; \text{and} \; \rho_{m,s^2}(\xi,\lambda) \propto 1/\lambda \\
& \propto \displaystyle \int_{0}^\infty  \frac{1}{\lambda^{(n+3)/2}} \int_{-\infty}^\infty  \exp\left( -\frac{\alpha \xi^2 -2 \beta \xi + \gamma}{2\lambda(c_0  -\vc ^T \mC^{-1} \vc)} \right) \, \D \xi \D \lambda, \\
\intertext{where}
\alpha & = (1 - \vc^T \mC^{-1} \vone)^2 + \vone^T \mC^{-1} \vone (c_0  -\vc ^T \mC^{-1} \vc),\\
\beta & =(1 - \vc^T \mC^{-1} \vone)(z - \vc^T \mC^{-1} \vy )
  + \vone^T \mC^{-1} \vy (c_0  -\vc ^T \mC^{-1} \vc),\\
\gamma &  = (z - \vc^T \mC^{-1} \vy )^2  + \vy^T \mC^{-1} \vy (c_0  -\vc ^T \mC^{-1} \vc).
\end{align*}
In the derivation above and below, factors that are independent of $\xi$, $\lambda$, or $z$ can be discarded since we only need to preserve the proportion.  But, factors that depend on $\xi$, $\lambda$, or $z$ must be kept.  
Completing the square $
\alpha \xi^2 -2 \beta \xi + \gamma 
= \alpha (\xi -\beta/\alpha)^2  - (\beta^2/\alpha) + \gamma,
$
allows us to evaluate the integrals with respect to $\xi$ and $\lambda$:
\begin{align*}
\MoveEqLeft{\rho_{\mu}(z | \vf = \vy)} \\
& \propto \displaystyle \int_{0}^\infty  \frac{1}{\lambda^{(n+3)/2}}  \exp\left( -\frac{  \gamma - \beta^2/\alpha}{2\lambda(c_0  -\vc ^T \mC^{-1} \vc)} \right)  \cdots \\
& \qquad \qquad \cdots \int_{-\infty}^\infty  \exp\left( -\frac{\alpha (\xi -\beta/\alpha)^2}{2\lambda(c_0  -\vc ^T \mC^{-1} \vc)} \right) \, \D \xi \D \lambda \\
& \propto \displaystyle \int_{0}^\infty  \frac{1}{\lambda^{(n+2)/2}}  \exp\left( -\frac{  \gamma - \beta^2/\alpha}{2\lambda(c_0  -\vc ^T \mC^{-1} \vc)} \right) \D \lambda \\
& \propto \left(\gamma - \frac{\beta^2}{\alpha}\right)^{-n/2} \propto \left(\alpha \gamma - \beta^2\right)^{-n/2}.
\end{align*}
Finally, we simplify the key term:
\begin{align*}
\alpha \gamma - \beta^2 
& = \vone^T \mC^{-1} \vone (c_0  -\vc ^T \mC^{-1} \vc) (z - \vc^T \mC^{-1} \vy )^2 \\
& \qquad \qquad - 2 \vone^T \mC^{-1} \vy (c_0  -\vc ^T \mC^{-1} \vc)(1 - \vc^T \mC^{-1} \vone)(z - \vc^T \mC^{-1} \vy ) \\
& \qquad \qquad + (1 - \vc^T \mC^{-1} \vone)^2 \vy^T \mC^{-1} \vy (c_0  -\vc ^T \mC^{-1} \vc) \\
&\qquad \qquad  + [\vone^T \mC^{-1} \vone \vy^T \mC^{-1} \vy - (\vone^T \mC^{-1} \vy)^2](c_0  -\vc ^T \mC^{-1} \vc)^2  \\
& \propto \vone^T \mC^{-1} \vone  \left (z - \vc^T \mC^{-1} \vy - \frac{(1 - \vc^T \mC^{-1} \vone)\vone^T \mC^{-1} \vy}{\vone^T \mC^{-1} \vone } \right )^2 \\
& \qquad \qquad -  \frac{[(1 - \vc^T \mC^{-1} \vone)\vone^T \mC^{-1} \vy]^2}{\vone^T \mC^{-1} \vone }  
+ (1 - \vc^T \mC^{-1} \vone)^2 \vy^T \mC^{-1} \vy \\
& \qquad \qquad (c_0  -\vc ^T \mC^{-1} \vc) [\vone^T \mC^{-1} \vone  \vy^T \mC^{-1} \vy
- (\vone^T \mC^{-1} \vy)^2] \\
& \propto \left (z - \left[ \frac{(1 - \vc^T \mC^{-1} \vone)\vone}{\vone^T \mC^{-1} \vone } + \vc \right]^T \mC^{-1} \vy \right )^2 \\
& \qquad  + \left[\frac{(1 - \vc^T \mC^{-1} \vone)^2}{\vone^T \mC^{-1} \vone} + (c_0  -\vc ^T \mC^{-1} \vc) \right] \times \vy^T\left[ \mC^{-1} 
- \frac{ \mC^{-1} \vone\vone^T \mC^{-1}}{\vone^T \mC^{-1} \vone}  \right]\vy
\\
& \propto (z - \widehat{\mu}_{\textup{full}})^2 + (n-1) \sigma_{\textup{full}}^2
\\
& \propto \left(1 +  \frac{1}{n-1} \frac{(z - \mu_{\textup{full}})^2}{\widehat{\sigma}_{\textup{full}}^2} \right),
\end{align*}
i.e.,
\begin{align}
\label{eqn:full_bayes_student_short}
\alpha \gamma - \beta^2 \propto 
\left(
1 +  \frac{(z - \hmu_{\textup{full}})^2}{(n-1)\widehat{\sigma}_{\textup{full}}^2}
\right),
\end{align}
where $\hmu_{\textup{full}} = \hmu_{\MLE}$ and 
\begin{multline*}
\hsigma^2_{\textup{full}} 
:= \frac{1}{n-1}
\vy^T\left[ \mC^{-1} 
- \frac{ \mC^{-1} \vone\vone^T \mC^{-1}}{\vone^T \mC^{-1} \vone}  \right]\vy
\times  \left[\frac{(1 - \vc^T \mC^{-1} \vone)^2}{\vone^T \mC^{-1} \vone} + (c_0  -\vc ^T \mC^{-1} \vc) \right].
\end{multline*}
The confidence interval is:
\begin{equation}
\label{eqn_prob_CI_full}
\mathbb{P}_f \left[
|\mu-\hmu_\MLE| \leq \err_{\textup{full}} \right]  = 99\%,
\end{equation}
where
\begin{equation*}
\err_{\textup{full}} 
:= t_{n-1,0.995} \hsigma_{\textup{full}} > \err_\MLE .
\end{equation*}
Here $t_{n-1,0.995}$ denotes the $99.5$ percentile of a standard Student's $t$-distribution with $n-1$ degrees of freedom. 
This means that $\mu \vert (\vf = \vy )$, properly centered and scaled, has a Student's $t$-distribution with $n-1$ degrees of freedom.  
The estimated integral is the same as in the empirical Bayes case, $\hmu_{\textup{full}} = \hmu_{\MLE}$, but the credible interval is wider.
In other words, the stopping criterion for the full Bayes case is more conservative than that in the empirical Bayes case, \eqref{eqn_prob_CI_MLE}.

Because the shape parameter, $\vtheta$, enters the definition of the covariance kernel in a non-trivial way, the only way to treat it as a hyperparameter and assign a tractable prior would be for the prior to be discrete.  We believe in practice that choosing such a prior involves more guesswork than using the empirical Bayes estimate of $\vtheta$ in \eqref{eqn:thetaMLE} or the cross-validation approach described next.


\Subsection{Full Bayes with general prior}
Rather than using non-informative, conjugate prior one can use general prior, namely $\vrho_{m,s^2}(\xi, \lambda) \propto g(\lambda)$, which can generalize to any general function. One would be curious if the posterior function can be obtained from the data, i.e, the integrand values.
The posterior density for the integral given the data using Bayes theorem is,
\begin{align*}
\MoveEqLeft[1]{\rho_{\mu}(z | \vf = \vy)} \\
& \propto \int_{0}^\infty \int_{-\infty}^\infty \rho_{\mu}(z | \vf = \vy, m = \xi, s^2 = \lambda)  \rho_{\vf}(\vy | \xi, \lambda ) \rho_{m, s^2}(\xi, \lambda) \, \D \xi \D \lambda \\
&\qquad \qquad \qquad \qquad \text{by the properties of conditional probability} 
\\
& \propto \int_{0}^\infty \int_{-\infty}^\infty \rho_{\mu}(z | \vf = \vy, m = \xi, s^2 = \lambda)  \rho_{\vf}(\vy | \xi, \lambda ) \rho_{m, s^2}(\xi, \lambda) \, \D \xi \D \lambda \\
&\qquad \qquad \qquad \qquad \text{by Bayes' Theorem} 
\\
& \propto \displaystyle \int_{0}^\infty  \frac{g(\lambda)}{\lambda^{(n+1)/2}} 
\int_{-\infty}^\infty  \exp \biggl( -\frac{1}{2\lambda}\biggl\{
\frac{
	[z - \xi (1 - \vc^T \mC^{-1} \vone)  -  \vc^T \mC^{-1} \vy]^2}
{c_0  -\vc ^T \mC^{-1} \vc}  \\
& \qquad \qquad \qquad  + (\vy - \xi \vone)^T \mC^{-1}(\vy - \xi \vone) \biggr \} \biggr) \, \D \xi \D \lambda \\
&\qquad \qquad
\text{by \eqref{eqn:fGaussDist}, \eqref{eqn:condInteg}} \; \text{and} \; \rho_{m,s^2}(\xi,\lambda) \propto g(\lambda) \\
& \propto \displaystyle \int_{0}^\infty  \frac{g(\lambda)}{\lambda^{(n+1)/2}} \int_{-\infty}^\infty  \exp\left( -\frac{\alpha \xi^2 -2 \beta \xi + \gamma}{2\lambda(c_0  -\vc ^T \mC^{-1} \vc)} \right) \, \D \xi \D \lambda, \\
\intertext{where}
\alpha & = (1 - \vc^T \mC^{-1} \vone)^2 + \vone^T \mC^{-1} \vone (c_0  -\vc ^T \mC^{-1} \vc),\\
\beta & =(1 - \vc^T \mC^{-1} \vone)(z - \vc^T \mC^{-1} \vy )
+ \vone^T \mC^{-1} \vy (c_0  -\vc ^T \mC^{-1} \vc),\\
\gamma &  = (z - \vc^T \mC^{-1} \vy )^2  + \vy^T \mC^{-1} \vy (c_0  -\vc ^T \mC^{-1} \vc).
\end{align*}
In the derivation above and below, factors that are independent of $\xi$, $\lambda$, or $z$ can be discarded since we only need to preserve the proportion.  But, factors that depend on $\xi$, $\lambda$, or $z$ must be kept.  
Completing the square $
\alpha \xi^2 -2 \beta \xi + \gamma 
= \alpha (\xi -\beta/\alpha)^2  - (\beta^2/\alpha) + \gamma,
$
allows us to evaluate the integrals with respect to $\xi$ and $\lambda$:
\begin{align*}
\MoveEqLeft{\rho_{\mu}(z | \vf = \vy)} &
\\
& \propto \displaystyle \int_{0}^\infty  \frac{g(\lambda)}{\lambda^{(n+1)/2}}  \exp\left( -\frac{  \gamma - \beta^2/\alpha}{2\lambda(c_0  -\vc ^T \mC^{-1} \vc)} \right)  \cdots 
\\
& \qquad \qquad \cdots \int_{-\infty}^\infty  \exp\left( -\frac{\alpha (\xi -\beta/\alpha)^2}{2\lambda(c_0  -\vc ^T \mC^{-1} \vc)} \right) \, \D \xi \D \lambda 
\\
& \propto \displaystyle \int_{0}^\infty  \frac{g(\lambda)}{\lambda^{n/2}}  \exp\left( -\frac{  \gamma - \beta^2/\alpha}{2\lambda(c_0  -\vc ^T \mC^{-1} \vc)} \right) \D \lambda 
.
\end{align*}
This can be interpreted as Laplace transform of $g(\lambda)$,
\begin{align*}
{\rho_{\mu}(z | \vf = \vy)} 
& \propto \displaystyle \int_{0}^\infty  \frac{g(\lambda)}{\lambda^{n/2}}  \exp\left( -\frac{  \gamma - \beta^2/\alpha}{2\lambda(c_0  -\vc ^T \mC^{-1} \vc)} \right) \D \lambda 
\\
& \propto \int_{0}^\infty \frac{g(\lambda)}{\lambda^{n/2}}
\exp \left(  - \frac{1}{\lambda} \chi \right)
\dif{\lambda}, 
\\
& \quad \text{where} \quad \chi = \frac{  \gamma - \beta^2/\alpha}{2(c_0  -\vc ^T \mC^{-1} \vc)} 
\propto { 1 +  \frac{(z - \hmu_{\textup{full}})^2}{(n-1)\widehat{\sigma}_{\textup{full}}^2} }.
\end{align*}
Let $\displaystyle \lambda = \frac{1}{w}, \quad \dif{\lambda} = -w^{-2} \dif{w}$ then,
\begin{align*}
{\rho_{\mu}(z | \vf = \vy)} 
& \propto \int_{0}^\infty \frac{g(\lambda)}{\lambda^{n/2}}
\exp \left(  - \frac{1}{\lambda} \chi \right)
\dif{\lambda} 
\\
&= \int_{0}^\infty \frac{g(1/w)  }{w^{-n/2}}
\exp \left(  - w \chi \right)
(-w^{-2})\dif{w}
\\
&= \int_\infty^0 -g(1/w) w^{\frac n2 - 2}
\exp \left(  - w \chi \right)
\D{w}
\\
&= \int_0^\infty g(1/w) w^{\frac{n-4}2}
\exp \left(  - w \chi \right)
\D{w}
\\
& = \mathcal{LT} \{ g(1/\cdot) \}^{(\frac{n-4}2)}\left(\chi\right),
\end{align*}
where $\mathcal{LT}(\cdot)$ denotes the Laplace transform and $(\frac{n-4}2)$ indicates the $\frac{n-4}2$th derivative taken after the transform. Here we used frequency domain derivative property of the Laplace transform. 
The above result can be further simplified by replacing $\gamma - \beta^2/\alpha$ from  \eqref{eqn:full_bayes_student_short},
\begin{align*}
{\rho_{\mu}(z | \vf = \vy)} & \propto \mathcal{LT} \{ g(1/\cdot) \}^{(\frac{n-4}2)} \left(\chi \right) \\
& \propto
\mathcal{LT} \left\{ g({1}/{\cdot})
 \right\}^{(\frac{n-4}2)} \left( 1 +  \frac{(z - \hmu_{\textup{full}})^2}{(n-1)\widehat{\sigma}_{\textup{full}}^2}\right) \quad \text{by} \quad \eqref{eqn:full_bayes_student_short}.
\end{align*}
Thus, $\rho_{\mu}(z | \vf = \vy)$  is proportional to $(\frac{n-4}{2})$th derivative of the Laplace transform of $g(1/\cdot)$ evaluated at $\chi$, where $\chi \propto { 1 +  \frac{(z - \hmu_{\textup{full}})^2}{(n-1)\widehat{\sigma}_{\textup{full}}^2} } $.

We demonstrate the general prior with the non-informative conjugate that we used above, i.e., if $\displaystyle g(1/\lambda) = {\lambda}$ then, 
\begin{align*}
{\rho_{\mu}(z | \vf = \vy)} 
&= \int_{0}^\infty g(1/w)  w^{\frac n2 -2}
\exp \left(  - w \chi \right)
\D w
\\
& = \displaystyle \left(\mathcal{LT}(g(1/t)) \right)^{(\frac n2 -2)} \lvert_{t=\chi}
\; = \; \displaystyle \left(\mathcal{LT}(t) \right)^{(\frac n2 -2)} \lvert_{t=\chi} 
\\
& = \left(1/u^{2} \right)^{(\frac n2 -2)} \lvert_{u=\chi}
\\
& \propto \chi^{-n/2} = \left( \frac{  \gamma - \beta^2/\alpha}{2(c_0  -\vc ^T \mC^{-1} \vc)} \right)^{-n/2}
\\
& \propto \left(\gamma - \frac{\beta^2}{\alpha}\right)^{-n/2}
\\
& \propto \left(\alpha \gamma - \beta^2\right)^{-n/2},
\end{align*}
where we used the fact that the Laplace transform of $g(1/t) = t$ is $1/u^2$. 
After the transform, taking $(\frac n2 -2)$th derivative gives us the result. This shows when using a generic prior, it leads to a posterior of the form
$ {\rho_{\mu}(z | \vf = \vy)}  \propto  \mathcal{LT} \{ g(1/\cdot) \}^{(\frac{n-4}2)} \left(\chi\right) $ with full Bayes approach, i.e, the posterior $\rho_{\mu}(z | \vf = \vy)$ is a function of ${ 1 +  \frac{(z - \hmu_{\textup{full}})^2}{(n-1)\widehat{\sigma}_{\textup{full}}^2} }$.

Our motivation to experiment with the general prior was to show that it may be possible to infer the prior from the integrand samples. We demonstrated it with the non-informative prior, which shows the possibility to compute the prior from function values. Obtaining an arbitrary prior from the integrand samples is the topic of future work.


\Section{Generalized Cross-Validation} \label{sec:GCV}

A third parameter optimization technique is \emph{leave-one-out cross-validation} (CV).  Let $\widetilde{y}_i = \Ex[f(\vx_i ) | \vf_{-i} = \vy_{-i}]$, where the subscript $-i$ denotes the vector excluding the $i^{\text{th}}$ component.  This is the conditional expectation of $f(\vx_i )$ given all data but the function value at $\vx_i$.  The cross-validation criterion, which is to be minimized, is sum of squares of the difference between these conditional expectations and the observed values:
\begin{equation} \label{FJH:eq:CVA}
\textup{CV} = \sum_{i=1}^n (y_i - \widetilde{y}_i)^2.
\end{equation}

Let $\mA = \mC^{-1}$, let $\vzeta = \mA (\vy - m \vone)$, and partition $\mC$, $\mA$, and $\vzeta$ as
\begin{gather*}
\mC = \begin{pmatrix} c_{ii}  & \vC_{-i,i}^T \\  \vC_{-i,i} & \mC_{-i,-i}\end{pmatrix}, \qquad
\mA = \begin{pmatrix} a_{ii}  & \vA_{-i,i}^T \\  \vA_{-i,i} & \mA_{-i,-i}\end{pmatrix}, \qquad \vzeta = \begin{pmatrix} \zeta_i   \\  \vzeta_{-i} \end{pmatrix},
\end{gather*}
where the subscript $i$ denotes the $i^{\text{th}}$ row or column, and the subscript $-i$ denotes all rows or columns except the $i^{\text{th}}$. Following this notation, Lemma \ref{thrm:condDist} implies that 
\begin{align*}
\widetilde{y}_i & = m + \vC^T_{-i,i} \mC_{-i,-i}^{-1} (\vy_{-i} -m \vone)  \\
\zeta_i  & = a_{ii}(y_i - m) + \vA_{-i,i}^T(\vy_{-i} - m \vone) \\
& = a_{ii}[(y_i - m) - \vC^T_{-i,i} \mC_{-i,-i}^{-1} (\vy_{-i} -m \vone)] \\
& = a_{ii}(y_i - \widetilde{y}_i).
\end{align*}
Thus, \eqref{FJH:eq:CVA} may be re-written as 
\begin{equation*} 
\textup{CV} = \sum_{i=1}^n \left(\frac{\zeta_i}{a_{ii}} \right)^2, \quad 
\text{where} \quad \vzeta = \mC^{-1}(\vy - m \vone).
\end{equation*}
The \emph{generalized cross-validation} criterion (GCV) replaces the $i^{\text{th}}$ diagonal element of $\mA$ in the denominator by the average diagonal element of $\mA$ \cite{CraWah79a,GolHeaWah79a,Wah90}:
\begin{align*} 
\textup{GCV} &
= \frac{\sum_{i=1}^n\zeta_i^2}{\left(\frac 1n \sum_{i=1}^n a_{ii} \right)^2} 
= \frac{(\vy - m\vone)^T \mC^{-2} (\vy - m \vone)}{\left(\frac 1n \trace(\mC^{-1}) \right)^2}.
\end{align*}

The loss function $\textup{GCV}$ depends on $m$ and $\vtheta$, but not on $s$.  Minimizing the GCV  yields
\begin{equation*}
m_{\textup{GCV}} = \frac{\vone^T \mC^{-2} \vy}{\vone^T \mC^{-2} \vone},\\ 
\end{equation*}
\begin{align*}
\vtheta_{\textup{GCV}} = \argmin_\vtheta \biggl\{\log \left(  \vy^T \left[\mC^{-2} - \frac{\mC^{-2} \vone \vone^T \mC^{-2}}{\vone^T \mC^{-2} \vone}  \right] \vy \right)  
- 2 \log \left ( \trace(\mC^{-1}) \right ) \biggr\} .
\end{align*}
Plugging this value of $m$ into \eqref{eqn:BayesCub} yields
\begin{equation*}
\widehat{\mu}_{\textup{GCV}}
= \left(\frac{(1 - \vone^T  \mC^{-1}\vc) \mC^{-1} \vone}{\vone^T \mC^{-2} \vone} + \vc \right)^T \mC^{-1} \vy.
\end{equation*}

An estimate for $s$ may be obtained by noting that by Lemma \ref{thrm:condDist},
\begin{align*}
\var[f(\vx_i ) | \vf_{-i} = \vy_{-i}] = s^2 a_{ii}^{-1}.
\end{align*}
Thus, we may estimate $s_\GCV$ using an argument similar to that used in deriving the GCV and then substituting $m_{\textup{GCV}}$ for $m$:
\begin{align*}
s^2 &= \var[f(\vx_i ) | \vf_{-i} = \vy_{-i}] a_{ii} \\ 
& \approx \frac 1n \sum_{i=1}^n (y_i - \widetilde{y}_i)^2a_{ii}  = \frac 1n \sum_{i=1}^n \frac{\zeta_i^2}{a_{ii}} \\ 
& \approx \frac{ \frac 1n \sum_{i=1}^n \zeta_i^2}{\frac 1n \sum_{i=1}^n a_{ii} } 
 = \frac{(\vy - m\vone)^T \mC^{-2} (\vy - m \vone)}{ \trace(\mC^{-1}) } \; =: \; s^2_{\textup{GCV}}. \\
\end{align*}

The confidence interval based on generalized cross-validation corresponds to \eqref{eqn_prob_confidence_interval} with the GCV estimates for $m$, $s$, and $\vtheta$:
\begin{align}
\label{GCVerr}
\err_{\textup{GCV}} = 2.58 s_{\textup{GCV}}  \sqrt{c_0 - \vc^T\mC^{-1}\vc}, \\
\label{GCVCI}
\mathbb{P}_f \left[
|\mu-\hmu_{\textup{GCV}}| \leq \err_{\textup{GCV}} 
\right] = 99\%.
\end{align}

The methods developed for hyperparameter estimation from the previous sections are summarized as a theorem below:
\begin{theorem} \label{thm:param} There are at least three approaches to estimating or integrating out the hyperparameters defining the Gaussian process from which the integrand is drawn: empirical Bayes, full Bayes, and generalized cross-validation.  
Under these three approaches, we have the following:
	\begin{align}
	\label{eqn_m_MLE}
	m_\MLE &= \frac{\vone^T \mCthetaInv \vy }{ \vone^T \mCthetaInv \vone}, \qquad
	m_{\textup{GCV}} = \frac{\vone^T \mC_\vtheta^{-2} \vy}{\vone^T \mC_\vtheta^{-2} \vone}, \\
	\label{eqn_s2_MLE}
	s^2_{\MLE} 
	&= 
	\frac{1}{n}
	\vy^T 
	\left[ \mCthetaInv - 
	\frac{ \mCthetaInv \vone \vone^T \mCthetaInv }{\vone^T\mCthetaInv \vone}
	\right] \vy, \\
	\nonumber
	\hsigma_{\full}^2 
	& = \frac{1}{n-1}
	\vy^T\left[ \mCthetaInv 
	- \frac{ \mCthetaInv \vone\vone^T \mCthetaInv}{\vone^T \mCthetaInv \vone}  \right]\vy
	\\ 
	\label{eqn:sigma2_full}
	& \times  \left[\frac{(1 - \vc^T \mCthetaInv \vone)^2}{\vone^T \mCthetaInv \vone} + (c_{0}  -\vc ^T \mCthetaInv \vc) \right], \\
	\nonumber
	s^2_{\textup{GCV}} & = \vy^T \left[\mC_\vtheta^{-2} - \frac{\mC_\vtheta^{-2} \vone \vone^T \mC_\vtheta^{-2}}{\vone^T \mC_\vtheta^{-2} \vone}  \right] \vy  \left[ \trace(\mCthetaInv) \right]^{-1}, \\
	\label{eqn:thetaMLE}
	\vtheta_\MLE
	&= \argmin_{\vtheta} \biggl \{
	\log\left(\vy^T 
	\left[ \mCthetaInv - 
	\frac{ \mCthetaInv \vone \vone^T \mCthetaInv }{\vone^T\mCthetaInv \vone}
	\right] \vy 
	\right)  
	+  \frac{1}{n} \log(\det(\mC_\vtheta))
	\biggr \}, \\
	\label{vthetaGCV}
	\vtheta_{\GCV} &= \argmin_\vtheta \biggl\{\log \left(  \vy^T \left[\mC^{-2}_\vtheta - \frac{\mC^{-2}_\vtheta \vone \vone^T \mC^{-2}_\vtheta}{\vone^T \mC^{-2}_\vtheta \vone}  \right] \vy \right)  
	- \log \left ( \trace(\mC^{-2}_\vtheta) \right ) \biggr\}, \\
	\label{eqn:cubMLE}
	\hmu_\MLE  &= \hmu_\full =
	\left(
	\frac{ (1 - \vone^T  \mCthetaInv\vc )  \vone }{ \vone^T \mCthetaInv \vone}   +  \vc 
	\right)^T  \mCthetaInv \vy, \\
	\label{eqn:muCV}
	\hmu_{\GCV}
	& = \left(\frac{(1 - \vone^T  \mCthetaInv\vc) \mCthetaInv \vone}{\vone^T \mC_\vtheta^{-2} \vone} + \vc \right)^T \mCthetaInv \vy.
	\intertext{The credible intervals widths, $\err_{\CI}$, are given by}
	\label{eqn:err_MLEGCV}
	\err_{\mathsf{x}} & = 2.58 s_{\mathsf{x}} \sqrt{c_{0} - \vc^T\mCthetaInv\vc }, \qquad \mathsf{x} \in \{\MLE, \GCV\},  \\ 
	\err_{\textup{full}} 
	& = t_{n-1,0.995} \hsigma_{\textup{full}} > \err_\MLE. \label{FJH:eq:errFull}
	\end{align}
	The resulting credible intervals are then
	\begin{align}
	\label{eqn_prob_CI}
	\mathbb{P}_f \left[
	|\mu-\hmu_{\mathsf{x}}| \leq \err_{\mathsf{x}} \right]  = 99\%, \qquad \mathsf{x} \in \{\MLE, \full, \GCV\}.
	\end{align}
	Here $t_{n-1,0.995}$ denotes the $99.5$ percentile of a standard Student's $t$-distribution with $n-1$ degrees of freedom.  In the formulas above, $\vtheta$ is assumed to take on the values $\vtheta_{\MLE}$ or $\vtheta_{\GCV}$ as appropriate.
\end{theorem}

In the theorem above, note that if the original covariance kernel, $C$, is replaced by $b C$ for some positive constant $b$, the cubature, $\hmu$, the estimates of $\vtheta$, and the credible interval half-widths, $\err_{\mathsf{x}}$ for $\mathsf{x} \in \{\MLE, \full, \GCV\}$, all remain unchanged.  The estimates of $s^2$ are multiplied by $b^{-1}$, as would be expected. 



\Section{Cone of Functions and the Credible interval} 
\label{sec:cone_of_functions}


In this research we assume that the integrand belongs to a cone of well-behaved functions, $\mathscr{C}$, to make the computations bounded in terms of function data. The concept of cone in general for cubature error analysis can be stated using the error bound definition. 
Suppose that 
\begin{align}
\label{eqn:cone-homogenity}
\abs{\mu(f) - \hmu_n(f)} \le \err_{\CI} (f(\vx_1), \cdots, f(\vx_n))
\end{align}
for some $f$, which it is 99\% of the time under our hypothesis. Also note that our $\err_{\CI}$ \eqref{eqn:err_MLEGCV} \eqref{FJH:eq:errFull} are positively homogeneous functions, meaning, 
\begin{align*}
\err_{\CI} (a y_1, \cdots, a y_n) = \abs{a} \err_{\CI} ( y_1, \cdots,  y_n).
\end{align*}
One can verify the homogeneity of \eqref{eqn:err_MLEGCV} and \eqref{FJH:eq:errFull} easily.
Thus if $f$ satisfies \eqref{eqn:cone-homogenity}, then
\begin{align*}
\abs{\mu(af) - \hmu_n(af)} &= \abs{a} \abs{\mu(f) - \hmu_n(f)} \\
& \le \abs{a} \err_{\CI} (f(\vx_1), \cdots, f(\vx_n) ) \\
& = \err_{\CI} (a f(\vx_1), \cdots, a f(\vx_n))
\end{align*}
for all real $a$. Thus the set of all $f$ satisfying \eqref{eqn:cone-homogenity} is a \emph{cone}, $\mathscr{C}$. Cones of functions satisfy the property that if $f \in \mathscr{C}$ then $af \in \mathscr{C}$.

In the context of Bayesian cubature, one can explain the cone concept beginning with the definition of credible interval \eqref{eqn_prob_confidence_interval}. 
Let $f \sim \mathcal{GP}$, be an instance of a Gaussian stochastic process:
\begin{align*}
\mathbb{P}_f \left[
|\mu(f)-\hmu_n(f)| \leq \err_{\CI}(f) \right] \ge 99\%.
\end{align*}
This can be interpreted as $|\mu(f)-\hmu_n(f)| \leq \err_{\CI}(f)$ with 99\% confidence. If $f$ is in the 99\% middle of the sample space with $f(\vx_i) = y_i$ then $af$ is also in the middle 99\% of the sample space with $a f(\vx_i) = a y_i$.

We demonstrate the credible interval using the following example. For this purpose, choose a smooth and periodic integrand $f_{\TRUE}(\vx) = \exp(\sum_{\ell=1}^{d} \cos(2\pi x_\ell))$ and another integrand $f_{\PEAKY}(\vx) = f_{\TRUE} + a_{\PEAKY} f_{\NOISE}$ where $a_{\PEAKY} \in \reals$. 
Here $f_{\NOISE}(\vx) = (1 - \exp(2\pi \sqrt{-1} \vx^T \vzeta))$, $\vzeta \in \reals^d$ is some $d$-dimensional vector belonging to the dual space of the lattice nodes for some $\{ \vx_i\}_{i=1}^n$. 
The $\vzeta$ in the dual space of lattice nodes implies that $f_{\NOISE}(\vx_i)=0$ at the sampling nodes $\{ \vx_i\}_{i=1}^n$.
The $f_{\NICE}$ is obtained by kernel interpolation of the $n$ samples of $f_\TRUE$ at $\{ \vx_i\}_{i=1}^n$. We chose the Mat\'ern kernel \eqref{eqn:matern_kernel} for the interpolation.
Please note that $f_{\PEAKY}(\vx_i) = f_{\NICE}(\vx_i) = f_{\TRUE}(\vx_i) $ for $i=1, \cdots, n$.

\begin{figure}[ht]
	\centering
	\includegraphics[width=0.8\linewidth]{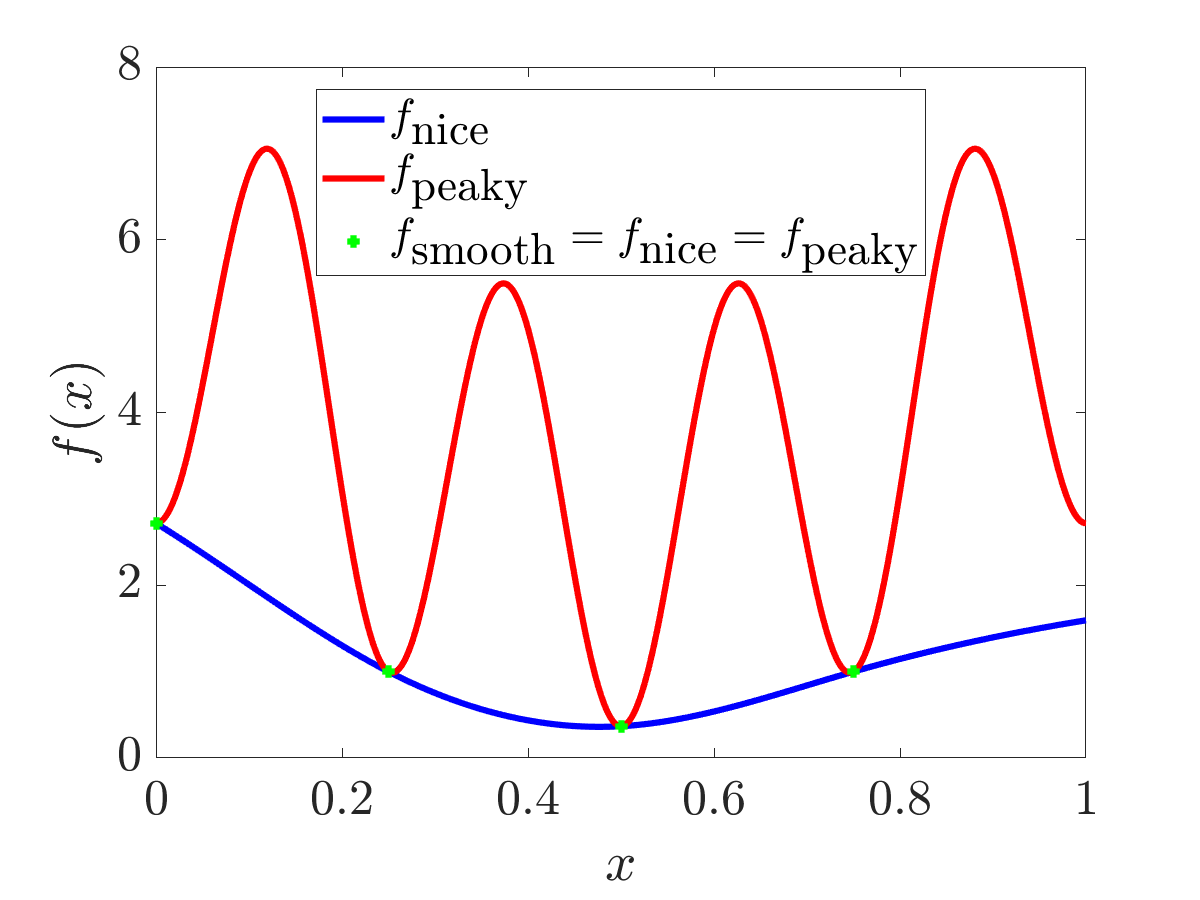}
	\caption{Example integrands 1) $f_{\NICE}$, a smooth function, 2) $f_{\PEAKY}$, a peaky function. The function values $f_{\PEAKY}(\vx_i) = f_{\NICE}(\vx_i) = f_{\TRUE}(\vx_i) $ for $i=1, \cdots, n$. This plot can be conditionally reproduced using \code{DemoCone.m}}
	\label{fig:cone_bayes_functions}
\end{figure}
\begin{figure}[ht]
	\centering
	\includegraphics[width=0.8\linewidth]{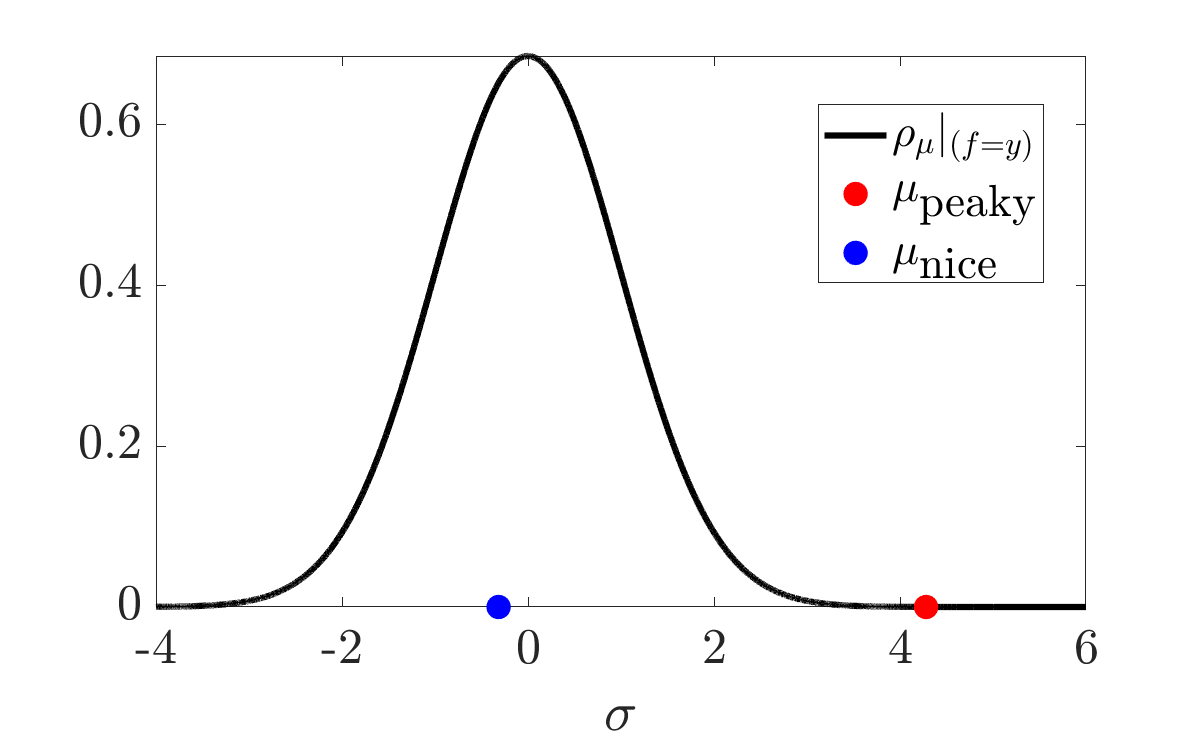}
	\caption{Probability distributions showing the relative integral position of a smooth and a peaky function. $f_{\NICE}$ lies within the center 99\% of the confidence interval, and $f_{\PEAKY}$ lies on the outside of 99\% of  the confidence interval. This plot can be conditionally reproduced using \code{DemoCone.m}}
	\label{fig:cone_bayes_posterior}
\end{figure}

In \figref{fig:cone_bayes_functions}, the sampled function values
are shown as dots. One can imagine these samples 
were obtained from $f_{\NICE}$, a moderately smoother function or from $f_{\PEAKY}$, a highly oscillating function. In this example, we used $a_{\PEAKY} = 2$.

When using $n=16$ rank-1 lattice points, and $r=1$ shift-invariant kernel, we get the posterior distribution of $\mu$ as shown in \figref{fig:cone_bayes_posterior}. The true integral value is shown as $\mu_{\TRUE}$ which is at the center of the plot. The integral of the peaky function $f_{\PEAKY}$ lies outside of the 99\% of the credible interval given by \eqref{eqn_prob_CI_MLE}, whereas the $\mu_{\NICE}$ falls within.

Our Bayesian cubature algorithms compute the approximate integral using only the samples of the integrand. 
Estimated integral value of our algorithm closely matches the integral of a smooth function that falls within the middle of the confidence interval. If the true integrand were to resemble the smooth approximate function then the estimated integral will be accurate.

\Section{The Automatic Bayesian Cubature Algorithm} 
\label{sec:bayes_cubature_algo}

The previous section presents three credible intervals, \eqref{eqn_prob_CI_MLE}, \eqref{eqn_prob_CI_full}, and \eqref{GCVCI}, for the $\mu$, the desired integral.  Each credible interval is based on different assumptions about the hyperparameters $m$, $s$, and $\vtheta$.  We stress that one must estimate these hyperparameters or assume a prior distribution on them because the credible intervals are used as stopping criteria for our cubature rule.  Since a credible interval makes a statement about a typical function---not an outlier---one must try to ensure that the integrand is a typical draw from the assumed Gaussian process.

Our  Bayesian cubature algorithm increases the sample size until the width of the credible interval is small enough.  This is accomplished through successively doubling the sample size.  The steps are detailed in Algorithm~\ref{algorithm1}.

We recognize that multiple applications of our credible intervals in one run of the algorithm is not strictly justified.  However, if our integrand comes from the middle of the sample space and not the extremes, we expect our automatic Bayesian cubature to approximate the integral within the desired error tolerance with high probability.  The example in the next section and the examples in Chapter~\ref{sec:NumExp} support that expectation. We also believe that an important factor contributing to the occasional failure of our algorithm is unreasonable parameterizations of the stochastic process from which the integrand is hypothesized to be drawn.  Overcoming this latter challenge is a topic for future research.

\algnewcommand{\IIf}[1]{\State\algorithmicif\ #1\ \algorithmicthen\ }
\algnewcommand{\IElse}{\unskip\ \algorithmicelse\ }
\algnewcommand{\EndIIf}{\unskip\ \algorithmicend\ \algorithmicif}

\begin{algorithm}
\caption{Automatic Bayesian Cubature}\label{algorithm1}
  \begin{algorithmic}[1]
  	\Require a generator for the sequence
  	$\vx_1, \vx_2, \ldots$; 
  	a black-box function, $f$; 
  	an absolute error tolerance,
  	$\varepsilon>0$; the positive initial sample size, $n_0$;
  	the maximum sample size $n_{\textup{max}}$
  	
      \State $n \gets n_0, \; n' \gets 0, \; \err \gets \infty$
      
      \While{$\err > \varepsilon$ and $n \le n_{\textup{max}}$}
      
        \State\label{LoopStart}Generate $\{ \vx_i\}_{i=n' + 1}^{n}$ and sample $\{f(\vx_i)\}_{i=n'+1}^{n}$
        \State Compute $\vtheta$ by \eqref{eqn:thetaMLE} or \eqref{vthetaGCV}
        \State Compute $\err$  according to \eqref{eqn:err_MLEGCV}, \eqref{FJH:eq:errFull}, or \eqref{GCVerr}
        
       	\State	$n' \gets n, \; n \gets 2n'$
        
        \EndWhile
        
        \State Sample size to compute $\hmu$, $n \gets n'$
        \State Compute $\hmu$, the approximate integral,   according to \eqref{eqn:cubMLE} or \eqref{eqn:muCV}
      \State \Return $\hmu, \; n$  and $\err$
  \end{algorithmic}
\end{algorithm}

As described above, the computational cost of Algorithm~\ref{algorithm1} is the sum of the following:
\begin{itemize}
	\item $\Order\bigl(n\$(f) \bigr)$ for the integrand data, where $\$(f)$ is the computational cost of a single $f(\vx)$; $\$(f)$ may be large if it is the result of an expensive simulation; $\$(f)$ is typically proportional to $d$;
	
	\item $\Order\bigl(N_{\opt} n^2 \$(C_\vtheta) \bigr)$ for the evaluation of the Gram matrix $\mC_{\vtheta}$, $N_\opt$ is the number of optimization steps required, and  $\$(C_\vtheta)$ is the computational cost of a single $C_\vtheta(\vt,\vx)$; $\$(C_\vtheta)$ is typically proportional to $d$; and
	
	\item $\Order\bigl(N_{\opt} n^3 \bigr)$ for the matrix inversions and determinant calculations; this cost is independent of $d$.
	
\end{itemize}

As we see in the example in the next section, the cost increases quickly as the $n$ required to meet the error tolerance increases.  This motivates the fast Bayesian cubature algorithm presented in Chapter~\ref{sec:fast_BC}.

\Section{Example with the Mat\'ern Kernel} \label{MVN_example}

To demonstrate automatic Bayesian cubature consider a Mat\'ern covariance kernel:
\begin{align}
\label{eqn:matern_kernel}
C_{\theta}(\vx, \vt) = \prod_{\ell=1}^d \exp(-\theta|\vx_\ell-\vt_\ell|)(1+\theta |\vx_\ell-\vt_\ell|).
\end{align}
Also, consider the integration problem of evaluating  \emph{multivariate Gaussian probabilities}:
\begin{equation}
\label{eqn:GaussDef}
\mu = \int_{(\va,\vb)} \frac{\exp\bigl(- \frac 12 \vt^T \mSigma^{-1} \vt \bigr)}{\sqrt{(2 \pi)^d \det(\mSigma)}} \, \dvt,
\end{equation}
where $(\va,\vb)$ is a finite, semi-infinite or infinite box in $\reals^d$.  This integral does not have an analytic expression for general $\mSigma$, so cubatures are required.  

Genz \cite{Gen93} introduced a variable transformation to transform \eqref{eqn:GaussDef} into an integral on the unit cube.  Not only does this variable transformation accommodate domains that are (semi-)infinite, it also tends to smooth out the integrand better, which expedites the cubature.  Let $\mSigma= \mL \mL^T$ be the Cholesky decomposition where $\mL = (l_{jk})_{j,k=1}^d$ is a lower triangular matrix.  Iteratively define
\begin{align*}
\alpha_1 = \Phi(a_1), 
&\qquad
\beta_1 = \Phi(b_1),
\\
\alpha_\ell(x_1,...,x_{\ell-1}) &= 
\Phi
\left(
\frac{1}{l_{\ell \ell}} 
\left(
a_\ell - \sum_{k=1}^{\ell-1} l_{\ell k} \Phi^{-1}(\alpha_k + x_k(\beta_k-\alpha_k))
\right)
\right), \quad \ell=2,...,d,
\\
\beta_\ell(x_1,...,x_{\ell-1}) &= 
\Phi
\left(
\frac{1}{l_{\ell\ell}} 
\left(
b_\ell - \sum_{k=1}^{\ell-1} l_{\ell k} \Phi^{-1}(\alpha_k + x_k(\beta_k-\alpha_k))
\right)
\right), \quad \ell=2,...,d,
\end{align*}
\begin{align}
\label{eqn:fGenzdef}
f_{\text{Genz}}(\vx) = \prod_{\ell=1}^d [\beta_\ell(\vx) - \alpha_\ell(\vx)].
\end{align}
where $\Phi$ is the cumulative standard normal distribution function.
Then, $$\mu = \int_{[0,1]^{d-1}} f_{\text{Genz}}(\vx) \, \dvx.$$
This approach transforms a $d'$ dimensional integral into a $d=d'-1$ dimensional integral.

\begin{figure}
	\centering
		\includegraphics[width=0.75\linewidth]{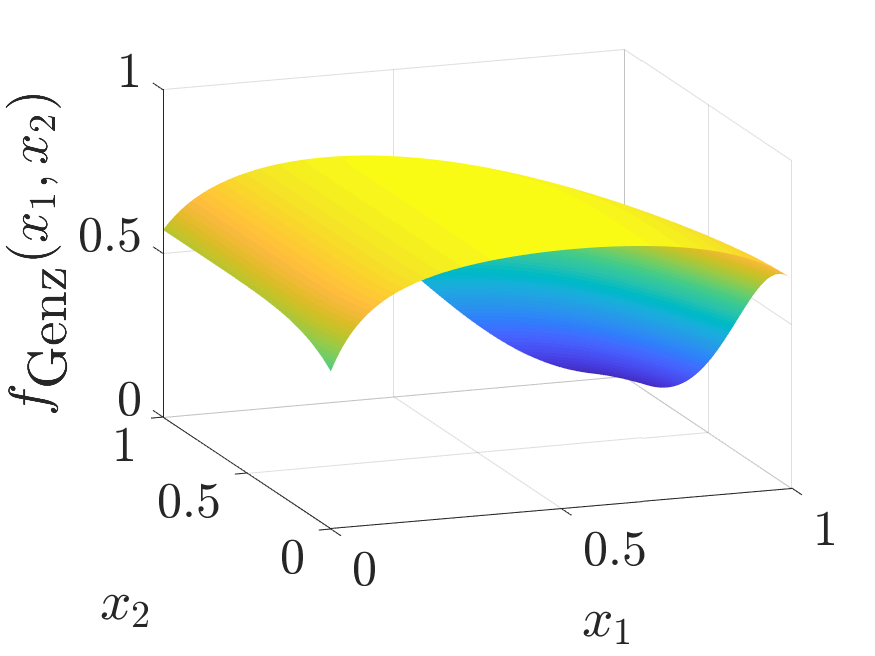}
	\caption{The $d=3$ multivariate normal probability transformed to an integral of $f_{\text{Genz}}$ with  $d=2$. This plot can be reproduced using \code{IntegrandPlots.m} in GAIL.}
	\label{fig:MVN_Genz}
\end{figure}

We use the following parameter values in the simulation: 
\begin{equation*}
d = 3, \quad \va = \begin{pmatrix}[1]
-6 \\ -2 \\ -2
\end{pmatrix}, \quad 
\vb = \begin{pmatrix}[1]
5 \\ 2 \\ 1
\end{pmatrix} , \quad 
\mL = \begin{pmatrix}[1]
4 & 1 & 1 \\ 0 & 1 & 0.5 \\ 0 & 0 & 0.25
\end{pmatrix}.
\end{equation*}

\begin{figure}[ht]
	\centering
	\includegraphics[width=0.9\linewidth]{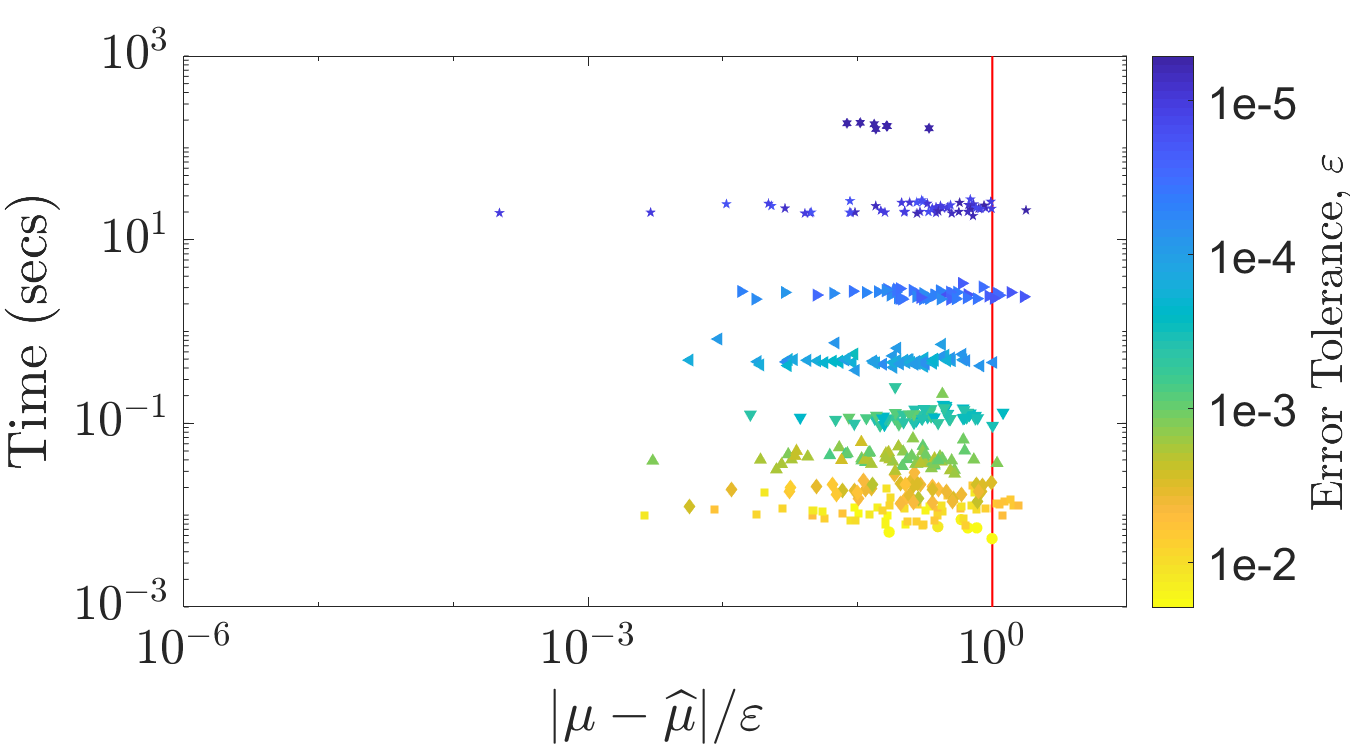}
	\caption{Multivariate Gaussian probability: Guaranteed integration using Mat\'ern kernel in $d=2$ using empirical Bayes stopping criterion within error tolerance $\varepsilon$. This figure can be conditionally reproduced using \code{matern\_guaranteed\_plots.m} in GAIL.}
	\label{fig:MVN_Metern_d2b2}
\end{figure}
\begin{figure}[ht]
	\centering
	\includegraphics[width=0.9\linewidth]{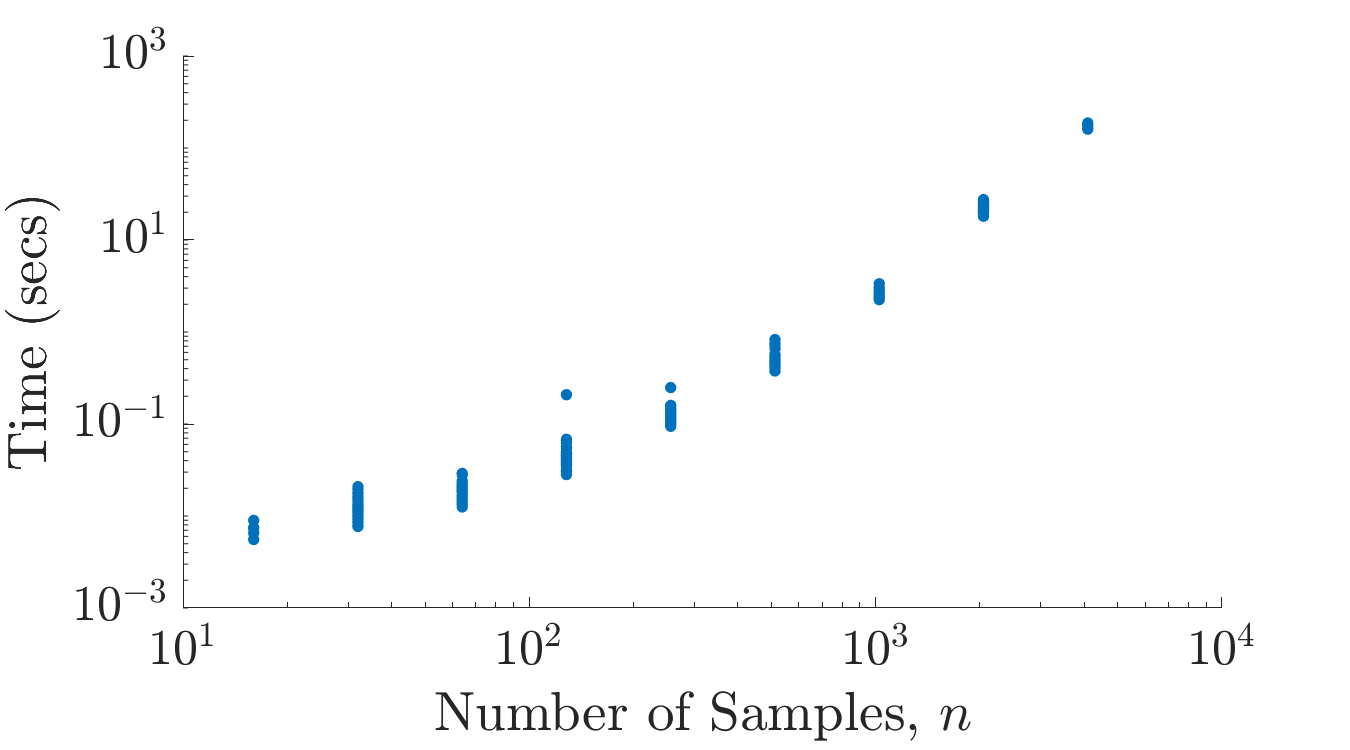}
	\caption{Multivariate Gaussian probability estimated using Mat\'ern kernel in $d=2$ using empirical Bayes stopping criterion. Computation time rapidly increases with increase of $n$. This figure can be conditionally reproduced using \code{matern\_guaranteed\_plots.m} in GAIL.}
	\label{fig:MVN_Metern_d2b2_time_growth}
\end{figure}
The node sets are randomly scrambled Sobol' points \cite{DicEtal14a,DicPil10a}.  The results are for 400 randomly chosen $\varepsilon$ in the interval $[10^{-5}, 10^{-2}]$ as shown in \figref{fig:MVN_Metern_d2b2}. In each run, the nodes are randomly scrambled.  We  observe the algorithm meets the error criterion 95\% of the time even though we used 99\% credible intervals.
One possible explanation is that the matrix inversions in the algorithm are ill-conditioned leading to numerical inaccuracies.  Another possible explanation is that this Mat\'ern covariance kernel is not a good match for the integrand.

On our test computer, it took more than an hour to compute $\hmu_n$ with $n=2^{14}$. 
As shown in \figref{fig:MVN_Metern_d2b2_time_growth}, the computation time increases rapidly with $n$. 
The empirical Bayes estimation of $\vtheta$, which requires repeated evaluation of the objective function, is the most time consuming of all. This is due to fact that the objective function needs to be computed multiple times in every iteration to find its minimum. It takes tens of seconds to compute $\hmu_n$ with $\varepsilon = 10^{-5}$.   In contrast, this example in Chapter~\ref{sec:NumExp} take less than a hundredth of a second to compute $\hmu_n$ with the same $\varepsilon$ using our new algorithm. Not only is the Bayesian cubature with the Mat\'ern kernel slow, but also $\mC_\vtheta$ becomes highly ill-conditioned as $n$ increases.
So, Algorithm~\ref{algorithm1} in its current form is impractical when $n$ must be large.

\Chapter{Fast Automatic Bayesian Cubature}\label{sec:fast_BC}


The generic automatic Bayesian cubature algorithm described in the previous section requires $\Order\bigl(n \$(f) +  N_{\opt} [n^2 \$(C_\vtheta) + n^3]\bigr)$ operations to compute the cubature. Now we explain how to speed up the calculations. A key is to choose covariance kernels that match the nodes, $\{\vx_i\}_{i=1}^n$, so that the vector-matrix operations required by Bayesian cubature can be accomplished using \emph{fast Bayesian transforms} at a computational cost of $\Order\bigl(n \$(f) + N_{\opt} [n \$(C_\vtheta)  + n \log(n)] \bigr)$.
We develop the concept of fast Bayesian transform and show how matching kernels and nodes with three key assumptions are used. 

\Section{Fast Bayesian Transform Kernel}

We make some assumptions about the relationship between the covariance kernel and the nodes. 
In Chapter~\ref{sec:shift_invariant_kernel} these assumptions are shown to hold  for rank-1 lattices and shift-invariant kernels and again in Chapter~\ref{sec:sobol_walsh} to hold for Sobol' nodes and Walsh kernels.  Although the integrands and covariance kernels are real, it is convenient to allow related vectors and matrices to be complex.  A relevant example is the fast Fourier transform (FFT) of a real-valued vector, which is a complex-valued vector.

We introduce some further notation
\begin{align}
\nonumber
\mC = \mCtheta &= \Big(C_\vtheta(\vx_i,\vx_j)\Big)_{i,j=1}^n  = (\vC_1,...,\vC_n) 
\\
\label{eqn:ftk_factor}
&= \frac 1n \mV \mLambda \mV^H , 
\quad \quad \mV^H = n \mV^{-1}, \\
\nonumber
\mV &= (\vv_1,...,\vv_n)^T = (\vV_1,...,\vV_n) \\
\nonumber
\mC^p  &= \frac 1n \mV \mLambda^{p} \mV^H, \qquad \forall p \in \integers,
\end{align}
where $\mV^H$ is the Hermitian of $\mV$, $\vC_1,\cdots,\vC_n$ are columns of $\mC$,  $\vV_1,\cdots,\vV_n$ are columns of $\mV$, and $\vv_1,\cdots,\vv_n$ are rows of $\mV$. 
The columns of matrix $\mV$ are eigenvectors of $\mC$, and $\mLambda$ is a diagonal matrix of eigenvalues of $\mC$.
In this and later sections, we drop the $\vtheta$ dependence of various quantities for simplicity of notation.  The normalization of $\mV$ assumed in \eqref{eqn:ftk_factor} conveniently allows the first eigenvector, $\vV_1$, to be the vector of ones in \eqref{fastcompAssumpB} below.  
For any $n \times 1$ vector $\vb$, define the notation  $\widetilde{\vb} := \mV^H \vb$.

We make three assumptions that allow the fast computation:
\begin{subequations} \label{fastcompAssump}
	\begin{gather}
	\label{fastcompAssumpA}
	\mV \text{ may be identified analytically}, \\
	\label{fastcompAssumpB}
	\vv_1 = \vV_1 = \vone, \\
	\label{fastcompAssumpC}
	\text{ Computing $\mV^H \vb$ requires only $\Order(n \log n)$ operations } \forall \vb.
	\end{gather}
\end{subequations}
We call the transformation $\vb \mapsto \mV^H \vb$ a \emph{fast Bayesian transform} and $C_\vtheta$ a \emph{fast Bayesian transform kernel} for the matching nodes $\{\vx_i\}_{i=1}^\infty$.  

Under assumptions \eqref{fastcompAssump} the eigenvalues may be identified as the fast Bayesian transform of the first column of $\mC$:
\begin{align}
\nonumber
\vlambda 
& = \begin{pmatrix}[0.8]
\lambda_1 \\ \vdots \\ \lambda_n
\end{pmatrix} 
= \mLambda \vone = \mLambda \vv_1^* 
= \underbrace{\left( \frac 1n \mV^H  \mV \right) }_{\mathsf{I}} \mLambda \vv_1^* \\
&= \mV^H \left( \frac 1n \mV \mLambda \vv_1^* \right)
= \mV^H \vC_1 =  \widetilde{\vC}_1,
\label{eqn:fast_transform_to_eigvalues}
\end{align}
where $\mathsf{I}$ is the identity matrix and $\vv_1^*$ is the complex conjugate of the first row of $\mV$. 
Also note that the fast Bayesian transform of $\vone$ has a simple form
\begin{align*} 
\widetilde{\vone}
& = \mV^H \vone = \mV^H \vV_1 = \begin{pmatrix}[0.8] n \\ 0 \\ \vdots \\ 0 \end{pmatrix}.
\label{eqn:fast_transform_one}
\end{align*}

Many of the terms that arise in the calculations in  Algorithm~\ref{algorithm1} take the form $\va^T\mC^{p}\vb$ for real $\va$ and $\vb$ and integer $p$.  These can be calculated via the transforms $\widetilde{\va} = \mV^H \va$ and $\widetilde{\vb} = \mV^H \vb$ as 
\begin{equation*}
\va^T\mC^p\vb = \frac 1n \va^T \mV \mLambda^p \mV^H \vb
= \frac 1n \widetilde{\va}^H\mLambda^p \widetilde{\vb}
= \frac 1n \sum_{i=1}^n \lambda_i^p \widetilde{a}_i^* \widetilde{b}_i, 
\end{equation*}
Note that $\widetilde{\va}^*$ appears on the right side of this equation because $\va^T \mV = (\mV^H \va)^* = \widetilde{\va}^*$. In particular,
\begin{align*}
\vone^T\mC^{-p}\vone & = \frac{n}{\lambda_1^p},
&
\vone^T\mC^{-p}\vy &= \frac{\widetilde{y}_1}{\lambda_1^p},
\\
\vy^T\mC^{-p} \vy &= \frac 1n \sum_{i=1}^n \frac{\abs{\widetilde{y}_i}^2}{\lambda_i^p},
&
\vc^T\mCInv \vone &= \frac{\widetilde{c}_1}{\lambda_1},\\
\vc^T\mCInv \vy &= \frac 1n \sum_{i=1}^n \frac{\widetilde{c}_i^* \widetilde{y}_i}{\lambda_i}, & 
\vc^T\mCInv \vc &= \frac 1n \sum_{i=1}^n \frac{\abs{\widetilde{c}_i}^2}{\lambda_i},
\end{align*}
where $\widetilde{\vy} = \mV^H \vy$ and 
$\widetilde{\vc} = \mV^H \vc$.  For any real $\vb$, with $\widetilde{\vb} = \mV^H\vb$, it follows that $\widetilde{b}_1$ is real since the first row of $\mV^H$ is $\vone$.

The covariance kernel used in practice also may satisfy an additional assumption:
\begin{equation} \label{addAssump}
\int_{[0,1]^d} C(\vt,\vx) \, \D \vt = 1 \qquad \forall \vx \in [0,1]^d,
\end{equation}
which implies that $c_{0 \vtheta} = 1$ and $\vc_{\vtheta} = \vone$.  Under \eqref{addAssump}, the expressions above may be further simplified:
\begin{equation*}
\vc^T\mCInv \vone =
\vc^T\mCInv \vc = \frac{n}{\lambda_1}.
\end{equation*}
We use the fast Bayesian transform to speedup the computation of the hyperparameter $\vtheta$, the credible interval width $\err_{\CI}$, and the integral estimate $\hmu$ that we presented in Theorem~\ref{thm:param} as shown next.
\iftrue
The assumptions and results in this chapter lead to the following theorem.

\begin{theorem} \label{thm:fastparam}
	Under assumptions \eqref{fastcompAssump}, the parameters and credible interval half-widths in Theorem~\ref{thm:param} may be expressed in terms of the fast Bayesian transforms of the integrand data, the first column of the Gram matrix, $c_0$, and $\vc$ as follows:
	\begin{align}
	\nonumber
	m_\MLE &=  m_{\full} = m_{\GCV} =  \frac{\widetilde{y}_1}{n} = \frac 1n \sum_{i=1}^n y_i,
	\\
	\nonumber
	s^2_\MLE 
	& =
	\frac{1}{n^2} 
	\sum_{i=2}^n \frac{\abs{\widetilde{y}_i}^2}{\lambda_i}, \\
	\nonumber
	\widehat{\sigma}^2_{\textup{full}}
	& =
	\frac{1}{n(n-1)} \sum_{i=2}^n \frac{\abs{\widetilde{y}_i}^2}{\lambda_i}
	\left[\frac{\lambda_1}{n}{\left(1 - \frac{\widetilde{c}_1}{\lambda_1}\right)^2} + \left(c_0  - \frac 1n \sum_{i=1}^n \frac{\abs{\widetilde{c}_i}^2}{\lambda_i}\right) \right], \\
	\nonumber 
	s^2_{\textup{GCV}} & =  \frac 1{n} \sum_{i=2}^n \frac{\abs{\widetilde{y}_i}^2}{\lambda_i^2}  \left [ \sum_{i=1}^n \frac{1}{\lambda_i} \right]^{-1},
	\end{align}
	\begin{subequations}
		\label{fastTheta}
		\begin{align}
		\label{eqn_MLE_loss_func_optimized_2}
		\vtheta_\MLE
		&= 
		\argmin_{\vtheta}
		\left[
		\log\left(
		\sum_{i=2}^n \frac{\abs{\widetilde{y}_i}^2}{\lambda_{i}}
		\right) 
		+ \frac{1}{n}\sum_{i=1}^n \log(\lambda_{i})
		\right],\\
		\label{thetaGCV} 
		\vtheta_{\GCV} 
		&= \argmin_\vtheta \left[ \log \left ( \sum_{i=2}^n \frac{\abs{\widetilde{y}_i}^2}{\lambda_{i}^2} 
		\right)  
		-2\log\left( \sum_{i=1}^n \frac{1}{\lambda_{i}} \right)
		\right], 
		\end{align}
	\end{subequations}
	\begin{align}
	\nonumber
	\hmu_\MLE  &= \hmu_{\full} = \hmu_{\GCV} =
	\frac{\widetilde{y}_1}{n} +
	\frac 1n \sum_{i=2}^n \frac{ \widetilde{c}_i^* \widetilde{y}_i}{\lambda_i}, \\
	\nonumber
	\err_\MLE  &
	=
	\frac{2.58}{n}\sqrt{
		\sum_{i=2}^{n} \frac{\abs{\widetilde{y}_i}^2}{\lambda_i}  
		\,
		\left( c_0 - \frac 1n \sum_{i=1}^n \frac{\abs{\widetilde{c}_i}^2}{\lambda_i} \right) 
	}, \\
	\nonumber
	\err_{\full} & = t_{n-1,0.995} \hsigma_{\textup{full}}, \\
	\nonumber
	\err_{\textup{GCV}} & =
	\frac{2.58}{n}\left\{\sum_{i=2}^n \frac{\abs{\widetilde{y}_i}^2}{\lambda_i^2}  \left [ \frac 1n \sum_{i=1}^n \frac{1}{\lambda_i} \right]^{-1} 
	\left( c_0 - \frac 1n \sum_{i=1}^n \frac{\abs{\widetilde{c}_i}^2}{\lambda_i} \right) 
	\right\}^{1/2}.
	\end{align}
	Under the further assumption \eqref{addAssump}, it follows that 
	\begin{equation}
	\label{muhatGCV-FB-MLE-Simple}
	\hmu_\MLE  = \hmu_{\full} = \hmu_{\GCV} =
	\frac{\widetilde{y}_1}{n} = \frac 1n \sum_{i=1}^n y_i,
	\end{equation}
	and so $\hmu$ is simply the sample mean.  Also, under assumption \eqref{addAssump}, the credible interval half-widths simplify to
	\begin{subequations}
		\label{errSimple}
		\begin{align}
		\label{eq:errMLEAllAsump}
		\err_\MLE  &
		=
		\frac{2.58}{n}\sqrt{
			\sum_{i=2}^{n} \frac{\abs{\widetilde{y}_i}^2}{\lambda_i}  
			\,
			\left( 1 -  \frac{n}{\lambda_1} \right) 
		}, \\
		\label{FJH:eq:errFullSimple}
		\err_{\textup{full}}
		&=
		t_{n-1,0.995}
		\sqrt{\frac{1}{n(n-1)} \sum_{i=2}^n \frac{\abs{\widetilde{y}_i}^2}{\lambda_i}  \left(\frac{\lambda_1}{n}  - 1  \right)}, \\
		\label{errGCVSimple}
		\err_{\textup{GCV}} & =
		\frac{2.58}{n}\left\{\sum_{i=2}^n \frac{\abs{\widetilde{y}_i}^2}{\lambda_i^2}  \left [ \frac 1n \sum_{i=1}^n \frac{1}{\lambda_i} \right]^{-1} 
		\left( 1 -  \frac{n}{\lambda_1} \right)  
		\right\}^{1/2}. 
		\end{align}
	\end{subequations}
	In the formulas for the credible interval half-widths and $\vlambda$ depends on $\vtheta$, and  $\vtheta$ is assumed to take on the values $\vtheta_{\MLE}$ or $\vtheta_{\GCV}$ as appropriate.
\end{theorem}
The remaining part of the chapter proves this theorem. We apply the fast Bayesian transform to speedup empirical Bayes, full Bayes and Generalized cross validation stopping criteria.

\fi

\Section{Empirical Bayes}

Under assumptions \eqref{fastcompAssump}, the empirical Bayes parameters in \eqref{eqn_m_MLE}, \eqref{eqn_s2_MLE}, \eqref{eqn:thetaMLE} \eqref{eqn:cubMLE}, and \eqref{eqn:err_MLEGCV} can be expressed in terms of the fast Bayesian transforms of the function data, the first column of the Gram matrix, and $\vc$ as follows:
\begin{align*}
\nonumber
m_\MLE &=  \frac{\widetilde{y}_1}{n} = \frac 1n \sum_{i=1}^n y_i,
\\
\nonumber
s^2_\MLE 
& =
\frac{1}{n^2} 
\sum_{i=2}^n \frac{\abs{\widetilde{y}_i}^2}{\lambda_i}, \\
\vtheta_\MLE
&= 
\argmin_{\vtheta}
\left[
\log\left(
\sum_{i=2}^n \frac{\abs{\widetilde{y}_i}^2}{\lambda_i}
\right)   + 
\frac{1}{n}\sum_{i=1}^n \log(\lambda_i)
\right],
\\
\nonumber
\hmu_\MLE  &= 
\frac{\widetilde{y}_1}{n} +
\frac 1n \sum_{i=2}^n \frac{ \widetilde{c}_i^* \widetilde{y}_i}{\lambda_i}, \\
\nonumber
\err_\MLE  &
=
\frac{2.58}{n}\sqrt{
	\sum_{i=2}^{n} \frac{\abs{\widetilde{y}_i}^2}{\lambda_i}  
	\,
	\left( c_0 - \frac 1n \sum_{i=1}^n \frac{\abs{\widetilde{c}_i}^2}{\lambda_i} \right) 
},
\end{align*}
The quantities on the right hand sides can be obtained in $\Order(n \log n)$ operations by fast Bayesian transforms. 

Under the further assumption \eqref{addAssump} it follows that 
\begin{align*}
\nonumber
\hmu_\MLE  &= 
\frac{\widetilde{y}_1}{n} = \frac 1n \sum_{i=1}^n y_i,\\
\err_\MLE  &
=
\frac{2.58}{n}\sqrt{
	\sum_{i=2}^{n} \frac{\abs{\widetilde{y}_i}^2}{\lambda_i}  
	\,
	\left( 1 -  \frac{n}{\lambda_1} \right) 
}.
\end{align*}
Thus, in this  case $\hmu$ is simply the sample mean.

\Subsection{Gradient of the objective function using fast Bayesian transform} \label{deriv_obj_func_MLE} 
We refer back to Section~\ref{grad_descent_MLE}, where we discuss about using gradient descent for hyperparameter search but the computational cost is of $\Order(N_{\opt} n^3)$. Here we develop a techniques to speed up the computation.
If $\mV$ does not depend on $\vtheta$ then one can fast compute the derivative of Gram matrix $\mC$. Starting from the definition \eqref{eqn:ftk_factor} and taking derivative w.r.t. $\theta_\ell$, 
\begin{align}
\nonumber
\displaystyle \frac{\partial \mC}{\partial \theta_\ell} 
& = \frac 1n \mV \frac{\partial {\mLambda}}{\partial \theta_\ell} \mV^H
= \frac 1n \mV \bar{\mLambda}_{(\ell)} \mV^H,
\\
\nonumber
& \text{where} \quad \bar{\mLambda}_{(\ell)} = \diag(\bar{\vlambda}_{(\ell)}), \quad \text{and}
\\
\label{eqn:deriv_eigenval_gram_matrix}
&  \quad \bar{\vlambda}_{(\ell)} = \frac{\partial \vlambda}{\partial \theta_\ell} = \left( \frac{\partial \lambda_i}{\partial \theta_\ell} \right)_{i=1}^n 
= \left( \frac{\partial }{\partial \theta_\ell} \mV^H {\vC_1} \right)
= \mV^H \left( \frac{\partial }{\partial \theta_\ell} {C_\vtheta(\vx_1,\vx_i)} \right)_{i=1}^n,
\end{align}
where we used the fast Bayesian transform property \eqref{eqn:fast_transform_to_eigvalues}.
We use the notation $\bar{\vlambda}_{(\ell)} = \mV^H \bar{\vC_1}_{(\ell)}$ to denote the derivative of the eigenvalue ${\vlambda}_{(\ell)}$,  where $\bar{\vC}_{1(\ell)}$ denotes the first row of the gram matrix after taking the derivative in the $\ell$th variable, i.e.
\begin{align*}
\bar{\vC}_{1{(\ell)}} = \left(\frac{\partial }{\partial{\theta}_\ell} C_\vtheta(\vx_1,\vx_i) \right)_{i=1}^n.
\end{align*}
The goal is to compute the derivative of the objective function faster. First, let's rewrite the objective function from \eqref{eqn_MLE_loss_func_optimized_2} in two parts,
\begin{align*}
\mathcal{L}_{\MLE}(\vtheta | \vy) &= 
\underbrace{\frac{1}{n}  \log(\det\, \mC)}_{\mathcal{L}_{\abs{\mC}}} + \underbrace{\log\left((\vy-m_\MLE\vone)^T\mCInv(\vy-m_\MLE\vone) \right)}_{\mathcal{L}_{\vy}},
\\ &=: \mathcal{L}_{\abs{\mC}} + \mathcal{L}_{\vy}.
\end{align*}
Now, take the derivative:
\begin{align*}
\frac{\partial}{\partial \theta_\ell} \mathcal{L}_{\MLE}(\vtheta | \vy)
&=  \frac{\partial}{\partial \theta_\ell} \mathcal{L}_{\abs{\mC}} + \frac{\partial}{\partial \theta_\ell} \mathcal{L}_{\vy} \;.
\end{align*}
Now we tackle the individual terms,
\begin{align*}
\frac{\partial}{\partial \theta_\ell} \mathcal{L}_{\abs{\mC}} &= \frac{\partial}{\partial \theta_\ell}  \frac{1}{n} \log(\det\, \mC) 
\\ & = \frac 1n \trace{\left( \mCInv \frac{\partial \mC}{\partial \vtheta_\ell} \right)}
= \frac{1}{n}
\trace{\left( \mV {\mLambda}^{-1} \mV^H
	\frac 1n \mV \overline{\mLambda}_{(\ell)} \mV^H
	\right)}
\\
& = \frac{1}{n}
\trace{\left(
	\mV {\mLambda}^{-1}  \overline{\mLambda}_{(\ell)} \mV^H
	\right)}, \quad \text{where we used } \; \mV^H \mV = n,
\\
& =\frac{1}{n}
\trace{\left(
	\mV \;
	\diag\left( \frac{\overline{\lambda}_{i(\ell)}}{\lambda_i} \right)_{i=1}^n \mV^H
	\right)}
= \frac{1}{n} \sum_{i=1}^{n} \frac{\overline{\lambda}_{i(\ell)}}{\lambda_i},
\end{align*}
where we used the fact from \cite{Hig08},
\begin{align*}
\log(\det\, \mC)  = \trace{ (\log( \mC)) }.
\end{align*}
Part of the $\mathcal{L}_{\vy}$ was already simplified using the fast Bayesian transform,
\begin{align*}
{(\vy-m_\MLE\vone)^T\mCInv(\vy-m_\MLE\vone)} = \frac{1}{n} \sum_{i=2}^n \frac{\abs{\widetilde{y}_i}^2}{\lambda_i}.
\end{align*}
Using the above result,
\begin{align*}
\frac{\partial}{\partial \theta_\ell} \mathcal{L}_{\vy} 
&= \frac{\partial}{\partial \vtheta_\ell} \log\left(\frac{1}{n} \sum_{i=2}^n \frac{\abs{\widetilde{y}_i}^2}{\lambda_i} \right)
\\ 
&= \left(\frac{1}{n} \sum_{i=2}^n \frac{\abs{\widetilde{y}_i}^2}{\lambda_i}\right)^{-1}
\;
\frac{\partial}{\partial \vtheta_\ell} \left(\frac{1}{n} \sum_{i=2}^n \frac{\abs{\widetilde{y}_i}^2}{\lambda_i} \right)
\\ &= \left(
\frac{1}{n} \sum_{i=2}^n \frac{\abs{\widetilde{y}_i}^2}{\lambda_i} \right)^{-1} \frac{1}{n} \sum_{i=2}^n \frac{\abs{\widetilde{y}_i}^2}{\lambda_i^2}
\left( -\frac{\partial \lambda_i}{\partial \vtheta_\ell} \right)
\\ &= -\left(
\sum_{i=2}^n \frac{\abs{\widetilde{y}_i}^2}{\lambda_i} \right)^{-1} 
\left( \sum_{i=2}^n \abs{\widetilde{y}_i}^2 \frac{ \bar{ \lambda}_{i(\ell)} }{\lambda_i^2}
\right).
\end{align*}
Finally, using the above results,
\begin{align}
\label{eqn:deriv_obj_func_MLE}
\frac{\partial}{\partial \theta_\ell} \mathcal{L}_{\MLE}(\vtheta | \vy)
&=  \frac 1n \sum_{i=1}^{n} \frac{\bar{\lambda}_{i(\ell)}}{\lambda_i}
- \left({ \sum_{i=2}^n \frac{\abs{\tvy_i}^2 \bar{\lambda}_{i(\ell)}}{\lambda_i^2}}\right)
\left( {\sum_{i=2}^n \frac{\abs{\tvy_i}^2}{\lambda_\ell}} \right)^{-1},
\end{align}
where $\bar{\lambda}_{i(\ell)}$ is the derivative of the $i$th eigenvalue of $\mC$ in the $\ell$th variable. Please recollect the gradient descent proposed in \eqref{eqn:deep_descent} can be computed faster in $\Order(n \log n)$ using the result \eqref{eqn:deriv_obj_func_MLE}. A technique to compute this faster is discussed in Section \ref{sec:product_kernel}.

\Section{Full Bayes}

For the full Bayes approach the cubature is the same as for empirical Bayes.  We also defer to empirical Bayes to estimate the parameter $\vtheta$.  The width of the confidence interval is $\err_{\textup{full}} 
:= t_{n-1,0.995} \hsigma_{\textup{full}}$, where $\hsigma_{\textup{full}}^2$ can also be computed swiftly under assumptions \eqref{fastcompAssump}:
\begin{align*} 
\widehat{\sigma}^2_{\textup{full}} =
\frac{1}{n(n-1)} \sum_{i=2}^n \frac{\abs{\widetilde{y}_i}^2}{\lambda_i}
\left[\frac{\lambda_1}{n}{\left(1 - \frac{\widetilde{c}_1}{\lambda_1}\right)^2} + \left(c_0  - \frac 1n \sum_{i=1}^n \frac{\abs{\widetilde{c}_i}^2}{\lambda_i}\right) \right],
\end{align*}
Under assumption \eqref{addAssump} further simplification can be made:
\begin{equation*} 
\widehat{\sigma}^2_{\textup{full}}
=
\frac{1}{n(n-1)} \sum_{i=2}^n \frac{\abs{\widetilde{y}_i}^2}{\lambda_i} \, \left(\frac{\lambda_1}{n}  - 1  \right),
\end{equation*}
It follows that
\begin{equation*} 
\err_{\textup{full}}
=
t_{n-1,0.995}
\sqrt{\frac{1}{n(n-1)} \sum_{i=2}^n \frac{\abs{\widetilde{y}_i}^2}{\lambda_i} \, \left(\frac{\lambda_1}{n}  - 1  \right)}.
\end{equation*}

\Section{Generalized Cross-Validation}

GCV yields a different cubature, which nevertheless can also be computed quickly using the fast Bayesian transform.  Under assumptions \eqref{fastcompAssump}:
\begin{align}
\nonumber
m_{\textup{GCV}} & = m_\MLE =  \frac{\widetilde{y}_1}{n} = \frac 1n \sum_{i=1}^n y_i,\\
\nonumber 
s^2_{\textup{GCV}} & : =  \frac 1{n} \sum_{i=2}^n \frac{\abs{\widetilde{y}_i}^2}{\lambda_i^2}  \left [ \sum_{i=1}^n \frac{1}{\lambda_i} \right]^{-1}, \\
\label{eqn:thetaGCV_fast}
\vtheta_{\GCV} 
&= \argmin_\vtheta \left[ \log \left ( \sum_{i=2}^n \frac{\abs{\widetilde{y}_i}^2}{\lambda_i^2} 
\right) -2\log\left( \sum_{i=1}^n \frac{1}{\lambda_i} \right)
\right], \\
\nonumber
\hmu_{\GCV}
&= \hmu_\MLE  = \frac{\widetilde{y}_1}{n} +
\frac 1n \sum_{i=2}^n \frac{ \widetilde{c}_i^* \widetilde{y}_i}{\lambda_i}, \\
\nonumber
\err_{\textup{GCV}} & =
\frac{2.58}{n}\left\{\sum_{i=2}^n \frac{\abs{\widetilde{y}_i}^2}{\lambda_i^2}  \left [ \frac 1n \sum_{i=1}^n \frac{1}{\lambda_i} \right]^{-1}  \times
\left( c_0 - \frac 1n \sum_{i=1}^n \frac{\abs{\widetilde{c}_i}^2}{\lambda_i} \right) 
\right\}^{1/2}.
\end{align}
Moreover, under further assumption \eqref{addAssump} it follows that 
\begin{align*}
\nonumber 
\hmu_{\textup{GCV}}
&= \hmu_\MLE = \hmu_{\textup{full}} =
\frac{\widetilde{y}_1}{n} = \frac 1n \sum_{i=1}^n y_i,\\
\err_{\textup{GCV}} & =
\frac{2.58}{n}\left\{\sum_{i=2}^n \frac{\abs{\widetilde{y}_i}^2}{\lambda_i^2}  \left [ \frac 1n \sum_{i=1}^n \frac{1}{\lambda_i} \right]^{-1}  
\left( 1 -  \frac{n}{\lambda_1} \right)  
\right\}^{1/2}.
\end{align*}
In this case too, $\hmu$ is simply the sample mean.

\Subsection{Gradient of the objective function}
Using the results obtained from the Section \ref{deriv_obj_func_MLE} with empirical Bayes, one can reduce the computational cost of the derivative of the objective function in \eqref{eqn:thetaGCV_fast},
\begin{align*}
\mathcal{L}_\GCV(\vtheta | \vy)
&= \log \left ( \sum_{i=2}^n \frac{\abs{\widetilde{y}_i}^2}{\lambda_i^2} 
\right) -2\log\left( \sum_{i=1}^n \frac{1}{\lambda_i} \right).
\end{align*}
Using the similar techniques from \secref{deriv_obj_func_MLE}, the derivative of the objective function w.r.t $\theta_\ell$:
\begin{align*}
& \frac{\partial}{\partial \theta_\ell}  \mathcal{L}_\GCV(\vtheta | \vy)
\\
&= \left ( \sum_{i=2}^n \frac{\abs{\widetilde{y}_i}^2}{\lambda_i^2} \right)^{-1}
\frac{\partial}{\partial \theta_\ell} \left ( \sum_{i=2}^n \frac{\abs{\widetilde{y}_i}^2}{\lambda_i^2} \right)
-2 \left( \sum_{i=1}^n \frac{1}{\lambda_i} \right)^{-1}
\frac{\partial}{\partial \theta_\ell} \left( \sum_{i=1}^n \frac{1}{\lambda_i} \right)
\\
&= \left ( \sum_{i=2}^n \frac{\abs{\widetilde{y}_i}^2}{\lambda_i^2} \right)^{-1}
\left ( \sum_{i=2}^n \frac{\abs{\widetilde{y}_i}^2}{\lambda_i^3} (-2) \frac{\partial\lambda_i}{\partial \theta_\ell}  \right)
\\ & \hspace{4cm} 
-2 \left( \sum_{i=1}^n \frac{1}{\lambda_i} \right)^{-1}
\left( \sum_{i=1}^n \frac{1}{\lambda_i^2} (-1)\frac{\partial \lambda_i}{\partial \theta_\ell}  \right)
\\
&= -2 \left ( \sum_{i=2}^n \frac{\abs{\widetilde{y}_i}^2}{\lambda_i^2} \right)^{-1}
\left ( \sum_{i=2}^n \frac{\abs{\widetilde{y}_i}^2 \bar{\lambda}_{i(\ell)} }{\lambda_i^3}    \right)
+ 2 \left( \sum_{i=1}^n \frac{1}{\lambda_i} \right)^{-1}
\left( \sum_{i=1}^n \frac{\bar{\lambda}_{i(\ell)} }{\lambda_i^2}  \right).
\end{align*}
Thus,
\begin{multline}
\frac{\partial}{\partial \theta_\ell}  \mathcal{L}_\GCV(\vtheta | \vy)
\label{eqn:deriv_obj_func_full}
= -2 \left ( \sum_{i=2}^n \frac{\abs{\widetilde{y}_i}^2}{\lambda_i^2} \right)^{-1}
\left ( \sum_{i=2}^n \frac{\abs{\widetilde{y}_i}^2 \bar{\lambda}_{i(\ell)} }{\lambda_i^3}    \right) \\
+ 2 \left( \sum_{i=1}^n \frac{1}{\lambda_i} \right)^{-1}
\left( \sum_{i=1}^n \frac{\bar{\lambda}_{i(\ell)} }{\lambda_i^2}  \right),
\end{multline}
where $\bar{\lambda}_{i(\ell)}$ is the derivative of the $i$th eigenvalue of the Gram matrix, $\mC$, in the $\ell$th variable. We discuss a technique to compute $\bar{\lambda}_{i(\ell)}$ in the next section below.

\Section{Product Kernels}
\label{sec:product_kernel}

In this research, we use product kernels in the demonstrations and numerical implementations. They got nice properties which are helpful to obtain analytical results easily. Product kernels in $d$ dimensions are of the form,
\begin{align}
\label{eqn:prod_kernel}
C_\vtheta(\vt, \vx) = 
\prod_{\ell=1}^d \biggl[ 1 - \eta_\ell \; \mathfrak{C}(x_\ell,t_\ell) \biggr]
\end{align}
where $\eta_\ell$ is called shape parameter in the $\ell$th variable for $\ell=1,\cdots,d$, and $\mathfrak{C}$ is chosen such that to ensure $C_{\vtheta}$ is symmetric and positive definite. Our goal is to compute $\bar{\lambda}_{i(\ell)}$ for which the kernel derivative is necessary. The derivative of the product kernels can be obtained easily. Please note that $\vtheta$ denotes all the hyper parameters of the kernel $C$ where $\eta$ is one of them and called the shape parameter.

\Subsection{Derivative of the product kernel when $\eta_1=\cdots=\eta_d=\eta$}
\label{sec:deriv_of_kernel}
It was suggested to use gradient descent to find optimal shape parameter in \secref{grad_descent_MLE}. In this section, we compute the gradient for product kernels. When the $\eta_1=\cdots=\eta_d=\eta$, the derivative of a product kernel  w.r.t. $\eta$ can be obtained as below,
\begin{align*}
\frac{\partial}{\partial \eta} C_\vtheta(\vt, \vx) 
& =
\frac{\partial}{\partial \eta} 
\prod_{j=1}^d \biggl[
1 - \eta \mathfrak{C}(x_j,t_j) \biggr] 
\\
& = 
\sum_{\ell=1}^d  
\prod_{j=1, j \neq \ell}^d \biggl[
1 - \eta \mathfrak{C}(x_j,t_j) \biggr]
\biggl( - \mathfrak{C}(x_\ell,t_\ell) \biggr)
\\
& =
\prod_{j=1}^d \biggl[
1 - \eta \mathfrak{C}(x_j,t_j) \biggr]
\sum_{\ell=1}^d 
\frac{
	\biggl( - \mathfrak{C}(x_\ell,t_\ell) \biggr)
}{
	1 - \eta \mathfrak{C}(x_\ell,t_\ell) 
}
\\
& =
C_\vtheta(\vt, \vx) 
\frac{1}{\eta}
\sum_{\ell=1}^d 
\frac{
	\biggl(1 - \eta \mathfrak{C}({x_\ell,t_\ell})  - 1 \biggr)
}{
	1 - \eta \mathfrak{C}(x_\ell,t_\ell) 
}
\\
& =
C_\vtheta(\vt, \vx) 
\frac{1}{\eta}
\sum_{\ell=1}^d 
\biggl(
1 - 
\frac{1
}{
	1 - \eta \mathfrak{C}(x_\ell,t_\ell) 
}
\biggr)
\\
& =
({d}/{\eta} )
\underbrace{
	\left(
	\prod_{j=1}^d \biggl[
	1 - \eta \mathfrak{C}(x_j,t_j) \biggr]
	\right) }_
{ C_\vtheta(\vt, \vx) }
\biggl(
1 - 
\frac{1}{d} \sum_{\ell=1}^d
\frac{1}
{ 1 - \eta \mathfrak{C}(x_\ell,t_\ell) }
\biggr)
.
\end{align*}
Thus,
\begin{align*}
\frac{\partial}{\partial \eta} C_\vtheta(\vt, \vx) = ({d}/{\eta} ) C_\vtheta(\vt, \vx) 
\biggl(
1 - 
\frac{1}{d} \sum_{\ell=1}^d
\frac{1}
{ 1 - \eta \mathfrak{C}(x_\ell,t_\ell) }
\biggr).
\end{align*}

\Subsubsection{When $\eta_\ell$ is different for each $\ell = 1,\cdots,d$}
In this case, we will have a vector of length $d$ shape parameters. Derivative of the kernel, $C_\vtheta(\vt, \vx)$ \eqref{eqn:prod_kernel}, with respect to $\eta_\ell$ is,
\begin{align*}
\frac{\partial}{\partial \eta_\ell} C_\vtheta(\vt, \vx) 
& =
\frac{\partial}{\partial \eta_\ell} 
\prod_{j=1}^d \biggl[
1 - \eta_j \mathfrak{C}(x_j,t_j) \biggr], \quad \text{where} \quad \ell = 1, \cdots, d
\\
& = 
\prod_{j=1, j \neq \ell}^d \biggl[
1 - \eta_j \mathfrak{C}(x_j,t_j) \biggr]
\biggl( - \mathfrak{C}(x_\ell,t_\ell) \biggr)
\\
& =
\prod_{j=1}^d \biggl[
1 - \eta_j \mathfrak{C}(x_j,t_j) \biggr]
\frac{
	\biggl( - \mathfrak{C}(x_\ell,t_\ell) \biggr)
}{
	1 - \eta_\ell \mathfrak{C}(x_\ell,t_\ell) 
}
\\
& =
C_\vtheta(\vt, \vx) 
\frac{1}{\eta_\ell}
\frac{
	\biggl(1 - \eta_\ell \mathfrak{C}(x_\ell,t_\ell)  - 1 \biggr)
}{
	1 - \eta_\ell \mathfrak{C}(x_\ell,t_\ell) 
}
\\
& =
C_\vtheta(\vt, \vx) 
\frac{1}{\eta_\ell}
\biggl(
1 - 
\frac{1
}{
	1 - \eta_\ell \mathfrak{C}(x_\ell,t_\ell) 
}
\biggr)
\\
& =
\frac{1}{\eta_\ell} 
\underbrace{
	\left(
	\prod_{j=1}^d \biggl[
	1 - \eta \mathfrak{C}(x_j,t_j) \biggr]
	\right) }_
{ C_\vtheta(\vt, \vx) }
\biggl(
1 - 
\frac{1}
{ 1 - \eta_\ell \mathfrak{C}(x_\ell, t_\ell) }
\biggr) 
.
\end{align*}
Thus, 
\begin{align*}
\frac{\partial}{\partial \eta_\ell} C_\vtheta(\vt, \vx) = \frac{1}{\eta_\ell} 
{ C_\vtheta(\vt, \vx) }
\biggl(
1 - 
\frac{1}
{ 1 - \eta_\ell \mathfrak{C}(x_\ell,t_\ell) }
\biggr).
\end{align*}
Please note that the above derivatives do not depend on $\mathfrak{C}(x,t)$ and most importantly these computations are applicable to any product kernel of the form \eqref{eqn:prod_kernel}.
The $\bar{\lambda}_{i(\ell)}$ can be computed now using \eqref{eqn:deriv_eigenval_gram_matrix} with the computed kernel derivative, $\frac{\partial}{\partial \eta_\ell} C_\vtheta$.

\Subsection{Shape parameter search using steepest descent}
Using the obtained derivative of the eigenvalues, $\bar{\lambda}_{i(\ell)}$, one can easily compute the gradient of the objective function \eqref{eqn:deriv_obj_func_MLE} or \eqref{eqn:deriv_obj_func_full}. This can be further used to implement the steepest descent search as introduced in \secref{grad_descent_MLE} 
\begin{align*}
\eta^{(j+1)}_\ell = \eta^{(j)}_\ell - \nu \frac{\partial}{\partial \eta_\ell} \mathcal{L}(\vtheta | \vy), \quad j=0,1,\cdots,  \quad \ell = 1, \cdots, d
\end{align*}
where $\nu$ is the step size for the gradient descent, $j$ is the iteration index, and $\frac{\partial}{\partial \eta_\ell} \mathcal{L}(\vtheta | \vy)$ is either \eqref{eqn:deriv_obj_func_MLE} or \eqref{eqn:deriv_obj_func_full} depending on the choice of the hyperparameter search method. The parameter $\eta_\ell$ is usually searched in the whole $\reals$ by using the simple domain transformation as explained in Section~\ref{sec:kernel_param_search}.


\Chapter{Integration lattices and \\ Shift Invariant Kernels}
\label{sec:shift_invariant_kernel}

The preceding sections lay out an automatic Bayesian cubature algorithm whose computational cost is drastically reduced.  However, this algorithm relies on covariance kernel functions, $C_{\vtheta}$ and node sets, $\{\vx_i\}_{i=1}^n$ that satisfy assumptions \eqref{fastcompAssump}.  
In this chapter, we demonstrate such a covariance kernel and matching design.
When periodic shift-invariant kernels are combined with rank-1 lattice nodes, the resulting Gram matrix is symmetric and circulant. 
This combination also satisfies assumption \eqref{addAssump}.  To conveniently facilitate the fast Bayesian transform, it is assumed in this section and the next that $n$ is power of $2$.

\Section{Extensible Integration Lattice Node Sets}

We choose set of nodes defined by a shifted extensible integration lattice node sequence, which takes the form
\begin{equation}
\label{eqn:lattice_def}
\vx_{i} = \vh \phi(i-1) + \vDelta \mod \vone, \qquad i \in \naturals.
\end{equation} 
Here, $\vh$ is a $d$-dimensional generating vector of positive integers, $\vDelta$ is some point in $[0,1)^d$, often chosen at random, and $\{\phi(i)\}_{i=0}^n$ is the van der Corput sequence, defined by reflecting the binary digits of the integer about the decimal point, i.e., 
\begin{equation} \label{vdCDef}
\begin{array}{r|ccccccccccccc}
i & 0 & 1 & 2 & 3 & 4 &  5 & 6 & 7 & \cdots \\
i & 0_2 & 1_2 & 10_2 & 11_2 & 100_2 & 101_2 & 110_2 & 111_2  & \cdots\\
\toprule
\phi(i) & {}_2.0 &  {}_2.1 & {}_2.01 &  {}_2.11  & {}_2.001 &  {}_2.101 & {}_2.011 &  {}_2.111 & \cdots\\
\phi(i) & 0 &  0.5 &  0.25 & 0.75 &  0.125 & 0.625  &  0.375 & 0.875 & \cdots
\end{array}
\end{equation}
Note that 
\begin{align} \label{phiprop}
n\phi:\{0, \ldots, n-1 \} \to \{0, \ldots, n-1\} \quad
\text{is one-to-one},
\end{align}
assuming $n$ is a power of $2$.

These node sets are called shifted rank-1 lattice node sets.
A random shift $\vDelta$ is added to $\vh \phi(i-1)$ to get $\{\vx_{i}\}_{i=1}^n$ which is to avoid zero at the origin in the node sets. However, this shift does not disturb the discrepancy properties of $\{\vx_{i}\}_{i=1}^n$. 
The rank-1 lattices with the modulo one addition have a very desirable group structure that helps to satisfy fast Bayesian transform kernel assumptions.

An example of $64$ nodes is given in Figure~\ref{latticefig}.  The even coverage of the unit cube is ensured by a well chosen generating vector $\vh$.  The choice of generating vector is typically done offline by computer search.  Please refer to \cite{DicEtal14a, HicNie03a} for more on extensible integration lattices. Lattice rules are designed to integrate the class of certain sinusoidal functions without error.
\begin{figure}[htp]
	\centering
	\includegraphics[width=0.8\linewidth]{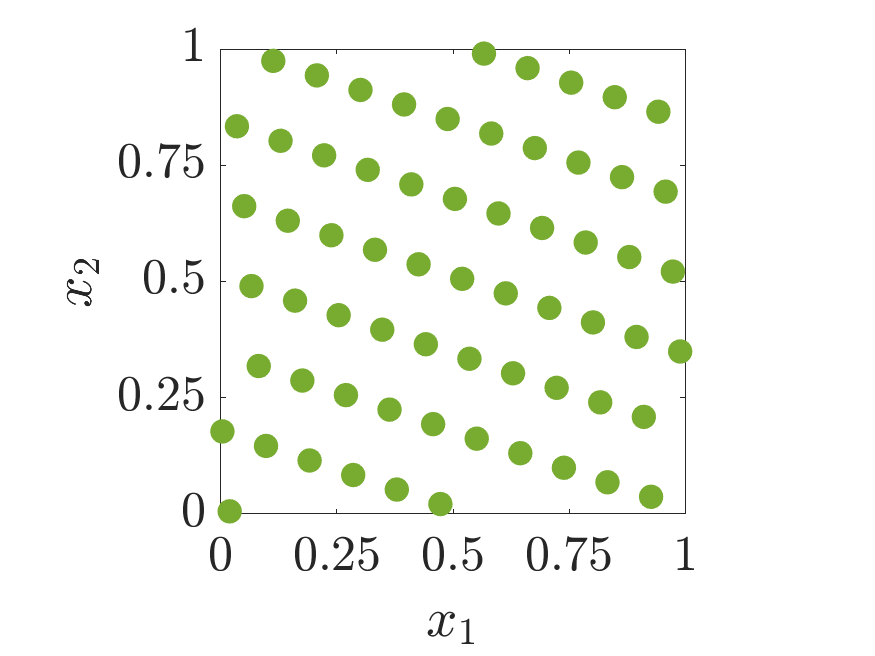}
	\caption{Example of a shifted integration lattice node set  in $d=2$. This plot can be reproduced using \code{PlotPoints.m}. \label{latticefig} }
\end{figure}

\Section{Shift Invariant Kernels}\label{sec:shift_invar_kern}

The covariance functions $C_{\vtheta}$ that match integration lattice node sets have the form
\begin{equation} \label{eq:shInv}
C_{\vtheta}(\vt,\vx) = K_{\vtheta}(\vt - \vx \bmod \vone).
\end{equation}
This is called a \emph{shift invariant kernel} because shifting both arguments of the covariance function by the same amount leaves the value unchanged.   By a proper scaling of the function $K_{\vtheta}$, the kernel satisfies the assumption \eqref{addAssump}. 
Here, $K_{\vtheta}$ is chosen such that to ensure
$C_{\vtheta}$ is symmetric and positive definite, as assumed in \eqref{FJH:eq:CondPosDef}. 

A family of shift invariant kernels is constructed via even degree Bernoulli polynomials. Symmetric, periodic, positive definite kernels of this form appear in  \cite{DicEtal14a} and \cite{Hic96a}:
\begin{align*}
C_\vtheta(\vx, \vt) := &  \sum_{\vk \in \mathbb{Z}^d} \alpha_{\vk,\vtheta}  e^{2 \pi\sqrt{-1} \vk^T\vx}
e^{-2 \pi\sqrt{-1} \vk^T\vt}, \quad \alpha_{-\vk,\vtheta} = \alpha_{\vk,\vtheta}
\end{align*}
where $d$ is the number of dimensions and $\alpha_{\vk}$ is a positive scalar. The Gram matrix formed by this kernel is symmetric and positive definite. 
The \textit{shape parameter} $\eta_\ell$ changes the kernel's shape, so that the integrand is in the middle of the function space spanned by the kernel. 
If the coefficients are chosen as 
\begin{align*}
\alpha_{\vk,\vtheta} := & \prod_{\ell=1, k_\ell \ne 0}^d 
\frac{{\eta_\ell}}{{|k_\ell|^r}}  \;,  \quad \text{with} \; {\alpha}_{\bm{0},\vtheta} = 1, \; r \in \naturals, 
\end{align*}
then there exists a simpler closed form expression.
\begin{multline}
\label{the_kernel_eqn_bernoulli}
K_\vtheta(\vx) =
\prod_{\ell=1}^d \biggl[
1 - (-1)^{r} \eta_\ell B_{2r}( {x_\ell} ) \biggr],  \\
\forall \vx \in [0,1]^d,   \vtheta := (r,\veta), \ r \in \naturals, \ \eta_\ell > 0. \qquad \qquad
\end{multline}
Larger $r$ implies a greater degree of smoothness of the kernel.  Larger $\eta_\ell$ implies greater fluctuations of the output with respect to the input $x_\ell$.  
The Bernoulli polynomials $B_{r}(x)$ are described in \cite[Chapter 24]{OlvEtal10a}
\begin{align*}
B_{r}(x) = \frac{-r!}{(2 \pi \sqrt{-1})^{r}} 
\sum_{\substack{k \neq 0,\\ k=-\infty}}^\infty 
\frac{e^{2\pi\sqrt{-1} k x}}{k^{r}}
\;\;
\begin{cases}
\text{for} \;\; r=1, \;\; 0 < x < 1 \\
\text{for} \;\; r=2,3,\hdots \;\; 0 \leq x \leq 1
\end{cases}
\end{align*}
Plots of $C(\cdot, 0.3)$ are given in \figref{fig:fourierkernel-dim1} for $d=1$ and for various $r$ and $\eta_1$ values.

\begin{figure}
	\centering  
	\includegraphics[width=0.9\linewidth]{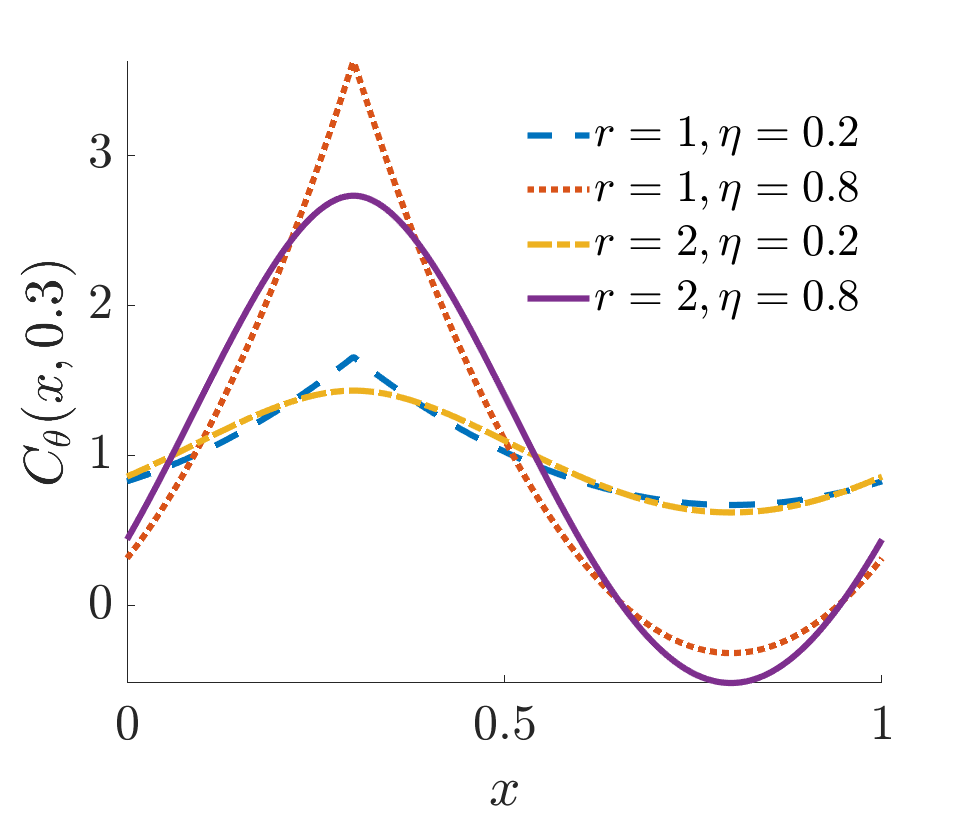}
	\caption[Fourier kernel]{Shift invariant kernel in $d=1$ shifted by 0.3 to show the discontinuity. This plot can be reproduced using \code{plot\_fourier\_kernel.m}}
	\label{fig:fourierkernel-dim1}
\end{figure}
Lattice cubature rules are known to have convergence rates that depend on the smoothness of the integrands, but that are rather independent of the choice of the integration lattice \cite{DicEtal14a}.  Thus, we expect integration lattice node sets to perform well regardless of the smoothness of the covariance kernel.  The bigger concern is whether the derivatives of the integrand are as smooth as the covariance kernel implies.  This topic is touched upon again in Section~\ref{sec:period_var_tx}.

\Subsection{Eigenvectors}
For general shift-invariance covariance functions the Gram matrix 
\begin{align}
\label{shInvKernGramMatrix}
\mC_\vtheta &= \bigl ( C_\vtheta(\vx_i, \vx_j) \bigr)_{i, j = 1}^n 
\end{align}
can be shown that to have the eigenvector matrix 
\begin{equation} \label{latticeVdef}
\mV = \Bigl ( \me^{2 \pi n \sqrt{-1} \phi(i-1)\phi(j-1)} \Bigr)_{i = 1}^n.
\end{equation}

One can interpret the sequence reordering from $\{\phi(i-1)\}_{i=1}^n$ to $(0, \ldots, 1-1/n)$, for $n$ a power of $2$, as a permutation. 
Let
\begin{equation} \label{PermMat}
\mP = \bigl( \delta_{n\phi(i-1), j-1}  \bigr)_{i,j=1}^n
\end{equation}
be a permutation matrix, where $\delta_{\cdot,\cdot}$ is the Kronecker delta function.  Then,
\begin{align}
\nonumber
\mC_\vtheta &= \bigl ( C_\vtheta(\vx_i, \vx_j) \bigr)_{i, j = 1}^n \\
\nonumber
& = \Bigl ( K_\vtheta \bigl(\vh(\phi(i-1) - \phi(j-1) \bigr) \bmod \vone ) \Bigr)_{i, j = 1}^n 
\qquad  \text{by \eqref{eqn:lattice_def} and \eqref{eq:shInv}}  \\
\nonumber
& = \biggl( 
\sum_{i',j'=1}^n \delta_{n\phi(i-1), i'-1}  \,
K_\vtheta \bigl (\vh (i'-j')/n \bmod \vone \bigr )
\delta_{j'-1,n\phi(j-1)} 
\biggr)_{i,j=1}^n \\
& = \mP \mK_{\vtheta}  \mP^T, \quad \text{by \eqref{PermMat}} \label{shInvKernGramMatrixPermut} 
\intertext{where } 
\mK_{\vtheta} &= \bigl ( K_\vtheta \bigl (\vh (i-j)/n \bmod \vone \bigr ) \bigr)_{i, j = 1}^n.
\end{align}
Because $\mK_\vtheta$ is circulant, we know the form of it's eigenvector-eigenvalue decomposition:
\begin{equation} \label{Keig}
\mK_{\vtheta} = \frac 1n \mW \mLambda_\vtheta \mW^H, \quad \text{where} \quad \mW =  \Bigl ( \me^{2 \pi \sqrt{-1} (i-1)(j-1)/n} \Bigr)_{i,j = 1}^n
\end{equation}
where $\mLambda_\vtheta$ is a diagonal matrix.
By \eqref{shInvKernGramMatrixPermut} we then have the eigenvector-eigenvalue decomposition for $\mC_{\vtheta}$ assumed in \eqref{eqn:ftk_factor}, namely
\begin{align*} 
\mC_{\vtheta} &= \mP \mK_\vtheta \mP^T \\
&= \frac 1n  \mP \mW \mLambda_\vtheta \mW^H \mP^T 
= \frac 1n  \underbrace{\mP \mW \mP^T} \mLambda_\vtheta \underbrace{\mP \mW^H \mP^T} 
\\
&= \frac 1n \mV \mLambda_\vtheta \mV^H .
\end{align*}
Thus
\begin{equation}
\label{Clateig}
\mC_{\vtheta} = \frac 1n \mV \mLambda_\vtheta \mV^H , \qquad \mV = \mP \mW \mP^T,
\end{equation}
where the eigenvalues of  $\mC_{\vtheta}$ and $\mK_{\vtheta}$ are identical.
Note that the matrix multiplication by $\mV$ can be performed in $\Order(n \log n)$ operations using the FFT.

\Section{Continuous Valued Kernel Order}
\label{sec:non_integer_kernel_order}

\JRNote{Need better and more convincing motivation}

In the previous sections, we assumed that the shift-invariant kernel's order is an even valued integer and also fixed. It requires the practitioner to be aware of the integrand's smoothness to precisely handpick the kernel order to match the integrand's smoothness. However, it is not possible to know the integrand's smoothness in most of the practical applications. The constraint to have an integer-valued kernel order also limits the ability to continuously vary the kernel's smoothness to match the integrand like the shape parameter is varied to match. 

The integer kernel order is not suitable to optimally search by standard optimization algorithm.
As a consequence, one usually ends up choosing a higher kernel order when the integrand is not  smooth or lower kernel order when the integrand is very smooth.
Often it leads to longer computation time or poor accuracy in the numerical integration.
Here we explore two alternative forms of the kernel which allow the kernel order to be positive continuous value greater than one or a continuous value in the range $(0,1)$. Let us recall the infinite series expression that was used to construct the kernel \eqref{the_kernel_eqn_bernoulli}:
\begin{align*}
C_\vtheta(\vx, \vt) := &  \sum_{\vk \in \mathbb{Z}^d} \alpha_{\vk,\vtheta}  e^{2 \pi\sqrt{-1} \vk^T\vx}
e^{-2 \pi\sqrt{-1} \vk^T\vt}, \quad \text{where} \; 
\alpha_{\vk,\vtheta} = \prod_{\ell=1}^d \frac{\eta_\ell}{{|k_\ell|}^r} 
\end{align*}
and $\vtheta = (r, \veta)$.  This form is convenient for analytical derivations.
To make the derivations easier to follow, we fix the dimension $d=1$,
\begin{align*}
C_\vtheta(x, t) = & 1 + \eta \sum_{k \in \mathbb{Z}, k \neq 0 } \frac{1}{\abs{k}^r} 
e^{ 2 \pi\sqrt{-1} k x}
e^{-2 \pi\sqrt{-1} k t}.
\end{align*}

\Subsection{Truncated series kernel}
\label{sec:trunc_series_kernel}
The following variation to the infinite series kernel \eqref{the_kernel_eqn_bernoulli} has the kernel order in the interval $(1, \infty)$. This kernel provides algebraic decay but it is more robust in the hyperparameter search.
We reuse the original definition of the infinite kernel \eqref{the_kernel_eqn_bernoulli} but truncate to a finite length. This allows the kernel order $r$ continuous valued so that it does not have to be an even integer,  which was a constraint previously. 
For $d=1$,
\begin{align*}
C_\vtheta(x, t) = & 1 + \eta \sum_{k \in \mathbb{Z}, k \neq 0 } \frac{1}{\abs{k}^r} 
e^{ 2 \pi\sqrt{-1} k (x-t)},
\end{align*}
where $\theta = (r, \eta)$. 
Since the infinite sum cannot be used directly, we truncate to length $n$,
\begin{align*}
C_{\vtheta, n}(x, t) = & 1 + \eta \sum_{k = - n/2 }^{n/2 - 1} \frac{1}{\abs{k}^r} 
e^{ 2 \pi\sqrt{-1} k (x-t)}.
\end{align*}
The Gram matrix is written as 
\begin{align*}
\mC_{\vtheta, n} = & \biggl( C_{\vtheta, n}(\vx_i, \vx_j) \biggr)_{i,j=1}^n,
\end{align*}
where $n$ is the number of samples. 
The reason for having the truncation length and the number of samples equal will be obvious as we proceed further.
The first column of the Gram matrix is
\begin{align*}
\vC_{\vtheta, n} &= \biggl( C_{\vtheta, n}(\vx_i, \vx_1) \biggr)_{i=1}^n
\\
&= \left( \prod_{\ell=1}^d \left[ 1 + \eta_\ell \sum_{k = - n/2, k \neq 0 }^{n/2 - 1} \frac{1}{\abs{k_l}^r} 
e^{ 2 \pi\sqrt{-1} k_l (x_{i\ell}-x_{1\ell})}\right] \right)_{i=1}^n, 
\end{align*}
where $d$ is the number of dimensions. 
However the direct computation involves $n^2$ computations since we have chosen the truncation length to $n$. We can reduce the computations to $\Order(n\log n)$ using the FFT.
Define
\begin{align*}
\mathfrak{C}_r(t) &:= \sum_{k = - n/2, k \neq 0 }^{n/2 -1} \frac{1}{\abs{k}^r} 
e^{ 2 \pi\sqrt{-1} k\, t}.
\end{align*}
Using the $\mathfrak{C}_r$, rewrite
\begin{align}
\label{eqn:trunc_series_kernel}
\vC_{\vtheta, n}
&= \left( \prod_{l=1}^d \left[ 1 + \eta \mathfrak{C}_r( \abs{x_{il} - x_{1l}})\right] \right)_{i=1}^n.
\end{align}
One can observe
$\abs{x_{i\ell}-x_{1\ell}} \in \lbrace 0, \frac 1n, \frac 2n, \dots \frac{n-1}{n}  \rbrace$ by using the definition of lattice points from \eqref{eqn:lattice_def}. This can be used to rewrite $\mathfrak{C}_r$ in a much simpler form,
\begin{align*}
\mathfrak{C}_r \left(\frac jn \right) &= \sum_{k = - n/2, k \neq 0 }^{n/2 - 1} \frac{1}{\abs{k}^r} 
e^{ 2 \pi\sqrt{-1} k (\frac jn)}, \quad \text{where} \;  j=0,1,\dots n-1.
\end{align*}
This notation is very convenient to show that $\widetilde{\mathfrak{C}}_r$,  the discrete Fourier transform of $\mathfrak{C}_r$, can be computed analytically
\begin{align*}
\widetilde{\mathfrak{C}}_r(m) &= \sum_{j=0}^{n-1} \mathfrak{C}_r (j/n) e^{- 2 \pi\sqrt{-1} jm/n} 
\\
&= \sum_{k = - n/2, k \neq 0 }^{n/2-1} 
\sum_{j=0}^{n-1} \frac{1}{\abs{k}^r} e^{ 2 \pi\sqrt{-1} (k-m)  j/n} , \quad \text{by \eqref{eqn:dft_delta_fact}}
\\
&= \sum_{k = - n/2, k \neq 0  }^{n/2-1} \frac{n}{\abs{k}^r} \; \delta_{k-m \bmod n, 0} \; .
\end{align*}
This is the reason we have chosen the truncation length to $n$. 
Based on the above result, it is evident that $\widetilde{\mathfrak{C}}_r$ can be computed analytically,
\begin{align} \label{dft_of_g}
\widetilde{\bm{\mathfrak{C}}}_r := \left(\widetilde{\mathfrak{C}}_{r}(m)\right)_{m=0}^{n-1}, \quad
\text{where} \quad
\widetilde{\mathfrak{C}}_r(m) = 
\begin{cases}
0, & \text{for} \quad m=0 \\
\frac{n}{\abs{m}^r}, & \text{for} \quad m=1,\dots,n/2-1 \\
\frac{n}{\abs{n-m}^r}, & \text{for} \quad m=n/2,\dots,n-1
\end{cases}
\end{align}
where we used the fact,
\begin{align}
\label{eqn:dft_delta_fact}
\sum_{i=0}^{n-1} e^{2 \pi \sqrt{-1} i j /n} = 
\begin{cases}
\frac{1 - e^{2\pi \sqrt{-1} j n /n}}{1 - e^{2\pi \sqrt{-1} j /n}} = 0, &j \ne 0 \bmod n
\\
n, & j = 0 \bmod n.
\end{cases}
\end{align}
Having these results, we can easily back-compute $\mathfrak{C}$ using inverse discrete Fourier transform. It can be shown that inverse DFT of $\widetilde{\mathfrak{C}}_{r}$ returns $\mathfrak{C}$,
\begin{align*}
\frac{1}{n} &\sum_{m=0}^{n-1} \widetilde{\mathfrak{C}}_{r} (m) e^{2 \pi\sqrt{-1} lm/n} \\
& = \frac{1}{n} \sum_{m=0}^{n-1} 
\sum_{j=0}^{n-1} \mathfrak{C}_{r} (j/n) e^{- 2 \pi\sqrt{-1} jm/n}
e^{2 \pi\sqrt{-1} lm/n}, \quad \text{by \eqref{eqn:dft_delta_fact}}  \\
& = \frac{1}{n}  \sum_{j=0}^{n-1} \mathfrak{C}_{r} (j/n) n \delta_{(l-j) \bmod n, 0} \\
& = \mathfrak{C}_{r} (l/n), \quad \text{for} \quad l=0,\dots,n-1
\end{align*}
This implies that to compute $n$ values of $\biggl( C_{\vtheta, n}(\vx_i, \vx_1) \biggr)_{i=1}^n$, we need to have the number of samples and the truncation length the same. 
The above results are summarized as an algorithm to compute $\mathfrak{C}$ using FFT in Algorithm~\ref{algorithm-cont-val-kernel-order}.
\begin{algorithm}
	\caption{The kernel with continuous valued order}\label{algorithm-cont-val-kernel-order}
	\begin{algorithmic}[1]
	
	\Require
	Number of points to use, $n$; 

	\State Analytically compute $ \widetilde{\bm{\mathfrak{C}}}_r $ in \eqref{eqn:trunc_series_kernel}, the discrete Fourier transform of ${\bm{\mathfrak{C}}}_r$ using \eqref{dft_of_g}
	\State Take the inverse FFT of $\widetilde{\bm{\mathfrak{C}}}_r$ to get ${\bm{\mathfrak{C}}}_r$
	
	\State Using ${\bm{\mathfrak{C}}}_r$ compute the truncated series of kernel of truncation length $n$ using \eqref{eqn:trunc_series_kernel}

	\end{algorithmic}

\end{algorithm}

In Algorithm \ref{algorithm-cont-val-kernel-order}, the computational cost of  computing ${\bm{\mathfrak{C}}}_r$ 
is $\Order(n \log n)$ instead of $\Order(n^2)$. Plugging-in the values of ${\bm{\mathfrak{C}}}_r$ in \eqref{eqn:trunc_series_kernel} gives the kernel. Another major benefit is that the FFT approach in Algorithm \ref{algorithm-cont-val-kernel-order} is the computations are numerically more stable than the direct sum approach. Please note that these kernels evolve with the truncation length $n$. The larger $n$ value the closer the kernel resembles the original infinite series kernel. One disadvantage is, the truncated series kernels obtain algebraic order decay at best. The infinite series kernel with little modification can be enhanced to obtain exponential decay as shown next.

\Subsection{Exponentially decaying kernel}
\label{sec:exp_decay_kernel}
We propose the following alternative form of the kernel. This kernel can provide exponential decay,
\begin{align*}
C_\vtheta(x, t) = & 1 + \eta \sum_{k \in \mathbb{Z}, k \neq 0 } q^{\abs{k}}  
e^{ 2 \pi\sqrt{-1} k (x-t)}, \quad \text{with} \quad 0 < q < 1
\end{align*}
where $q$ is used to denote the kernel order to distinguish it from the notation in \eqref{eqn:trunc_series_kernel}. This can be rewritten as
\begin{align*}
C_\vtheta(x, t) = & 1 + \eta \sum_{k \in \mathbb{Z}, k \neq 0 } 
e^{ 2 \pi\sqrt{-1} k (x-t) + \abs{k} \log(q)}
\\
=& 1 + \eta 
\left(
\sum_{k=1}^\infty e^{ 2 \pi\sqrt{-1} k (x-t) + \abs{k} \log(q)} 
+
\sum_{k=\infty}^{-1} e^{ 2 \pi\sqrt{-1} k (x-t) + \abs{k} \log(q)}
\right)
\\
=& 1 + \eta 
\left(
\sum_{k=1}^\infty e^{ 2 \pi\sqrt{-1} k (x-t) + \abs{k} \log(q)} 
+
\sum_{k=-\infty}^{-1} e^{ 2 \pi\sqrt{-1} k (x-t) + \abs{k} \log(q)}
\right)
\\
=& 1 + \eta 
\left(
\underbrace{
	\sum_{k=1}^\infty e^{ 2 \pi\sqrt{-1} k (x-t) + \abs{k} \log(q)} }_{*}
+
\sum_{k=1}^{\infty} e^{ -2 \pi\sqrt{-1} k (x-t) + \abs{k} \log(q)}
\right).
\end{align*}
Let us focus on the first term $(*)$ within the parenthesis in the previous equation,
\begin{align*}
\sum_{k=1}^\infty e^{ 2 \pi\sqrt{-1} k (x-t) + \abs{k} \log(q)} & =
\sum_{k=1}^\infty \left[e^{ 2 \pi\sqrt{-1} (x-t) +  \log(q)} \right]^k
\\
& = \frac{e^{ 2 \pi\sqrt{-1} (x-t) +  \log(q)}}{1- e^{ 2 \pi\sqrt{-1} (x-t) +  \log(q)}}
= \frac{1}{ e^{- 2 \pi\sqrt{-1} (x-t) -  \log(q)} -1 }
\\
& =\frac{1}{ q^{-1} e^{- 2 \pi\sqrt{-1} (x-t)} -1 }
\end{align*}
Using this result
\begin{align*}
C_\vtheta(x, t) &= 
1 + \eta 
\left(
\frac{1}{ q^{-1} e^{- 2 \pi\sqrt{-1} (x-t)} -1 }
+
\frac{1}{ q^{-1} e^{ 2 \pi\sqrt{-1} (x-t)} -1 }
\right)
\\
&= 
1 + \eta 
\left(
\frac{q^{-1} \left(e^{2 \pi\sqrt{-1} (x-t) }+ e^{ -2 \pi\sqrt{-1} (x-t)}\right) -2 }
{q^{-2} - q^{-1} \left(e^{ 2 \pi\sqrt{-1} (x-t)} + e^{ -2 \pi\sqrt{-1} (x-t)}\right) + 1 }
\right)
\\
&= 
1 + \eta 
\left(
\frac{2 q^{-1} \cos({2 \pi\sqrt{-1} (x-t) }) -2 }
{q^{-2} - 2 q^{-1} \cos({ 2 \pi\sqrt{-1} (x-t)})  + 1 }
\right)
\\
&= 
1 + 2 \eta q
\left(
\frac{ \cos({2 \pi\sqrt{-1} (x-t) }) - q }
{q^{2} - 2 q \cos({ 2 \pi\sqrt{-1} (x-t)})  + 1 }
\right).
\end{align*}
Using the fact $\cos^2(t) + \sin^2(t) = 1$,
\begin{align*}
C_\vtheta(x, t) &= 
1 + 2 \eta q
\left(
\frac{ \cos({2 \pi\sqrt{-1} (x-t) }) - q }
{ \left[\cos({ 2 \pi\sqrt{-1} (x-t)})-q\right]^2 + \sin^2({ 2 \pi\sqrt{-1} (x-t)}) }
\right),
\end{align*}
which shows that the kernel order $q$ can be continuously varied while searching for the optimal value. 
The hyperparameters need to be $\eta > 0$ and $ 0 < q < 1$ while searching for the optimum value, so we use the transformations demonstrated in Section~\ref{sec:kernel_param_search} to map the values to or from $\reals$, where the search is usually done.
One disadvantage of this kernel is that it is very sensitive to the changes in kernel order $q \in (0,1)$, for even small values, which might cause the hyperparameter search to miss the global minima.

\Section{Summary}

We summarize the results of this and the previous chapter as a theorem below. 
\begin{theorem}
Let $\mCtheta$ be any symmetric, positive definite, shift-invariant covariance kernel of the form \eqref{eq:shInv}, where $K_\vtheta$ has period one in every variable. Furthermore, let $K_\vtheta$ be scaled to satisfy \eqref{addAssump}. When matched with rank-1 lattice data-sites, $\mCtheta$ must satisfy assumptions \eqref{fastcompAssump}. The cubature, $\hmu$, is just the sample mean. 
The \emph{fast Fourier transform} (FFT) can be used to expedite the estimates of $\vtheta$ in \eqref{thetaSimple} and the credible interval widths \eqref{fastStoppingCriterions} in $\Order(n \log n)$ operations. 

\end{theorem}

Although the third part of the computational cost has the largest dependence on $n$, in practice it need not be the largest contributor to the computational cost.  If function values are the result of an expensive simulation, then the first part may consume most of the computation time.

We have implemented the fast adaptive Bayesian cubature algorithm in MATLAB as part of the Guaranteed Adaptive Integration Library (GAIL) \cite{ChoEtal17b} as \allowbreak \\ \code{cubBayesLattice\_g}. This algorithm uses the kernel defined in  \eqref{the_kernel_eqn_bernoulli} with  $r=1,2$ or the continuous valued order kernel \eqref{eqn:trunc_series_kernel}, and the periodizing variable transforms in \secref{sec:period_var_tx}. The rank-1 lattice node generator is taken from \cite{NuyMagic} (\code{exod2\_base2\_m20}).

\Section{Periodizing Variable Transformations}
\label{sec:period_var_tx}

The shift-invariant covariance kernels underlying our \code{cubBayesLattice\_g} \\
Bayesian cubature  assume that the integrand has a degree of periodicity, with the smoothness assumed depending on the smoothness of the kernel.  
In other-words, non-periodic functions do not live in the space spanned by the shift-invariant covariance kernels.
While integrands arising in practice may be smooth, they might not be periodic.  
Variable transformation or periodization transform techniques are typically used to enforce the periodicity in multi-dimensional numerical integrations where boundary conditions needs to be enforced. These transformations could be either polynomial, exponential and also trigonometric in nature. Some of the most popular transformation are provided here for reference. 

Suppose that the original integral has been expressed as 
\begin{equation*}
\mu := \int_{[0,1]^d} g(\vt) \, \dif \vt,
\end{equation*}
where $g$ has sufficient smoothness, but lacks periodicity.  
The goal is to transform the integral above to the form of \eqref{eqn:defn_mu}, where the integrand $f$---and perhaps its derivatives---are  periodic.  

The Baker's transform, also called tent transform,
\begin{align} \label{eq:bakerTrans}
\vPsi: \vx \mapsto (\Psi(x_1),  \ldots, \Psi(x_d)),  \quad \Psi(x)  =1 - 2 \abs{x - 1/2},
\end{align}
allows us to write $\mu$ in the form of \eqref{eqn:defn_mu}, where $f(\vx) = g(\vPsi(\vx))$.  
Since  $\Psi'(x)$ is not continuous, $f$ does not have continuous derivatives. 

A family of smoother variable transforms that can also preserve continuity of derivatives from the original integrand $g$ takes the form
\begin{subequations} 
	\begin{equation*}
	\vPsi: \vx \mapsto (\Psi(x_1),  \ldots, \Psi(x_d)), \quad \Psi:[0,1] \mapsto [0,1].
	\end{equation*}
	This allows us to write $\mu$ in the form of \eqref{eqn:defn_mu} with
	\begin{equation*}
	f(\vx) = g(\vPsi(\vx)) \prod_{\ell = 1}^d \Psi'(x_l).
	\end{equation*}
\end{subequations}
For $r \in \natzero$, if the following hold:
\begin{itemize}
	\item $\Psi \in C^{r+1}[0,1]$,
	\item  $\lim_{x \downarrow 0}x^{-r-1}\Psi'(x) = \lim_{x \uparrow 1} (1-x)^{-r-1}\Psi'(x) = 0$, and 
	\item $g \in C^{(r, \ldots, r)}[0,1]^d$,
\end{itemize}
then $f$ has continuous, periodic mixed partial derivatives of up to order $r$ in each direction. 
Examples of this kind of transform include \cite{Sid08a}:
\begin{align*}
C^0 &: \Psi(x) =  3 x^2 - 2 x^3, \quad   \Psi'(x) = 6x(1-x), \\
C^1 & : \Psi(x) = x^3(10-15x+6x^2),  \\
&\qquad \qquad \qquad   \Psi'(x) = 30x^2(1-x)^2 \\
\text{Sidi's } C^1 & : \Psi(x) = x - \frac{\sin(2\pi x)}{2 \pi}, \\
&\qquad \qquad \qquad   \Psi'(x) = 1 - \cos(2\pi x), \\
\text{Sidi's } C^2 & : \Psi(x) = \frac {8 - 9 \cos(\pi x) + \cos(3 \pi x)}{16} ,  \\
&\qquad \qquad \Psi'(x) = \frac {3 \pi[3 \sin(\pi x) - \sin(3 \pi x)]}{16}.
\end{align*}

These transforms vary in terms of computational complexity and accuracy and shall be chosen to match the covariance kernel and integrand accordingly. Choosing an optimal periodizing is a topic of future research. Baker's transform is the least complex of all which is a tent map in each coordinate. It preserves only continuity but it is easier to compute and it does not include product term up to the length dimension of the integrand, making it more numerically stable. $C^0$ is a polynomial transformation only and ensures periodicity of function. $C^1$ is a polynomial transformation and preserving the first derivative.
Sidi's $C^1$, a transform which uses trigonometric Sine, preserves the first derivative and is, in general, a better option than $C^1$.
Sidi's $C^2$, also a transform which uses trigonometric Sine, preserves up to second derivative. We use this when smoothness of Sidi's $C^1$ is not sufficient and need to preserve up to second derivative.

Periodizing variable transforms are used in the numerical examples in \secref{sec:NumExp}. In some cases, they can speed the convergence of the Bayesian cubature because they allow one to take advantage of smoother covariance kernels. 
However, there is a trade-off.  Smoother periodizing transformations tend to give integrands $f$ with larger inferred $s$ values and thus wider credible intervals.

\Chapter{Sobol' Nets and Walsh Kernels}
\label{sec:sobol_walsh}

The previous section shows an automatic Bayesian cubature algorithm using rank-1 lattice nodes and shift-invariant kernels. 
In this chapter, we demonstrate a second approach to formulate fast Bayesian transform using matching kernel and point sets. 
Scrambled Sobol' nets and Walsh kernels are paired to achieve $\Order(n^{-1 + \epsilon})$ order error convergence where $n$ is the sample size. 
Sobol' nets \cite{Sob67} are low discrepancy points, used extensively in numerical integration, simulation, and optimization. 
The results of this chapter can be summarized as a theorem,


\begin{theorem}
	Any symmetric, positive definite, digital shift-invariant covariance kernel of the form \eqref{eqn:walsh_kernel} scaled to satisfy \eqref{addAssump}, when matched with digital net data-sites, satisfies assumptions \eqref{fastcompAssump}.  The \emph{fast Walsh-Hadamard transform} (FWHT) can be used to expedite the estimates of $\vtheta$ in \eqref{thetaSimple} and the credible interval widths \eqref{fastStoppingCriterions} in $\Order(n \log n)$ operations. The cubature, $\hmu$, is just the sample mean.
\end{theorem}
We introduce the necessary concepts and prove this theorem in the remaining of this chapter.

\Section{Sobol' Nets}


Nets were developed to provide deterministic sample points for quasi-Monte Carlo rules \cite{Nie05a}. Nets are defined geometrically using elementary intervals, which are subintervals of the unit cube $[0,1)^d$.
The $(t,m, d)$-nets in base $b$, introduced by Niederreiter, 
whose quality is governed by $t$. Lower values of $t$ correspond to $(t,m, d)$-nets of higher quality \cite{Bald10a}.

\begin{defn}
	\label{defn:tmd_net}
	Let $\mathcal{A}$ be the set of all elementary intervals $\mathcal{A} \subset [0, 1)^d$ where
	$\mathcal{A} = \prod_{\ell=1}^d [\alpha_\ell b^{-\gamma_\ell} , (\alpha_\ell + 1) b^{-\gamma_\ell})$, 
	with $d,b,\gamma_\ell \in \naturals, b \ge 2$ 
	and $b^{\gamma_\ell}
	> \alpha_\ell \ge 0$. For $m,t \in \naturals, m \ge t \ge 0$, the point set $\mathcal{P}_m \in [0, 1)^d$ with $n = b^m$ points is a $(t, m, d)$ -- net in base $b$ if every $\mathcal{A}$ with volume $b^{t-m}$ contains $b^t$ points of $\mathcal{P}_m$.
\end{defn}

Digital $(t,m, d)$-nets are a special case of $(t,m, d)$-nets, constructed using matrix-vector multiplications over finite fields. Digital sequences are infinite length digital nets, i.e., the first $n=b^m$ points of a digital sequence comprise a digital net for all integer $m \in \naturals_0$.

\begin{defn}
For any non-negative integer $i = \dots i_3 i_2 i_1(\textup{base} \, b)$, define the $\infty \times 1$ vector $\vec{\imath}$ as the vector of its digits, that is, $\vec{\imath} = (i_1, i_2, \dots)^T$. 
For any point $z = 0.z_1 z_2 \dots (\textup{base}\, b) \in [0, 1)$, define the $\infty \times 1$ vector of the digits of $z$, that is, $\vec{z} = (z_1, z_2, \dots)^T$. 
Let $ \mathsf{G}_1, \dots , \mathsf{G}_d$ denote predetermined $\infty \times \infty$ generator matrices. 
The digital sequence in \textup{base} $b$ is $\{\vz_0, \vz_1, \vz_2, \dots\}$, where each $\vz_i = ( z_{i1}, \dots , z_{id})^T \in [0, 1)^d$ is defined by
\begin{align*}
\vec{z}_{i\ell} = \mathsf{G}_{\ell} \, \vec{\imath}, \quad \ell = 1, \dots, d, \quad i = 0, 1, \dots \;.
\end{align*}
The value of $t$ as mentioned in Definition \ref{defn:tmd_net} depends on the choice of $\mathsf{G}_{\ell}$.
\end{defn}

Digital nets have a group structure under digitwise addition, which is a very useful property exploited in our algorithm, especially to develop a fast Bayesian transform that speedups computations.
Digitwise addition, $\oplus$, and subtraction $\ominus$, are defined in terms of $b$-ary expansions of points in $[0, 1)^d$,
\begin{align*}
\vz \oplus \vy = \left( \sum_{j=1}^\infty [z_{\ell j} + y_{\ell j} \bmod b] b^{-j} \bmod 1 \right)_{\ell=1}^d,
\\
\vz \ominus \vy = \left( \sum_{j=1}^\infty [z_{\ell j} - y_{\ell j} \bmod b] b^{-j} \bmod 1 \right)_{\ell=1}^d,
\end{align*}
where
\begin{align*}
\vz = \left( \sum_{j=1}^{\infty} z_{\ell j}b^{-j}\right)_{\ell=1}^d, \quad
\vy = \left( \sum_{j=1}^{\infty} y_{\ell j}b^{-j}\right)_{\ell=1}^d, \quad
z_{\ell j}, y_{\ell j} \in \{0,\cdots,b-1\}.
\end{align*}

Similarly for integer values in $\naturals_0^d$, the digitwise addition, $\oplus$, and subtraction $\ominus$, are defined in terms of their $b$-ary expansions,
\begin{align*}
\vk \oplus \vl = \left( \sum_{j=0}^\infty [k_{\ell j} + l_{\ell j} \bmod b] b^{j} \bmod 1 \right)_{\ell=1}^d,
\\
\vk \ominus \vl = \left( \sum_{j=0}^\infty [k_{\ell j} - l_{\ell j} \bmod b] b^{j} \bmod 1 \right)_{\ell=1}^d,
\end{align*}
where
\begin{align*}
\vk = \left( \sum_{j=0}^{\infty} k_{\ell j}b^{j}\right)_{\ell=1}^d, \quad
\vl = \left( \sum_{j=0}^{\infty} l_{\ell j}b^{j}\right)_{\ell=1}^d, \quad
\vk_{\ell j}, \vl_{\ell j} \in \{0,\cdots,b-1\}.
\end{align*}

Let $\{\vz_i\}_{i=0}^{b^m-1}$ be a digital net. Then
\begin{align*}
\forall i_1, i_2 \in \{0,\cdots,b^m-1\}, \quad \vz_{j_1} \oplus \vz_{i_2} = \vz_{i_3}, \quad \text{for some} \; i_3 \in \{0,\cdots,b^m-1\}.
\end{align*}

The following very useful result, which will be further used to obtain the fast Bayesian transform, arises from the fundamental property of digital nets.

\begin{lemma}
\label{lemma:digital_net_prop}
Let $\{\vz_i\}_{i=0}^{b^{m}-1}$ be the digital-net and the corresponding digitally shifted net be $\{\vx_i\}_{i=0}^{b^{m}-1}$, i.e.,
\begin{align*}
\vec{x}_{i \ell} = \vec{z}_{i \ell} + \vec{\Delta}_l \bmod 1,
\end{align*}
where $\vec{x}_{i \ell}$ is the $\ell$th component of $i$th digital net and $\vec{\Delta}_{\ell}$ is the digital shift for the $\ell$th component. 
Then,
\begin{align}
\label{eqn:digital_shift_prop}
\vx_i \ominus \vx_j = \vz_i \ominus \vz_j = \vz_{i \ominus j}, \quad \forall i,j \in \naturals_0. 
\end{align}
 Also the digital subtraction is symmetric,
\begin{align}
\label{eqn:digital_net_symmetric_prop}
\vx_i \ominus \vx_i = \boldsymbol{ 0}, \qquad 
\vx_i \ominus \vx_j = \vx_j \ominus \vx_i, \quad \forall i,j \in \naturals_0.
\end{align}
\end{lemma}

\begin{proof}

The proof can be obtained from the definition of digital nets which stated that the digital nets are obtained using generator matrices, $\vec{z}_{i \ell} = \mG_{\ell} \, \vec{\imath} \bmod b$. Rewriting the subtraction using the generating matrix provides the result,
\begin{align*}
\vec{z}_{i \ell} - \vec{z}_{j\ell} \bmod b & = (\mG_{\ell} \vec{\imath} \bmod b) - (\mG_{\ell} \vec{\jmath} \bmod b) \\
& = (\mG_{\ell} \vec{\imath} - \mG_{\ell} \vec{\jmath} ) \bmod b \\
& = \mG_{\ell} (\vec{\imath} - \vec{\jmath} ) \bmod b \\
& = \mG_{\ell} (\overrightarrow{i \ominus j} ) \bmod b \\
& = \vec{z}_{i \ominus j \; {\ell}}.
\end{align*}
The rest of the lemma is obvious from the definition of digital nets.
\end{proof}

We chose digitally shifted and scrambled nets \cite{HicYue00} for our Bayesian cubature algorithm. Digital shifts help to avoid having nodes at the origin, similar to the random shift used with lattice nodes.
Scrambling helps to eliminate bias while retaining the low-discrepancy properties.
A proof that a scrambled net preserves the property of $(t, m, d)$-net almost surely can be found in Owen \cite{Owe95}. The scrambling method proposed by Matou\v{s}ek \cite{Mat98} is preferred since it is more efficient than the Owen's scrambling.

Sobol' nets \cite{Sob76} are a special case of $(t,m, d)$-nets when base $b=2$. 
An example of $64$ Sobol' nets in $d=2$ is given in \figref{fig:sobol-fig}.  The even coverage of the unit cube is ensured by a well chosen generating matrix.  The choice of generating vector is typically done offline by computer search.  See \cite{KuoNuyens2016} and \cite{NuySoft} for more on generating matrices. We use randomly scrambled and digitally shifted Sobol' sequences in this research \cite{HonHic00a}. 

\begin{figure}[htp]
	\label{fig:sobol-fig}
	\centering
	\includegraphics[width=0.8\linewidth]{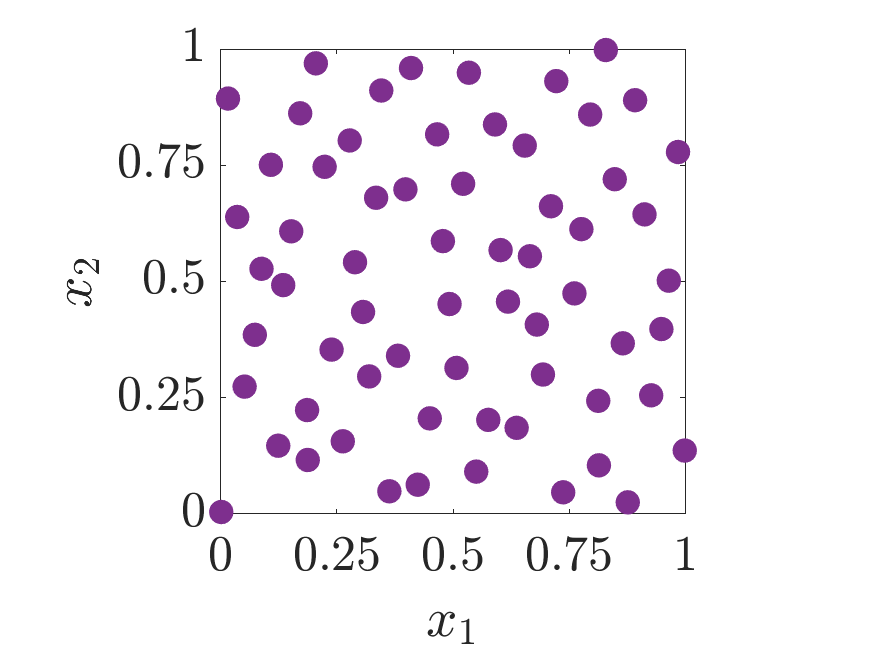}
	\caption{Example of a scrambled Sobol' node set  in $d=2$.  This plot can be reproduced using \code{PlotPoints.m}. \label{sobolfig} }
\end{figure}

\Section{Walsh Kernels}

Walsh kernels are product kernels based on the Walsh functions. We introduce the necessary concepts in this section.

\Subsection{Walsh functions}
Like the Fourier transform used with lattice points (Section~\ref{sec:shift_invar_kern}), the Walsh-Hadamard transform, which we will simply call Walsh transform, is used for the digital nets. The Walsh transform is defined using Walsh functions. Recall $\naturals_0 := \lbrace 0,1,2,\cdots \rbrace$.
The one-dimensional Walsh functions in base $b$ are defined as
\begin{align}
\label{eqn:walsh_func}
\textup{wal}_{b,k}(x) := e^{2\pi \sqrt{-1} (x_1 k_0 + x_2 k_1 + \cdots)/b} 
=
e^{2\pi \sqrt{-1} {\vec{k}}^T{\vec{x}}/b},
\end{align}
for $x \in [0,1)$ and $k \in \naturals_0$ and the unique base $b$ expansions 
$x = \sum_{j \ge 1} x_j b^{-i} = (0.x_1 x_2 \cdots)_b$, $\vec{x} =  (x_1,x_2,\cdots )^T$
$k = \sum_{j \ge 0} k_j b^{j} = ( \cdots k_1 k_0)_b$, $\vec{k} =  (k_0,k_1,\cdots )^T$, and ${\vec{k}}^T{\vec{x}} = x_1 k_0 + x_2 k_1 + \cdots$
where the number of digits used in \eqref{eqn:walsh_func} are limited to the length required to represent $x$ or $k$, i.e., $\max \left( {\ceil{ -\log_b{x}}, \ceil{\log_b{k}}  } \right)$.
Multivariate Walsh functions are defined as the product of the one-dimensional Walsh functions,
\begin{align*}
\textup{wal}_{b,\vk} (\vx) := \prod_{\ell=1}^d \textup{wal}_{b,k_\ell} (x_\ell
)
\end{align*}
As shown in \eqref{eqn:walsh_func}, for the case of $b=2$, the Walsh functions only take the values in $\{1, -1\}$, i.e., $\textup{wal}_{b,\vk} : [0,1)^d \to {\{-1, 1\}} , \; k \in \naturals_0^d$. Walsh functions form an orthonormal basis of the Hilbert space $L^2[0,1)^d$,
\begin{align*}
\int_{[0,1)^d}
\textup{wal}_{b,\vl} (\vx) \textup{wal}_{b,\vk}(\vx) \dx = \delta_{\vl, \vk}, \quad \forall \vl, \vk \in \naturals_0^d
\end{align*}
Digital nets are designed to integrate certain Walsh functions without error.
Thus our Bayesian cubature algorithm integrates linear combinations of 
certain Walsh functions without error. Functions that are well approximated by such linear combinations are then integrated with small errors.

In this research we use Sobol' nodes which are digital nets with base $b=2$. So here afterwards base $b=2$ is assumed. 
In this case, the Walsh function is simply $$\textup{wal}_{2,\vk} (\vx) = (-1)^{\vec{\vk}^T \vec{\vx}}.$$

\Subsection{Walsh kernels}
Consider the covariance kernels of the form,
\begin{align}
\label{eqn:digital_shift_in_kernel}
C_{\vtheta}(\vx, \vt) = K_{\vtheta} (\vx \ominus \vt) 
\end{align}
where $\ominus$ is bitwise subtraction.
This is called a \emph{digitally shift invariant kernel} because shifting both arguments of the covariance function by the same amount leaves the value unchanged. By a proper scaling of the function $K_{\vtheta}$, it follows that assumption \eqref{addAssump} is satisfied. The function $K_{\vtheta}$ must be of the form that ensures that $C_{\vtheta}$ is symmetric and positive definite, as assumed in \eqref{FJH:eq:CondPosDef}. We drop the ${\vtheta}$ sometimes to make the notation simpler.
The Walsh kernels are of the form,
\begin{align}
\label{eqn:walsh_kernel}
K_{\vtheta} (\vx \ominus \vt) =  
\prod_{\ell=1}^d  1 + \eta_\ell \omega_{r} (x_\ell \ominus t_\ell), \quad \veta = (\eta_1, \cdots, \eta_d), \quad \vtheta = (r, \veta)
\end{align}
where $r$ is the kernel order, $\veta$ is the kernel shape parameter, and
\begin{align*}
\omega_r(x) = \sum_{k=1}^\infty 
\frac{\textup{wal}_{2,k}(x) }{2^{2r \lfloor \log_2 k \rfloor}}.
\end{align*}
Explicit expression is available for $\omega_{r}$ in the case of order $r=1$ \cite{Nuyens2013}, 
\begin{align}
\label{eqn:omega1}
\omega_1(x) 
 = 6\left( \frac 16 - 2^{\lfloor \log_2 x \rfloor -1 }\right).
\end{align}

\begin{figure}
	\centering
	\includegraphics[width=0.9\linewidth]{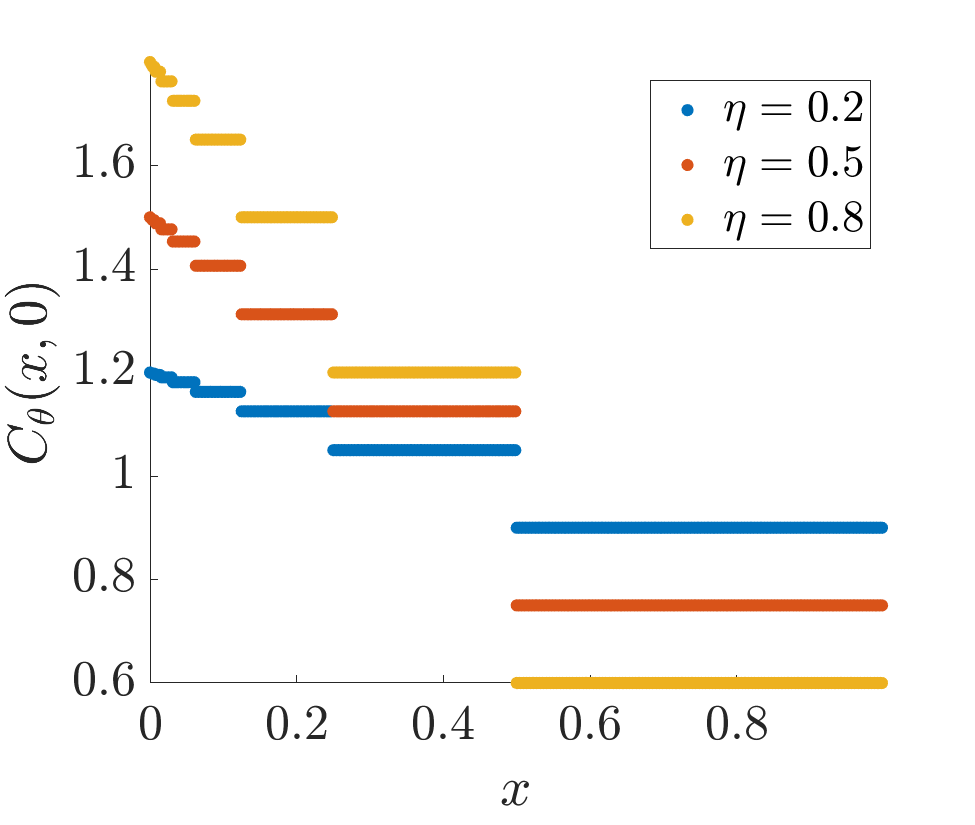}
	\caption[Walsh kernel]{Walsh kernel of order $r=1$ in dimension $d=1$. This figure can be reproduced using \code{plot\_walsh\_kernel.m}. 
	}
	\label{fig:walshkernel-dim1}
\end{figure}

The \figref{fig:walshkernel-dim1} shows the Walsh kernel \eqref{eqn:walsh_kernel} of order $r=1$ in the interval $[0,1)$. Unlike the shift-invariant kernels used with lattice nodes, low order Walsh kernels are discontinuous and are only piecewise constant. Smaller $\eta_\ell$ implies lesser variation in the amplitude of the kernel. Also, the Walsh kernels are digitally shift invariant but not periodic.

\Section{Eigenvectors}

We show the eigenvectors $\mV$ in \eqref{eqn:ftk_factor} of the Gram matrix formed by the covariance kernel \eqref{eqn:walsh_kernel} and Sobol' nets are the columns of the Walsh-Hadamard matrix. First we introduce the necessary concepts.

\Subsection{Walsh transform}
The Walsh-Hadamard transform (WHT) is a generalized class of discrete Fourier transform (DFT) and is much simpler to compute than the DFT. The WHT matrices are comprised of only $\pm 1$ values, so the computation usually involves only ordinary additions and subtractions. Hence, the WHT is also sometimes called the integer transform. In comparison, the DFT that was used with lattice nodes,  uses complex exponential functions and the computation involves complex, non-integer multiplications. 

The WHT involves multiplications by $2^m \times 2^m$ Walsh-Hadamard matrices, which is constructed recursively, starting with $\mH^{(0)} = 1$,
\begin{align}
\nonumber
\arraycolsep=1.4pt\def\arraystretch{0.9}
\mH^{(1)} &=
\begin{pmatrix}
1 & 1 \\ 1 & -1
\end{pmatrix}, \\
\nonumber
\mH^{(2)} &= 
\begin{pmatrix}
1 & 1 & 1 & 1 \\ 
1 & -1 & 1 & -1 \\
1 & 1 & -1 & -1 \\ 
1 & -1 & -1 & 1 \\
\end{pmatrix}, \\
\nonumber
& \qquad \vdots
\\
\label{eqn:hadamard_matrix}
\mH^{(m)} &= 
\begin{pmatrix}
\mH^{(m-1)} & \mH^{(m-1)} \\ \mH^{(m-1)} & -\mH^{(m-1)}
\end{pmatrix} 
= \underbrace{\mH^{(1)} \bigotimes \cdots \bigotimes \mH^{(1)}}_{m \ \text{times}} 
= \mH^{(1)} \bigotimes \mH^{(m-1)}
\end{align}
where $\bigotimes$ is Kronecker product. Alternatively for base $b=2$, these matrices can be  directly obtained by,
\begin{align*}
\mH^{(m)} 
= \bigg((-1)^{(\vec{\imath}^T \vec{\jmath})} \bigg)_{i,j=0}^{2^m-1},
\end{align*}
where the notation $\vec{\imath}^T \vec{\jmath}$ indicates the bitwise dot product. 

\Subsection{Eigenvectors of $\mC$ are columns of Walsh-Hadamard matrix}
\label{sec:eigenvector_hadamard}
The Gram matrix $\mCtheta$ formed by Walsh kernels and Sobol' nodes have a special structure called  block-Toeplitz, which can be used to construct the fast Bayesian transform. 
A Toeplitz matrix is a diagonal-constant matrix in which each descending diagonal from left to right is constant. A block Toeplitz matrix is a special block matrix, which contains blocks that are repeated down the diagonals of the matrix.
We prove that the eigenvectors of $\mCtheta$ are columns of a Walsh-Hadamard matrix in two theorems. 

\begin{theorem}
\label{thrm:block-toeplitz}
Let $\left(\vx_i\right)_{i=0}^{n-1}$ be digitally shifted Sobol' nodes and $K$ be any function,
then the Gram matrix,
\begin{align*}
\mCtheta = \bigl(C(\vx_i, \vx_j)\bigr)_{i,j=0}^{n-1} &= \bigl(K(\vx_i \ominus \vx_j)\bigr)_{i,j=0}^{n-1},   \\ & \text{where} \quad \quad n=2^m, \quad C(\vx, \vt) = K(\vx \ominus \vt), \quad  \vx, \vt \in [0,1)^d, \qquad
\end{align*}
is a $2\times 2$ block-Toeplitz matrix and all the sub-blocks and their sub-sub-blocks, etc. are also $2\times 2$ block-Toeplitz. 
\end{theorem}

\begin{proof}

We prove this theorem by induction. Let $\mC_{\vtheta}^{(m)}$ denote the Gram matrix of size $2^m \times 2^m$.
The relation between sub-block matrices can be deciphered using the properties of digital nets.
To help with the proof of block-Toeplitz structure, consider the digital net properties \eqref{eqn:digital_shift_prop},  \eqref{eqn:digital_net_symmetric_prop},
and notations,
\begin{align*}
\mK^{(m)} &:= 
\begin{pmatrix}
K({\vz_{i} \ominus \vz_{j}})
\end{pmatrix}_{i,j=0}^{2^{m}-1}
=
\begin{pmatrix}
K({\vz_{i \ominus j}})
\end{pmatrix}_{i,j=0}^{2^{m}-1}, \quad m=1,2,\cdots,
\\
\mK^{(m,q)} &:= 
\begin{pmatrix}
K({\vz_{i \ominus j + q 2^m}})
\end{pmatrix}_{i,j=0}^{2^{m}-1},
\quad 
 q = 0,1,\cdots .
\end{align*} 
These two notations are related by $\mK^{(m)} = \mK^{(m,0)}$. 
Please note that $\mC_{\vtheta}^{(m)} = \mK^{(m,0)}$.
We will prove $\mK^{(m,q)}$ is a $2\times 2$ block-toeplitz matrix for all $m \in \naturals, q \in \naturals$.

\iftrue
As the first step, we verify the property holds for $m=1$,
\begin{align*}
\mK^{(1, q)} &= \begin{pmatrix}
K(\vz_{0 \ominus 0 + q 2^1}) & K(\vz_{1 \ominus 0 + q 2^1})  \\
K(\vz_{0 \ominus 1 + q 2^1}) & K(\vz_{1 \ominus 1 + q 2^1})  
\end{pmatrix} = 
\begin{pmatrix}
K(\vz_{2q}) & K(\vz_{1 + 2q}) \\ K(\vz_{1 + 2q}) & K(\vz_{2q})
\end{pmatrix}, \quad \text{by \eqref{eqn:digital_shift_prop}}
\end{align*} has diagonal elements repeated. Thus by definition, it is a $2\times 2$ block-Toeplitz.
\fi 

Now assume that $\mK^{(m,q)}$ is block-Toeplitz.
We need to prove $\mK^{(m+1, q)}$ is also a $2\times 2$ block-Toeplitz. Let $n=2^m$,
\begin{align*}
\mK^{(m+1)} &= 
\begin{pmatrix}
K(\vz_{0    \ominus 0}) & \hdots & K(\vz_{0    \ominus n-1}) & K(\vz_{0    \ominus n}) & \hdots & K(\vz_{0    \ominus 2n-1}) \\
\vdots             & \vdots &             \vdots          &           \vdots      & \vdots &             \vdots         \\
K(\vz_{n-1  \ominus 0}) & \hdots & K(\vz_{n-1  \ominus n-1}) & K(\vz_{n-1  \ominus n}) & \hdots & K(\vz_{n-1  \ominus 2n-1}) \\
K(\vz_{n    \ominus 0}) & \hdots & K(\vz_{n    \ominus n-1}) & K(\vz_{n    \ominus n}) & \hdots & K(\vz_{n    \ominus 2n-1}) \\
\vdots      & \vdots &             \vdots        &             \vdots      & \vdots &             \vdots         \\
K(\vz_{2n-1 \ominus 0}) & \hdots & K(\vz_{2n-1 \ominus n-1}) & K(\vz_{2n-1 \ominus n}) & \hdots & K(\vz_{2n-1 \ominus 2n-1}) 
\end{pmatrix} 
\\
& = 
\begin{pmatrix}
\begin{pmatrix}
K(\vz_{  0   }) & \hdots & K(\vz_{ n-1}) \\
\vdots          & \vdots &    \vdots     \\
K(\vz_{ n-1  }) & \hdots & K(\vz_{ 0 })
\end{pmatrix}
& 
\begin{pmatrix}
K(\vz_{ n})     & \hdots & K(\vz_{ 2n-1}) \\
\vdots          & \vdots &     \vdots     \\
K(\vz_{ 2n-1 }) & \hdots & K(\vz_{ n })   
\end{pmatrix}
\\
\begin{pmatrix}
K(\vz_{ n})     & \hdots & K(\vz_{ 2n-1}) \\
\vdots          & \vdots &     \vdots     \\
K(\vz_{ 2n-1 }) & \hdots & K(\vz_{ n })   
\end{pmatrix}
&
\begin{pmatrix}
K(\vz_{  0   }) & \hdots & K(\vz_{ n-1}) \\
\vdots          & \vdots &    \vdots     \\
K(\vz_{ n-1  }) & \hdots & K(\vz_{ 0 })
\end{pmatrix}
\end{pmatrix} 
\\
& =
\begin{pmatrix}
\mK^{(m)} & \mK^{(m,1)} \\ \mK^{(m,1)} & \mK^{(m)}
\end{pmatrix}
\end{align*}
is a $2\times 2$ block-Toeplitz, where we used the properties  \eqref{eqn:digital_shift_prop}, \eqref{eqn:digital_net_symmetric_prop} and facts $2n-1 \ominus n = n-1$, $2n-1 \ominus n-1 = n$, and $n \ominus n-1 = 2n-1$. 
Thus $\mK^{(m+1)}$ is a $2\times2$ block-Toeplitz. 
Similarly 
\begin{align*}
\mK^{(m+1,q)} 
& =
\begin{pmatrix}
\mK^{(m,q)} & \mK^{(m,q+1)} \\ \mK^{(m,q+1)} & \mK^{(m,q)}
\end{pmatrix}
\end{align*}
 is a $2\times2$ block-Toeplitz. 
Thus $\mC_{\vtheta}^{(m)}$ of size $2^m\times 2^m$, for $m \in \naturals$, is a $2\times2$ block-Toeplitz and every block and it's sub-blocks of size $2^p, \; p \in \naturals, \; p \le m$ are also $2\times2$ block-Toeplitz.
\end{proof}

\begin{theorem}
\label{thrm:hadamard_eigenvector}
The Walsh-Hadamard matrix $\mH^{(m)}$ factorizes $\mC_{\vtheta}^{(m)}$, so that the columns of Walsh-Hadamard matrix are the eigenvectors of $\mC_{\vtheta}^{(m)}$, i.e.,
\begin{align*}
\mH^{(m)} \mC_{\vtheta}^{(m)}  = \mLambda^{(m)} \mH^{(m)}, \quad m \in \naturals. 
\end{align*}
\end{theorem}

\begin{proof}
	
Again, we use the proof-by-induction technique to show that the Walsh-Hadamard matrix factorizes $\mK^{(m,q)}$.
We can easily see the Hadamard matrix $\mH^{(1)}$ diagonalizes $\mK^{(1,q)}$,
\begin{align*}
\mH^{(1)} \mK^{(1,q)} &= 
\begin{pmatrix}
1 &  1 \\ 1 & -1
\end{pmatrix}
\begin{pmatrix}
K(\vz_{0 + q2^1}) & K(\vz_{1 + q2^1}) \\ K(\vz_{1 + q2^1}) & K(\vz_{0 + q2^1})
\end{pmatrix}, \quad \text{by Theorem \ref{thrm:block-toeplitz}} 
\\
& = \begin{pmatrix} K(\vz_{2q})+K(\vz_{2q+1}) & K(\vz_{2q})+K(\vz_{2q+1}) \\ K(\vz_{2q})-K(\vz_{2q+1}) & K(\vz_{2q+1})-K(\vz_{2q}) \end{pmatrix} \\ 
& = \begin{pmatrix} K(\vz_{2q})+K(\vz_{2q+1}) & 0 \\ 0 & K(\vz_{2q})-K(\vz_{2q+1}) \end{pmatrix} 
\begin{pmatrix}
1 &  1 \\ 1 & -1
\end{pmatrix} \\
&= \mLambda^{(1,q)} \mH^{(1)},
\end{align*}
where $\mLambda^{(1,q)}$ is a diagonal matrix, thus $\mH^{(1)}$ factorizes $\mK^{(1,q)}$.

Now assume $\mH^{(m)}$ factorizes $\mK^{(m,q)}$, so $\mH^{(m)} \mK^{(m,q)} = \mLambda^{(m,q)} \mH^{(m)}$ where $\mLambda^{(m,q)}$ is diagonal. We need to prove $\mH^{(m+1)}$ factorizes $\mK^{(m+1,q)}$,
\begin{align*}
\mH^{(m+1)} \mK^{(m+1,q)} &= 
\begin{pmatrix}
\mH^{(m)} & \mH^{(m)} \\ \mH^{(m)} & -\mH^{(m)}
\end{pmatrix}
\begin{pmatrix}
\mK^{(m,q)} & \mK^{(m,q+1)} \\ \mK^{(m,q+1)} & \mK^{(m,q)}
\end{pmatrix}, \quad \text{by Theorem \ref{thrm:block-toeplitz}}
\\
& = \begin{pmatrix}
\mH^{(m)} (\mK^{(m,q)}   + \mK^{(m,q+1)}) & 
\mH^{(m)} (\mK^{(m,q)}   + \mK^{(m,q+1)}) \\ 
\mH^{(m)} (\mK^{(m,q)}   - \mK^{(m,q+1)}) & 
\mH^{(m)} (\mK^{(m,q+1)} - \mK^{(m,q)}) 
\end{pmatrix} \\ 
& = \begin{pmatrix}
(\mLambda^{(m,q)}  + \mLambda^{(m,q+1)})  \mH^{(m)} & 
(\mLambda^{(m,q)}   + \mLambda^{(m,q+1)}) \mH^{(m)} \\ 
(\mLambda^{(m,q)}   - \mLambda^{(m,q+1)}) \mH^{(m)} & 
(\mLambda^{(m,q+1)} - \mLambda^{(m,q)})   \mH^{(m)}
\end{pmatrix} \\ 
& = 
\begin{pmatrix}
\mLambda^{(m,q)} + \mLambda^{(m,q+1)} & 0 \\ 0 & \mLambda^{(m,q)} - \mLambda^{(m,q+1)}
\end{pmatrix}
\begin{pmatrix}
\mH^{(m)} & \mH^{(m)} \\ \mH^{(m)} & -\mH^{(m)}
\end{pmatrix}
\\ 
& = \mLambda^{(m+1,q)} \mH^{(m+1)} .
\end{align*}
Thus, $\mH^{(m+1)}$ factorizes $\mK^{(m+1,q)}$ to a diagonal matrix $\mLambda^{(m+1,q)}$. This implies $\mH^{(p)}$ factorizes $\mC_\vtheta^{(p)}$ for $p \in \naturals$. Please recall $\mC_\vtheta^{(p)} = \mK^{(p,0)}$.  Here we used the fact that both $\mH$ and $\mK$ are symmetric positive definite. 
\end{proof}

\Subsection{Fast Bayesian transform}
We can easily show that the Walsh-Hadamard matrices satisfy the assumptions of fast Bayesian transform \eqref{fastcompAssump}. As shown in Section~\ref{sec:eigenvector_hadamard} the columns of $\mH^{({m})}$ are the eigenvectors. Since the Gram matrix $\mC$ is symmetric, the columns/rows of Walsh-Hadamard matrices are mutually orthogonal. Thus the Gram matrix can be written as 
\begin{align}
\label{eqn:hadamard_fwht}
\mC^{(m)} = \frac{1}{n} \mH^{(m)} \mLambda^{(m)} \mH^{(m)}, \quad \text{where} \quad \mH^{({m})} = \underbrace{ \mH^{(1)} \bigotimes \cdots \bigotimes \mH^{(1)} }_{m \; \text{times}}.
\end{align}
Assumption \eqref{fastcompAssumpB} follows automatically by the fact that Walsh-Hadamard matrices can be constructed analytically. Assumption \eqref{fastcompAssumpA} can also be verified as the first row/column are one vectors. Finally, assumption \eqref{fastcompAssumpC} is satisfied due to the fact that fast Walsh transform can be computed in $\Order({n \log n})$ operations using fast Walsh-Hadamard transform.
Thus the Walsh-Hadamard transform is a fast Bayesian transform, $\mV := \mH$, as per \eqref{fastcompAssump}.

We have implemented a fast adaptive Bayesian cubature algorithm using the kernel \eqref{eqn:walsh_kernel} with $r=1$ and Sobol' points \cite{BraFox88} in MATLAB as part of the Guaranteed Adaptive Integration Library (GAIL) \cite{ChoEtal17b} as \allowbreak \code{cubBayesNet\_g}. The Sobol' points used in this algorithm are generated using MATLAB's builtin function \code{sobolset} and scrambled using MATLAB function \code{scramble} \cite{HonHic00a}. The fast Walsh-Hadamard transform \eqref{eqn:hadamard_fwht} is computed using MATLAB's builtin function \code{fwht} with \emph{hadamard} ordering. 

\Subsection{Iterative Computation of Walsh Transform}
In every iteration of our algorithm, we double the number of function values. Using the technique described here, we have to only compute the Walsh transform for the newly added function values.
Similar to the lattice points, Sobol' points are extensible by definition. This property is used in our algorithm to improve the integration accuracy till the required error tolerance is met. Sobol' nodes can be combined with Hadamard matrices as demonstrated here for iterative computation. 
Let $\widetilde{\vy} = \mH^{(m+1)} {\vy}$ for some arbitrary $\vy \in \reals^{2n}$, $n = 2^m$. Define, 
\begin{gather*}
\vy = \begin{pmatrix}[1.1] y_1 \\ \vdots \\ y_{2n} \end{pmatrix}, \quad 
\vy^{(1)} = \begin{pmatrix}[1.1] y_1 \\ \vdots \\ y_{n} \end{pmatrix}, \quad 
\vy^{(2)}  = \begin{pmatrix}[1.1] y_{n+1} \\ \vdots \\ y_{2n} \end{pmatrix}, \\ 
\widetilde{\vy}^{(1)} = \mH^{(m)} \vy^{(1)} = 
\begin{pmatrix}[1.0] \widetilde{y}^{(1)}_1 \\ \widetilde{y}^{(1)}_2 \\ \vdots \\ \widetilde{y}^{(1)}_{n} \end{pmatrix}, \quad 
\widetilde{\vy}^{(2)}  =  \mH^{(m)} \vy^{(2)} =
\begin{pmatrix}[1.0] \widetilde{y}^{(2)}_{1} \\  \widetilde{y}^{(2)}_{2} \\ \vdots \\ \widetilde{y}^{(2)}_{n} \end{pmatrix}. 
\end{gather*}
Then,
\begin{align*}
\widetilde{\vy} &= \mH^{({m+1})} {\vy} \\
& = \begin{pmatrix}
\mH^{(m)} & \mH^{(m)} \\ \mH^{(m)} & - \mH^{(m)}
\end{pmatrix} 
\begin{pmatrix}
\vy^{(1)} \\ \vy^{(2)}
\end{pmatrix}, \qquad \text{by \eqref{eqn:hadamard_matrix}} \\
&= 
\begin{pmatrix}
\mH^{(m)} \vy^{(1)} + \mH^{(m)} \vy^{(2)} \\ 
\mH^{(m)} \vy^{(1)} - \mH^{(m)} \vy^{(2)}
\end{pmatrix}\\
&= 
\begin{pmatrix}
\widetilde{\vy}^{(1)} + \widetilde {\vy}^{(2)} \\ 
\widetilde {\vy}^{(1)} - \widetilde {\vy}^{(2)}
\end{pmatrix} =: \widetilde{\vy} \quad.
\end{align*}
As before with the lattice nodes, the computational cost to compute $\mV^{(m+1)H} \vy$ is 
twice the cost of computing $\mV^{(m)H} \vy^{(1)}$ plus $2n$ additions, where $n=2^m$. An inductive argument shows that for any $m \in \naturals$, $\mV^{(m)H}\vy$ requires only $\Order(n \log n)$ operations. Usually the multiplications in $\mV^{(m)H} \vy^{(1)}$ are multiplications by $-1$ which are simply accomplished using sign change or negation, requiring no multiplications at all.

\Section{Higher Order Nets}

Higher order digital nets are an extension of $(t,m,d)$-nets, introduced in \cite{Dic08a}. They can be used to numerically integrate smoother functions which are not necessarily periodic, but have square integrable mixed partial derivatives of order $\alpha$, at a rate of $\Order(n^{-\alpha})$ multiplied by a power of a $\log n$ factor using rules corresponding to the modified $(t,m, d)$-nets.
We want to emphasize that quasi-Monte Carlo rules based on these point sets can achieve convergence rates faster than $\Order(n^{-1})$. 
Higher order digital nets are constructed using matrix-vector multiplications over finite
fields. 

One could develop matching digitally shift invariant kernels to formulate the fast Bayesian cubature. Bayesian cubatures using higher order digital nets are a topic for future research.

\clearpage

\Chapter{Numerical Implementation} 
\label{sec:NumImpl}

\Section{Overcoming Cancellation Error}
\label{sec:overcome_cancel_error}

We now refer back to general setting for the fast automatic Bayesian cubature in Section~\ref{sec:fast_BC}.
For the covariance kernels used in our computation, it often happens that $n/\lambda_1$ is close to $1$, especially for larger $n$.  Thus, the term $1-n/\lambda_1$, which appears in the credible interval widths, $\err_{\MLE}$, $\err_{\textup{full}}$, and $\err_{\textup{GCV}}$  \eqref{errSimple}, may suffer from cancellation error.  We can avoid this cancellation error by modifying how we compute the Gram matrix and its eigenvalues.

Any shift-invariant or digital shift-invariant covariance kernel satisfying \eqref{addAssump} can be written as $C_\vtheta = 1 + \rC_\vtheta$, where $\rC_\vtheta$ is also symmetric and positive definite.
The associated Gram matrix for $\rC_\vtheta$ is then $\rmC_\vtheta = \mC_\vtheta - \vone \vone^T$, and the eigenvalues of $\rmC_\vtheta$ are $\rlambda_1 = \lambda_1 - n, \lambda_2, \ldots, \lambda_n$, which follows because $\vone$ is the first eigenvector of both $\mC_\vtheta$ and $\rmC_\vtheta$. Note that $\rC_\vtheta$ inherits the shift-invariant properties of $C_\vtheta$. Then,
\begin{equation*}
1 - \frac{n}{\lambda_1}  = \frac{\lambda_1 - n}{\lambda_1} = \frac{\rlambda_1}{\rlambda_1 +n},
\end{equation*}
where now the right hand side is free of cancellation error.

We show how to compute $\rC_\vtheta$ without introducing round-off error.  The covariance functions that we use in both Chapter~\ref{sec:shift_invariant_kernel} and \ref{sec:sobol_walsh} are of product form, namely,
\begin{equation*}
C_\vtheta(\vt, \vx) = \prod_{\ell=1}^d \left[1 + \rC_{_\vtheta,\ell}(t_\ell,x_\ell) \right], \qquad  \rC_{_\vtheta,\ell}:[0,1] \times [0,1] \to \reals.
\end{equation*}
Direct computation of $\rC_\vtheta (\vt,\vx) = C_\vtheta(\vt,\vx) -1$ introduces cancellation error if the $ \rC_\ell$ are small.  So, we employ the iteration,
\begin{align*}
\rC_\vtheta^{(1)}(\vt,\vx) &= \rC_{_\vtheta,1}(t_1,x_1),  \\
\rC_\vtheta^{(\ell)}(\vt,\vx) &  = \rC_\vtheta^{(\ell-1)}[1 + \rC_{\vtheta,\ell}(t_\ell,x_\ell)] + \rC_{\vtheta,\ell}(t_\ell,x_\ell),  \hspace{1cm} \ell = 2, \ldots, d, \\
\rC_\vtheta(\vt,\vx)  & = \rC_\vtheta^{(d)}(\vt,\vx).
\end{align*}
In this way, the Gram matrix $\rmC_\vtheta$, whose $i,j$-element is $\rC_\vtheta(\vx_i,\vx_j)$ can be constructed with minimal round-off error because we avoid subtraction.

Computing the eigenvalues of $\rmC_\vtheta$ via the procedure given in \eqref{eqn:fast_transform_to_eigvalues} yields $\rlambda_1 = \lambda_1 - n, \lambda_2, \ldots, \lambda_n$. The estimates of $\vtheta$ are computed in terms of the eigenvalues of $\rmC_\vtheta$.  So \eqref{eqn_MLE_loss_func_optimized_2} and \eqref{thetaGCV} become
\begin{subequations}
	\label{thetaSimple}
	\begin{align}
	\vtheta_\MLE
	&= 
	\argmin_{\vtheta}
	\left[
	\log\left(
	\sum_{i=2}^n \frac{\abs{\widetilde{y}_i}^2}{\lambda_i}
	\right) 
	+ 
	\frac{1}{n}\sum_{i=1}^n \log(\lambda_i)
	\right], \\
	\vtheta_{\GCV} 
	&= \argmin_\vtheta \left[ \log \left ( \sum_{i=2}^n \frac{\abs{\widetilde{y}_i}^2}{\lambda_i^2} 
	\right)  -2\log\left( \sum_{i=1}^n \frac{1}{\lambda_i} \right)
	\right],
	\end{align}
\end{subequations}
where $\lambda_1 = n + \rlambda_1$.  The widths of the credible intervals in \eqref{eq:errMLEAllAsump}, \eqref{FJH:eq:errFullSimple}, and   \eqref{errGCVSimple} become,
\begin{subequations}
	\label{fastStoppingCriterions}
	\begin{align}
	\label{fastStoppingCriterionMLE}
	\err_\MLE  &
	=
	\frac{2.58}{n}\sqrt{
		\frac{\rlambda_1}{\lambda_1}
		\sum_{i=2}^{n} \frac{\abs{\widetilde{y}_i}^2}{\lambda_i}  
	}, 
	\\
	\label{fastStoppingCriterionFull}
	\err_{\textup{full}} 
	& = \frac{t_{n-1,0.995}}{n} \sqrt{
		\frac{\rlambda_1}{n-1} \sum_{i=2}^n \frac{\abs{\widetilde{y}_i}^2}{\lambda_i}
	}, \\
	\label{fastStoppingCriterionGCV}
	\err_{\textup{GCV}} & =
	\frac{2.58}{n}\sqrt{	\frac{\rlambda_1}{\lambda_1} \sum_{i=2}^n \frac{\abs{\widetilde{y}_i}^2}{\lambda_i^2}  \left [ \frac 1n \sum_{i=1}^n \frac{1}{\lambda_i} \right]^{-1}} .
	\end{align}
\end{subequations}
Since $\rlambda_1 = \lambda_1 - n$ and $\lambda_1 \sim n$ it follows $\rlambda_1/\lambda_1 \approx \rlambda_1/(n-1)$ and is small for  large $n$.  Moreover, for large $n$, the credible intervals via empirical Bayes and full Bayes are similar, since $t_{n-1,0.995}$ is approximately $2.58$.

\begin{algorithm}
	\caption{Fast Automatic Bayesian Cubature}\label{algorithm2}
	\begin{algorithmic}[1]
		\Require 
		A choice of generator for the point set $\vx_1, \vx_2, \ldots$ and a matching kernel $C_\vtheta$ from, 
		1) rank-1 Lattice points and a matching shift-invariant  kernel, 2) Sobol' sequence and a matching digital shift-invariant kernel;
		a black-box function, $f$; 
		an absolute error tolerance,
		$\varepsilon>0$; the positive initial sample size, $n_0$, that is a power of $2$;
		the maximum sample size $n_{\textup{max}}$
		
		\State $n \gets n_0, \; n' \gets 0, \; \err_{\CI} \gets \infty$
		
		\While{$\err_{\CI} > \varepsilon$ and $n \le n_{\textup{max}}$}
		
		\State\label{LoopStartA2}Generate $\{ \vx_i\}_{i=n' + 1}^{n}$ and sample $\{f(\vx_i)\}_{i=n'+1}^{n}$
		\State Compute $\vtheta$ by \eqref{eqn_MLE_loss_func_optimized_2} or \eqref{thetaGCV} by using the techniques from Chapter~\ref{sec:shift_invariant_kernel} or \ref{sec:sobol_walsh}
		\State Compute $\err_{\CI}$  according to \eqref{fastStoppingCriterionMLE}, \eqref{fastStoppingCriterionFull}, or \eqref{fastStoppingCriterionGCV}  by using the techniques from Chapter~\ref{sec:shift_invariant_kernel} or \ref{sec:sobol_walsh}
		
		\State	$n' \gets n, \; n \gets 2n'$
		
		\EndWhile
		
		\State Update sample size to compute $\hmu$, $n \gets n'$
		\State Compute $\hmu$, the approximate integral,   according to \eqref{muhatGCV-FB-MLE-Simple}
		\State \Return $\hmu, \; n$  and $\err_{\CI}$
	\end{algorithmic}
\end{algorithm}

The computational steps for the improved, faster, automatic Bayesian cubature are detailed in Algorithm \ref{algorithm2}.
In comparison to Algorithm \ref{algorithm1}, the second and third components of the computational cost of Algorithm  \ref{algorithm2} are substantially reduced.
The Algorithm \ref{algorithm2} has a computational cost which is the sum of the following:
\begin{itemize}
	\item $\Order\bigl(n\$(f) \bigr)$ for the integrand data, where $\$(f)$ is the computational cost of a single $f(\vx)$
	
	\item $\Order\bigl(N_{\opt} n \$(C_\vtheta) \bigr)$ for the evaluations of the vector $\vC_{1}$, where $N_\opt$ is the number of optimization steps required, and  $\$(C_\vtheta)$ is the computational cost of a single $C_\vtheta(\vt,\vx)$
	
	\item $\Order\bigl(N_{\opt} n \log(n) \bigr)$ for the FFT calculations; there is no $d$ dependence in these calculations
\end{itemize}

\Section{Kernel Hyperparameters Search}
\label{sec:kernel_param_search}

\JRNote{Explain the transformation used to make the search range positive, $> 0$ , etc.}

The various hyperparameters introduced and used by our algorithms need to be optimally chosen. 
The parameter search can be done in two major ways. 
Bounded minima search, if the search interval is known, else unbounded search.  Most of the scenarios, the search interval is unknown. So the natural choice is to use unbounded search over the unbound domain such as \code{fminsearch} provided by MATLAB.
However hyperparameters need to live in a domain that is bounded or semi-bounded. 
There are some simple domain transformations available to achieve this.

\Subsection{Positive kernel shape parameter}
The following parameter map is used to ensure that the shape parameter values are positive real numbers. For $\eta > 0$ as introduced in Section \ref{sec:deriv_of_kernel}, let
\begin{align*}
\eta{(t_1)} = & \quad e^{t_1}, & \eta : (-\infty, \infty) \to (0, \infty).
\end{align*}
Instead of searching for $\eta \in (0, \infty)$, we may search for the optimal $t_1 = \log(\eta)$ over the whole real line $\reals$. 
The optimal value $t_{1, \textup{opt}}$ can be transformed back to the $(0, \infty)$ interval using 
\begin{align*}
\eta_{\textup{opt}} = e^{t_{1,\textup{opt}}}.
\end{align*}

\Subsection{Kernel order $1 < r < \infty$}
The following map is used to ensure that the kernel order values are positive real number and greater than one, i.e., in the $(1, \infty)$ interval as required in \secref{sec:trunc_series_kernel},
\begin{align*}
r(t_2) = & \quad {1 + e^{t_2}}, & r:  (-\infty, \infty) \to (1,\infty).
\end{align*}
So one may search for the optimal $t_2 = -\log(r-1)$ in the whole real line $\reals$.
The optimal value $t_{2, \textup{opt}}$ can be transformed back to the desired interval $(0,1)$ using 
\begin{align*}
r_{\textup{opt}} = 1+ e^{t_{2, \textup{opt}}}.
\end{align*}

\Subsection{Kernel order $0 < q < 1$}
The following multivariate map is used to ensure that the kernel order values are positive real and less than one, i.e., in the $(0,1)$ interval to use with exponentially decaying kernel, as introduced in \secref{sec:exp_decay_kernel},
\begin{align*}
q(t_3) = & \quad \frac{1}{1 + e^{t_3}}, & q: (-\infty, \infty) \to (0, 1).
\end{align*}
So one may search for the optimal $t_3 = \log(q^{-1}-1)$ in the whole real line $\reals$.
The optimal value $t_{3, \textup{opt}}$ can be transformed back to the desired interval $(0,1)$ by using
\begin{align*}
q_{\textup{opt}} = \frac{1}{1 + e^{t_{3, \textup{opt}}}}.
\end{align*}

\Subsection{Combined searching of kernel order $r$ and shape parameter $\veta$}
Instead of searching $\veta$ and $r$ separately one would prefer to search them together so that the most optimal values can be obtained, where $\veta = (\eta_1, \cdots, \eta_d)$ such that $\eta_l \ne \eta_k$,  for $l \ne k$.
We can combine the parameter maps used above to ensure that the kernel order values in $(1, \infty)$ and shape parameter $\veta$ in $(0,\infty)^d$ as required in \secref{sec:trunc_series_kernel},
\begin{align*}
\vtheta(\vt) = 
\begin{pmatrix}[1]
r(t_{1}) \\ \eta(t_2) \\ \vdots \\ \eta(t_{d+1}) 
\end{pmatrix} =
\begin{pmatrix}[1]
1 + e^{t_1} \\ e^{t_2} \\ \vdots \\ e^{t_{d+1}}
\end{pmatrix}, 
\quad
\vtheta: \reals^{d+1} \to (1,\infty) \times (0,\infty)^d .
\end{align*}
So instead of searching for $\vtheta_{\textup{opt}}$ in $(1,\infty) \times (0,\infty)^d$, one may search for the optimal 
$$
\vt = \vt^{-1} (\vtheta) = 
\begin{pmatrix}[1]
\log(r-1) \\ \log(\eta_1) \\ \vdots \\ \log(\eta_d)
\end{pmatrix}
$$ in the whole real line $\reals^{d+1}$.
The optimal value $\vt_{\textup{opt}}$ can be transformed back to the desired interval $(1,\infty) \times (0,\infty)^d$ using 
\begin{align*}
\vtheta_{\textup{opt}} = 
\begin{pmatrix}[1]
1 + e^{t_{1, \textup{opt}}} \\
e^{t_{2, \textup{opt}}} \\
\vdots \\
e^{t_{d+1, \textup{opt}}} \\
\end{pmatrix}.
\end{align*}
Similarly one can map the kernel order $q \in (0,1)$ \secref{sec:exp_decay_kernel}, and $\veta$ in to a multivariate hyperparameter search.

\clearpage

\Chapter{Numerical Results and Observations} 
\label{sec:NumExp}

\JRNote{use uniformly randomly chosen $\varepsilon$ instead 4 fixed}

Fast Bayesian cubature algorithms developed in this research are demonstrated using three commonly used integration examples.
These integrals were evaluated using both the algorithms \code{cubBayesLattice\_g} and \code{cubBayesNet\_g}. The first example shows evaluating a multivariate Gaussian probability given the interval. The second example shows integrating the Keister's function, and the final example shows computing an Asian arithmetic option pricing.  

\Section{Testing Methodology}
\label{sec:numerical_experiments_cubBayesLattice}

Four hundred different error tolerances, $\varepsilon$, were randomly chosen from a fixed interval  for each example. 
The intervals for error tolerance were chosen depending on the difficulty of the problem.
The nodes used in \allowbreak \code{cubBayesLattice\_g} were the randomly shifted lattice points supplied by GAIL, whereas the nodes used in \code{cubBayesNet\_g} were the randomly scrambled and shifted Sobol' points supplied by MATLAB's Sobol' sequence generator. 

For each integral example, and each stopping criteria---empirical Bayes, full Bayes, and generalized cross-validation---our algorithm is run with each randomly chosen error tolerance as mentioned above.  For each test, the execution time is plotted against $\abs{\mu - \hmu}/\varepsilon$.  We expect $\abs{\mu - \hmu}/\varepsilon$ to be no greater than one, but hope that it is not too much smaller than one, which would indicate a stopping criterion that is too conservative.

Periodization variable transforms are used in the examples with \\ \allowbreak \code{cubBayesLattice\_g}, which assumes the integrands to be periodic in $[0,1]^d$. But the \allowbreak \code{cubBayesNet\_g} does not need this additional requirement, so the integrands are used directly.



\Section{Multivariate Gaussian Probability}

This example is introduced in Section~\ref{MVN_example}, where we use the Mat\'ern covariance kernel.  We reuse $f_{\textup{Genz}}$ \eqref{eqn:fGenzdef} and apply a periodization transform to obtain $f_{\textup{GenzP}}$ when required.

\Subsection{Using \code{cubBayesLattice\_g}}
As required by the algorithm, we apply Sidi's $C^2$  periodization to $ f_{\textup{Genz}}$ \eqref{eqn:fGenzdef}, and chose $d=3$ and $r=2$. The simulation results for this example integrand  are summarized in Figures \ref{fig:mvn-guaranteed-MLE}, \ref{fig:mvn-guaranteed-FB}, and \ref{fig:mvn-guaranteed-GCV}.  In all cases, $\code{cubBayesLattice\_g}$ returns an approximation within the prescribed error tolerance. We used the same setting as before with generic slow Bayesian cubature in \secref{MVN_example} for comparision. For error threshold $\varepsilon=10^{-5}$ with empirical stopping criterion, our fast algorithm takes 0.001 seconds as shown in \figref{fig:mvn-guaranteed-MLE} whereas the basic algorithm takes 30 seconds as shown in \figref{fig:MVN_Metern_d2b2}. 
Amongst the three stopping criteria, GCV achieved the results faster than others but it is less conservative. 
One can also observe from the figures that the credible intervals are wider, causing true error much smaller than requested.
This could be due to the periodization transformed integrand, $f_\textup{GenzP}$, being smoother than the $r=2$ kernel approximation. Using a kernel of matching smoothness could produce right credible intervals.

\begin{figure}
	\centering
	\includegraphics[width=0.95\linewidth]{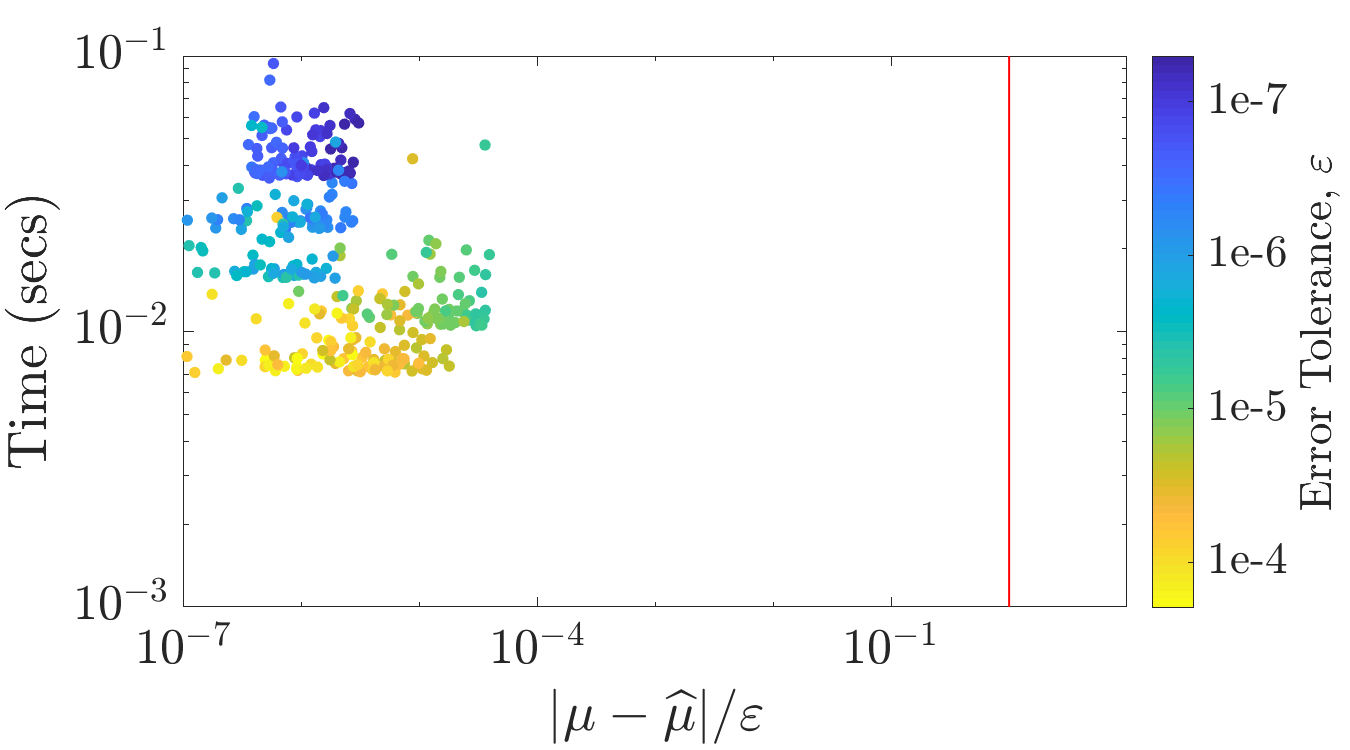}
	\caption[Lattice: MVN guaranteed: MLE]{\code{cubBayesLattice\_g}: Multivariate normal probability example using the empirical Bayes stopping criterion.}
	\label{fig:mvn-guaranteed-MLE}
\end{figure}
\begin{figure}
	\centering
	\includegraphics[width=0.95\linewidth]{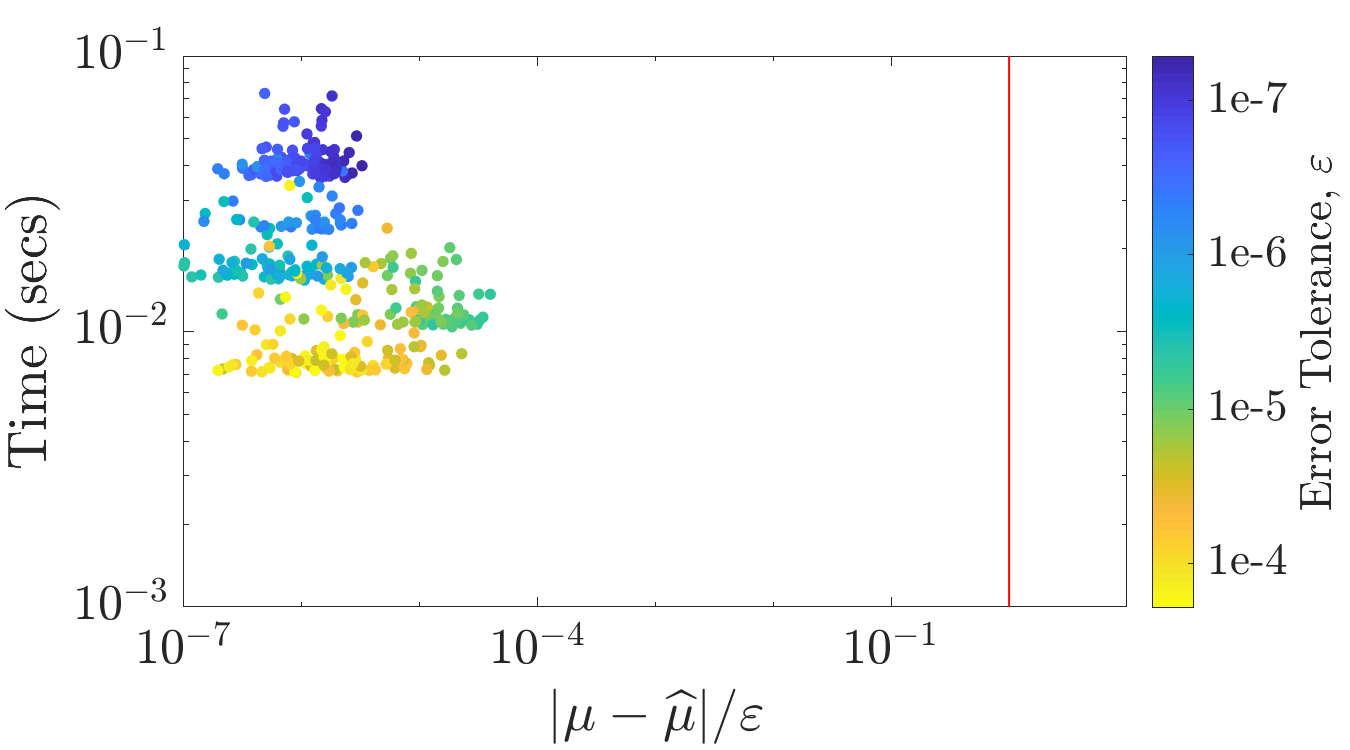}
	\caption[Lattice: MVN guaranteed: Full Bayes]{\code{cubBayesLattice\_g}: Multivariate normal probability example using the full Bayes stopping criterion.}
	\label{fig:mvn-guaranteed-FB}
\end{figure}
\begin{figure}
	\centering
	\includegraphics[width=0.95\linewidth]{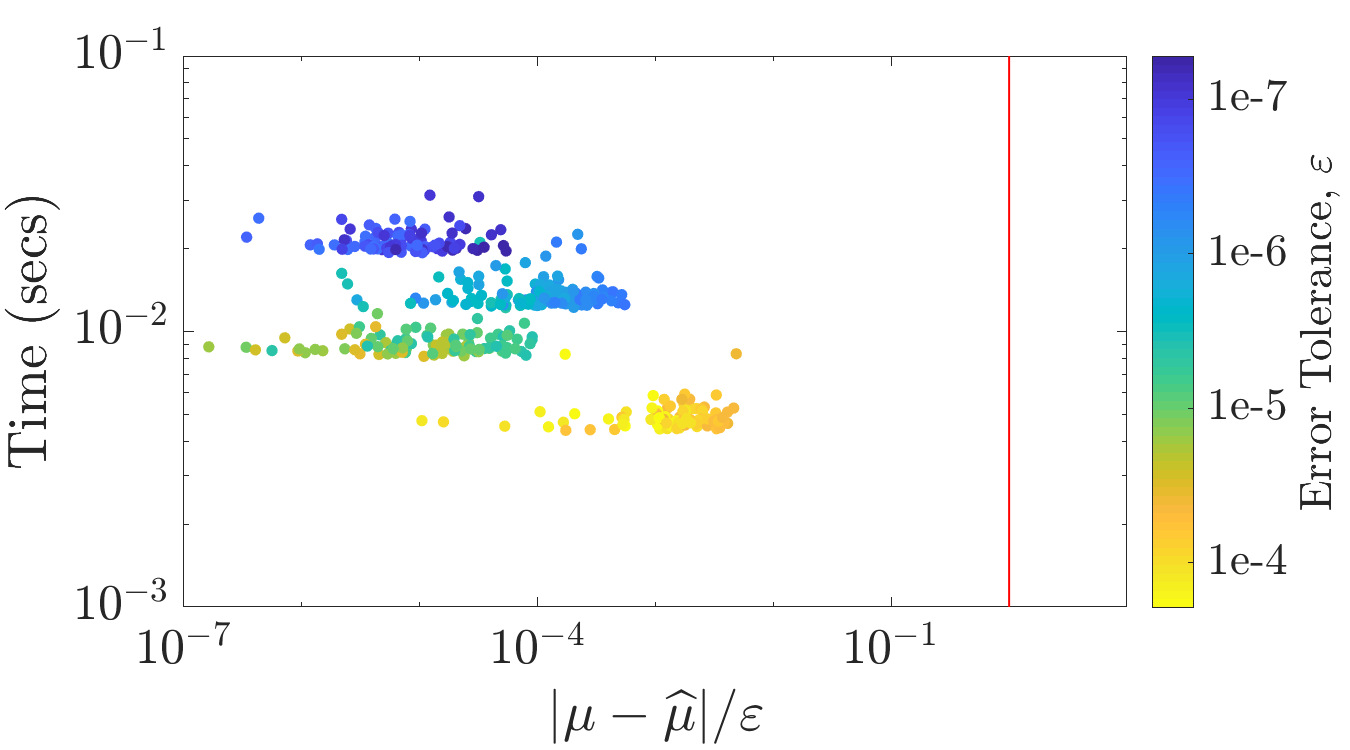}
	\caption[Lattice: MVN guaranteed: GCV]{\code{cubBayesLattice\_g}: Multivariate normal probability example using the GCV stopping criterion.}
	\label{fig:mvn-guaranteed-GCV}
\end{figure}

\Subsection{Using \code{cubBayesNet\_g}}
Here we use $ f_{\textup{Genz}}$ \eqref{eqn:fGenzdef} without any periodization, and chose $d=3$ and $r=1$. The simulation results for this example integrand are summarized in Figures \ref{fig:Sobol-mvn-guaranteed-MLE}, \ref{fig:Sobol-mvn-guaranteed-FB}, and \ref{fig:Sobol-mvn-guaranteed-GCV}.  In all cases, {\code{cubBayesNet\_g} returns an approximation within the prescribed error tolerance. We used the same setting as before with generic slow Bayesian cubature in \secref{MVN_example} for comparision. For error threshold $\varepsilon=10^{-5}$ with empirical stopping criterion, our fast algorithm takes about 2 seconds as shown in \figref{fig:mvn-guaranteed-MLE} whereas the basic algorithm takes 30 seconds as shown in \figref{fig:MVN_Metern_d2b2}. 
{\code{cubBayesNet\_g} uses fast Walsh transform which is slower in MATLAB due to the way it was implemented. This is reason it takes more longer the {\code{cubBayesLattice\_g}. 
But comparing the number of samples, $n$, used for integration provides more insight which directly relates to alogrithm's computational cost. The {\code{cubBayesLattice\_g} used $n=16384$ samples whereas {\code{cubBayesNet\_g} used $n=32768$ samples even with $r=1$ order kernel.

Amongst the three stopping criteria, GCV achieved the results faster than others but it is less conservative. 
One can also observe from the figures that the credible intervals are narrower than in \figref{fig:mvn-guaranteed-MLE}.
This shows that {\code{cubBayesNet\_g} with $r=1$ kernel more accurately approximates the integrand.

\begin{figure}
\centering
\includegraphics[width=0.95\linewidth]{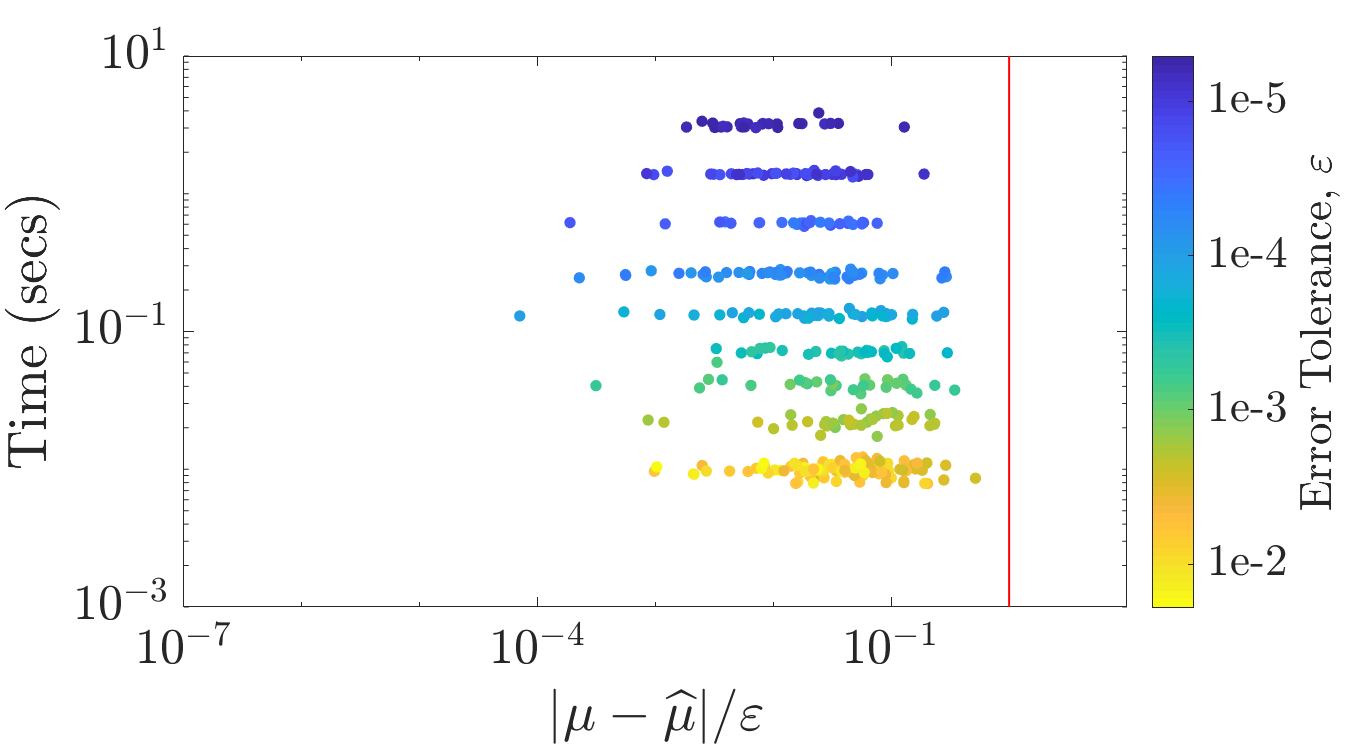}
\caption[Sobol: MVN guaranteed: MLE]{\code{cubBayesNet\_g}: Multivariate normal probability example with empirical Bayes stopping criterion.}
\label{fig:Sobol-mvn-guaranteed-MLE}
\end{figure}
\begin{figure}
\centering
\includegraphics[width=0.95\linewidth]{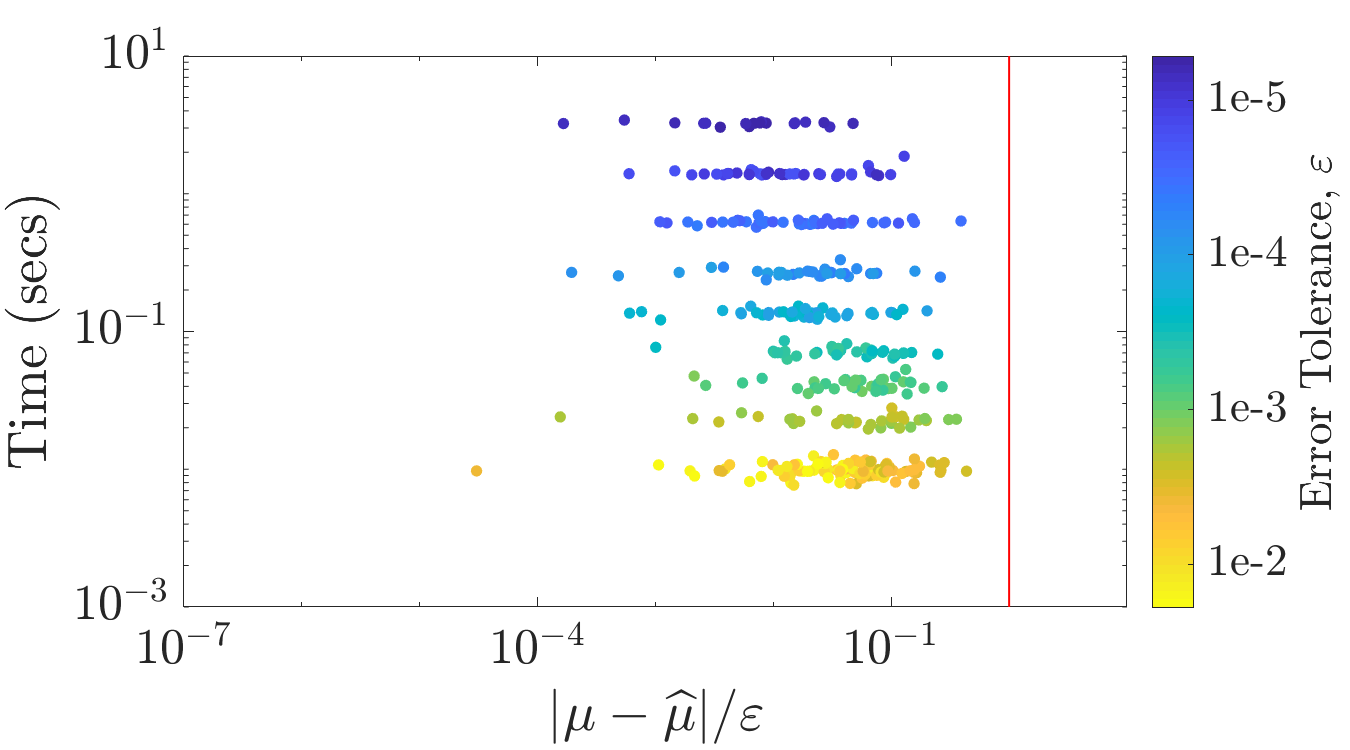}
\caption[Sobol: MVN guaranteed: Full Bayes]{\code{cubBayesNet\_g}: Multivariate normal probability example with the full-Bayes stopping criterion.}
\label{fig:Sobol-mvn-guaranteed-FB}
\end{figure}
\begin{figure}
\centering
\includegraphics[width=0.95\linewidth]{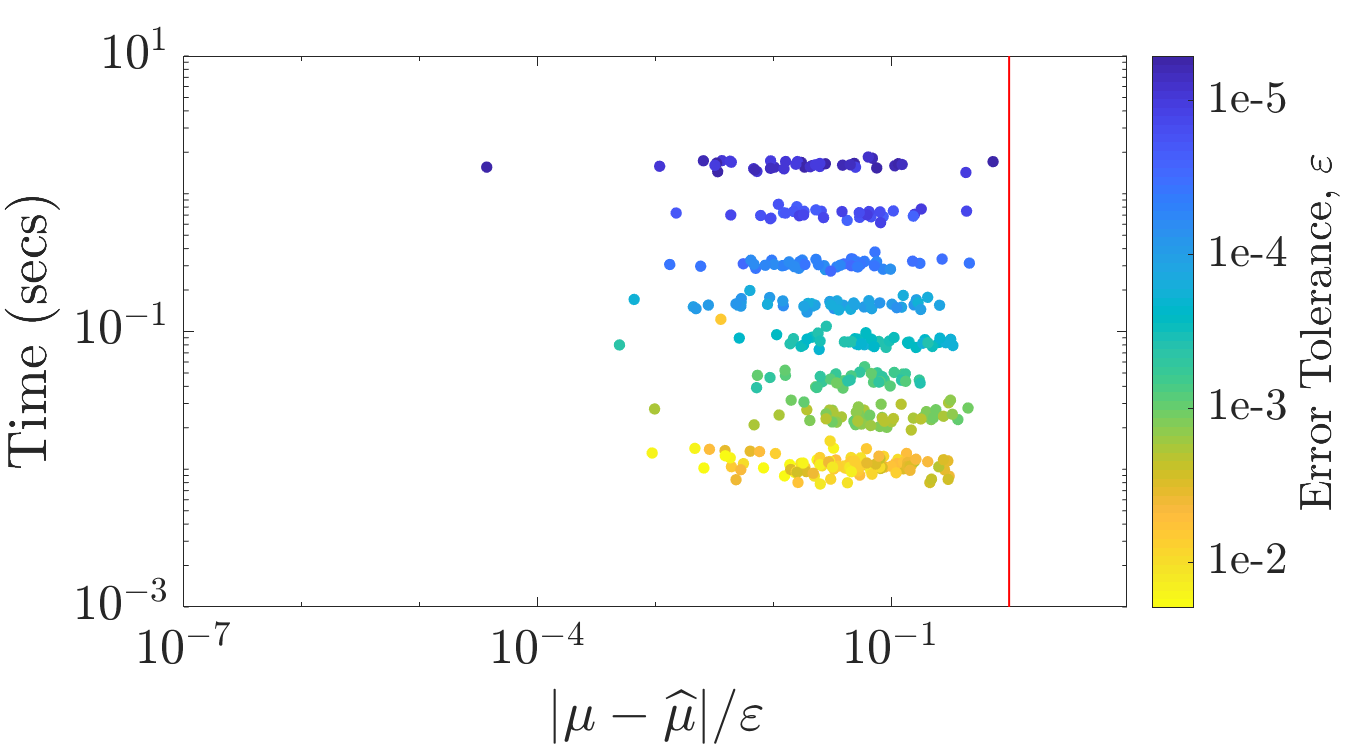}
\caption[Sobol: MVN guaranteed: GCV]{\code{cubBayesNet\_g}: Multivariate normal probability example with the GCV stopping criterion.}
\label{fig:Sobol-mvn-guaranteed-GCV}
\end{figure}

\Section{Keister's Example}

This multidimensional integral function comes from \cite{Kei96} and is inspired by a physics application:
\begin{align}
\label{eqn:keister_integral}
\mu & =  \int_{\reals^d} \cos(\norm{ \vt}) \exp(-\norm{ \vt }^2) \, \dvt \\
\nonumber
& = \int_{[0,1]^d} f_{\textup{Keister}}(\vx) \, \dvx,\\
\intertext{where }
\nonumber
f_\textup{Keister}(\vx) &= \pi^{d/2} \cos\left(\norm{ \Phi^{-1}(\vx)/2}\right)  ,
\end{align}
and $\Phi$ is the standard normal distribution.
The true value of $\mu$ can be calculated iteratively in terms of a quadrature as follows:  
\begin{equation*}
\mu = \frac{2 \pi^{d/2} I_c(d)}{\Gamma(d/2)}, \quad d=1,2, \ldots
\end{equation*}
where $\Gamma$ denotes the gamma function, and
\begin{align*}
I_c(1) &= \frac{\sqrt{\pi}}{2 \exp(1/4)}, 
\\
I_s(1) &= \int_{x=0}^\infty \exp(-\vx^T\vx)\sin(\vx) \, \dvx 
\\
& =  0.4244363835020225,
\\
I_c(2) &= \frac{1-I_s(1)}{2}, \qquad
I_s(2) = \frac{I_c(1)}{2}
\\
I_c(j) &= \frac{(j-2)I_c(j-2)-I_s(j-1)}{2},
\qquad j =3,4,\ldots
\\
I_s(j) &= \frac{(j-2)I_s(j-2)-I_c(j-1)}{2},
\qquad j =3,4,\ldots.
\end{align*}

\Subsection{Using \code{cubBayesLattice\_g}}
\label{sec:cubBayesLattice_keister_example}
\JRNote{Discuss the big gap between 1e-5 and 1e-4}
\begin{figure}
	\centering
	\includegraphics[width=0.95\linewidth]{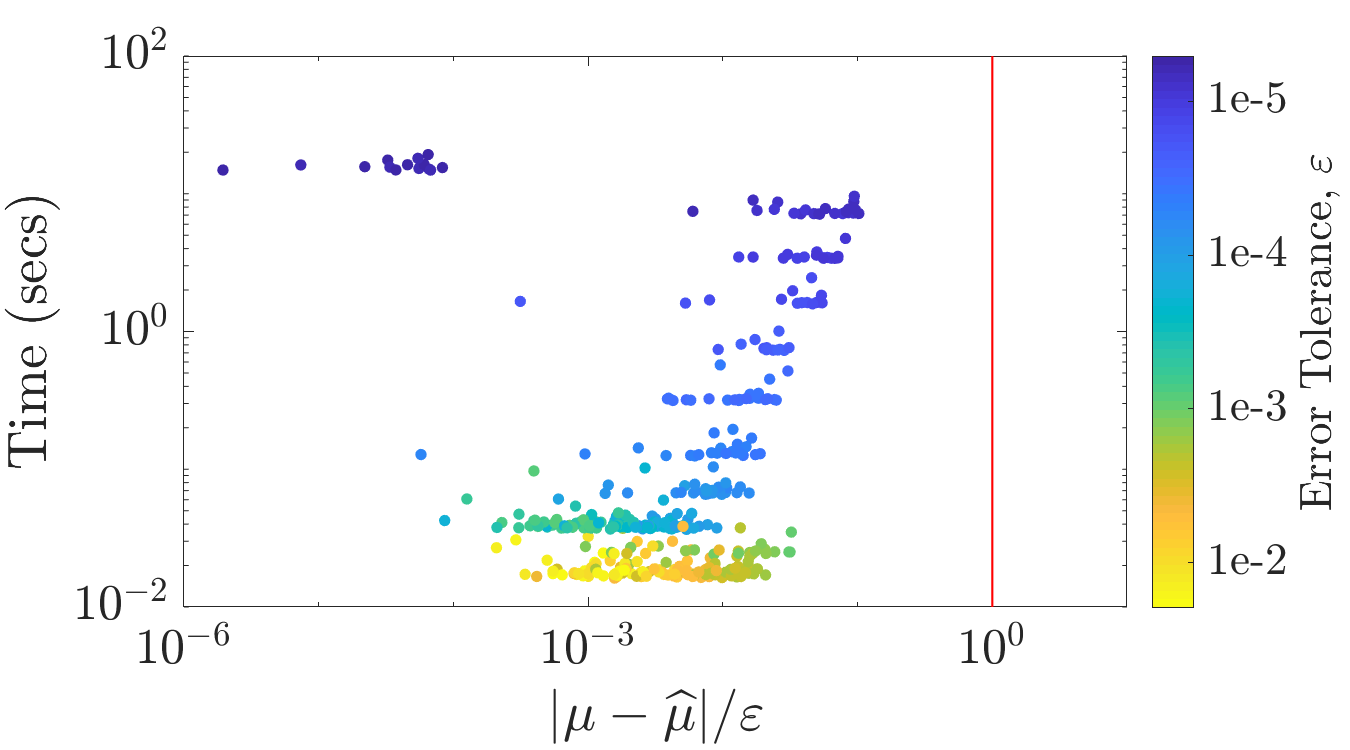}
	\caption[Lattice: Keister guaranteed: MLE]{\code{cubBayesLattice\_g}: Keister example using the empirical Bayes stopping criterion.}
	\label{fig:keister-guaranteed-MLE}
\end{figure}
\begin{figure}
	\centering
	\includegraphics[width=0.95\linewidth]{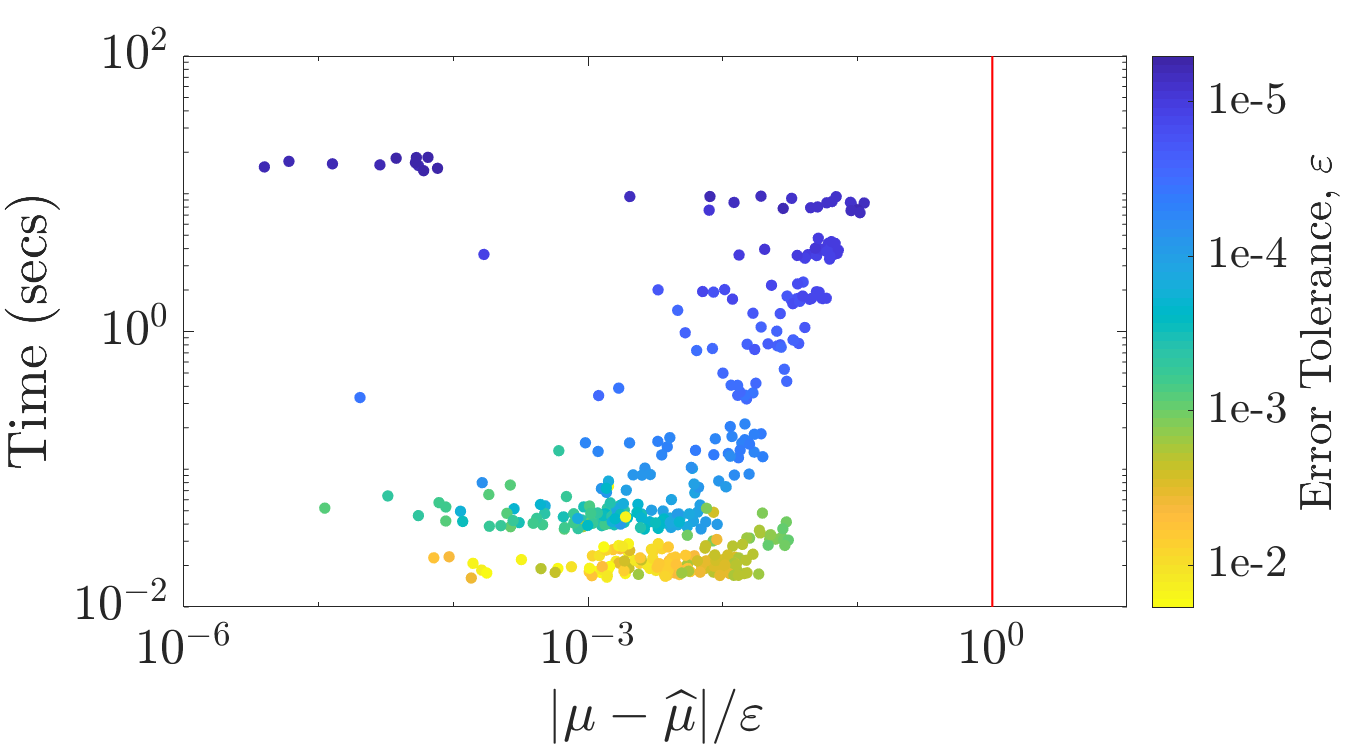}
	\caption[Lattice: Keister guaranteed: Full Bayes]{\code{cubBayesLattice\_g}: Keister example using the full Bayes stopping criterion.}
	\label{fig:keister-guaranteed-FB}
\end{figure}
\begin{figure}
	\centering
	\includegraphics[width=0.95\linewidth]{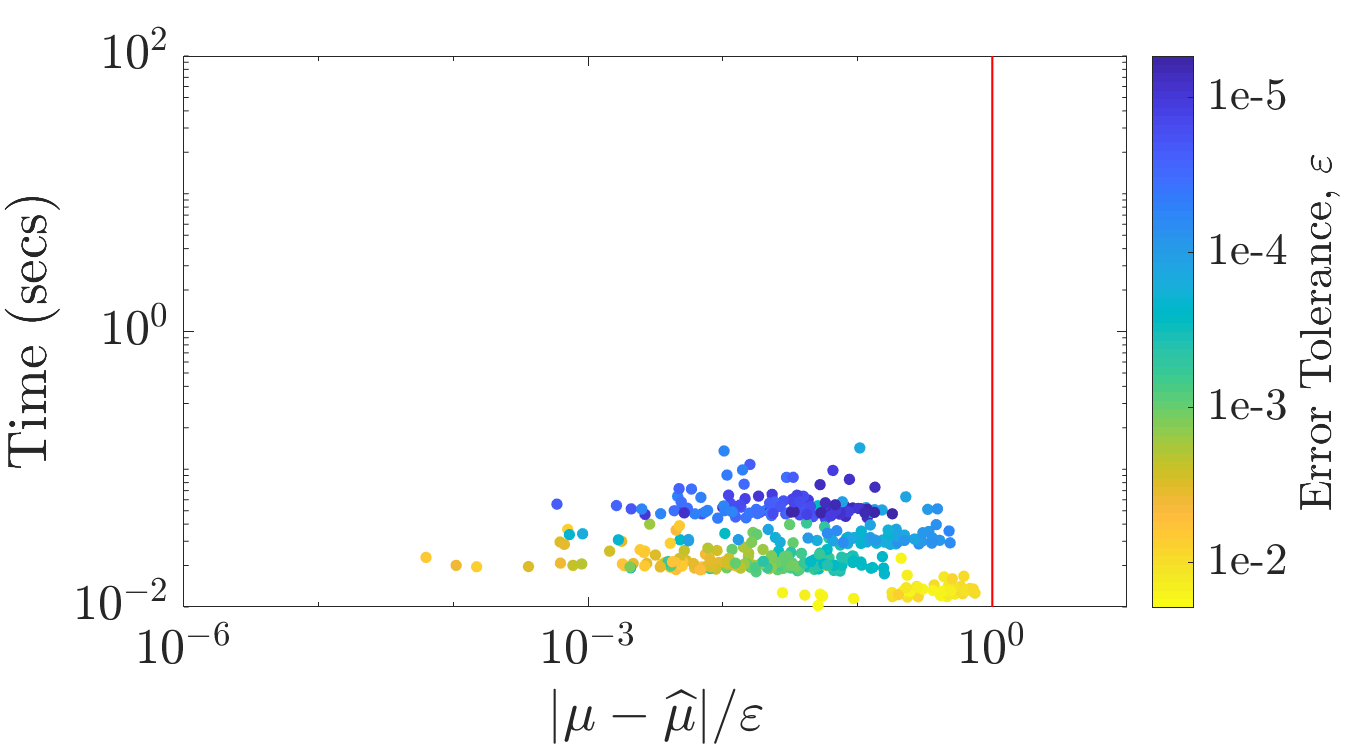}
	\caption[Lattice: Keister guaranteed: GCV]{\code{cubBayesLattice\_g}: Keister example using the GCV stopping criterion.}
	\label{fig:keister-guaranteed-GCV}
\end{figure}
Figures \ref{fig:keister-guaranteed-MLE}, \ref{fig:keister-guaranteed-FB} and \ref{fig:keister-guaranteed-GCV} summarize the numerical tests for this integral.  We used the Sidi's $C^1$ periodization, dimension $d=4$, and $r=2$. 
As we can see the GCV stopping criterion achieved the results faster than the others but it is less conservative similar to the multivariate Gaussian case.

\Subsection{Using \code{cubBayesNet\_g}}
Figures \ref{fig:Sobol-keister-guaranteed-MLE}, \ref{fig:Sobol-keister-guaranteed-FB} and \ref{fig:Sobol-keister-guaranteed-GCV} summarize the numerical tests for this case. We used  dimension $d=4$, and $r=1$.  No periodization transform was used as the integrand need not be periodic. 
In this example, we use $r=1$ order kernel whereas in \secref{sec:cubBayesLattice_keister_example}, $r=2$ kernel was used. This necessitates \code{cubBayesNet\_g} to use more samples for integration.
As observed from the figures, the GCV stopping criterion achieved the results faster than the others but it is less conservative which is also the case with the multivariate Gaussian example.

\begin{figure}
	\centering
	\includegraphics[width=0.95\linewidth]{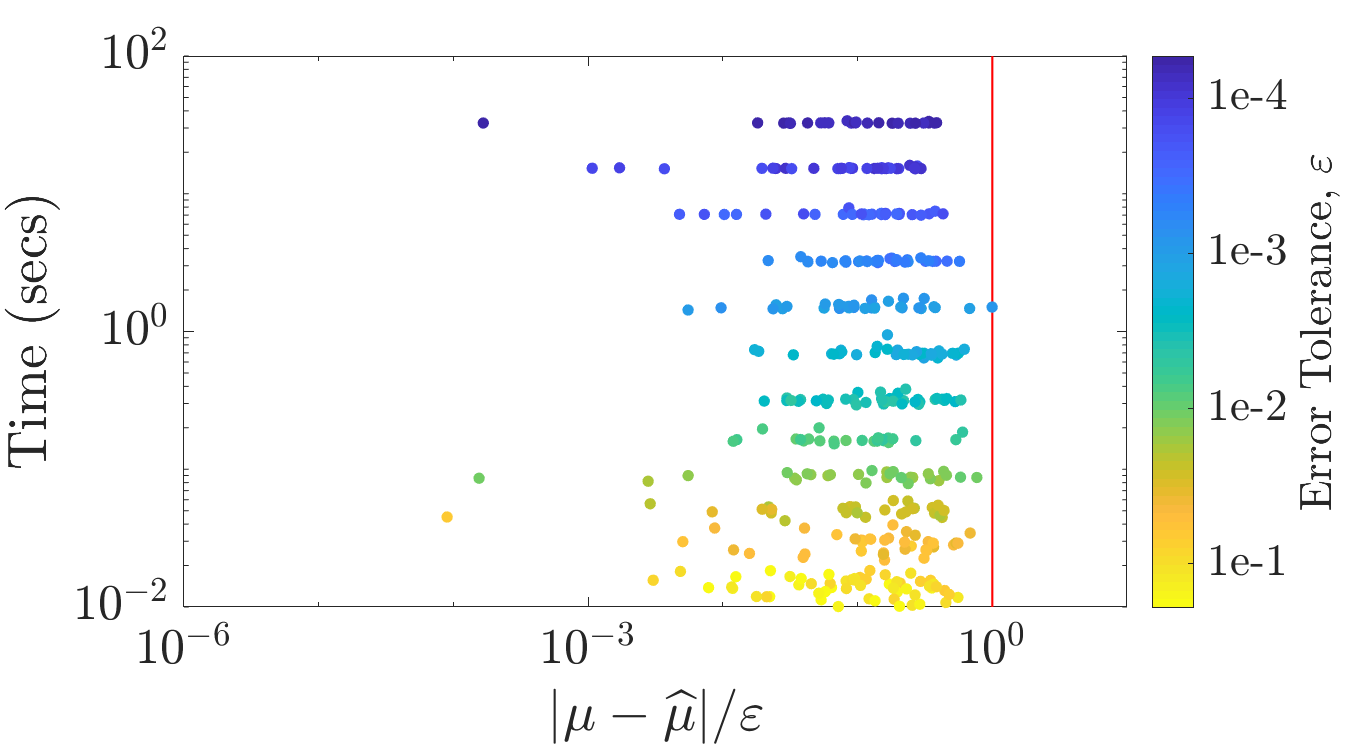}
	\caption[Sobol: Keister guaranteed: MLE]{\code{cubBayesNet\_g}: Keister example using the empirical Bayes stopping criterion.}
	\label{fig:Sobol-keister-guaranteed-MLE}
\end{figure}
\begin{figure}
	\centering
	\includegraphics[width=0.95\linewidth]{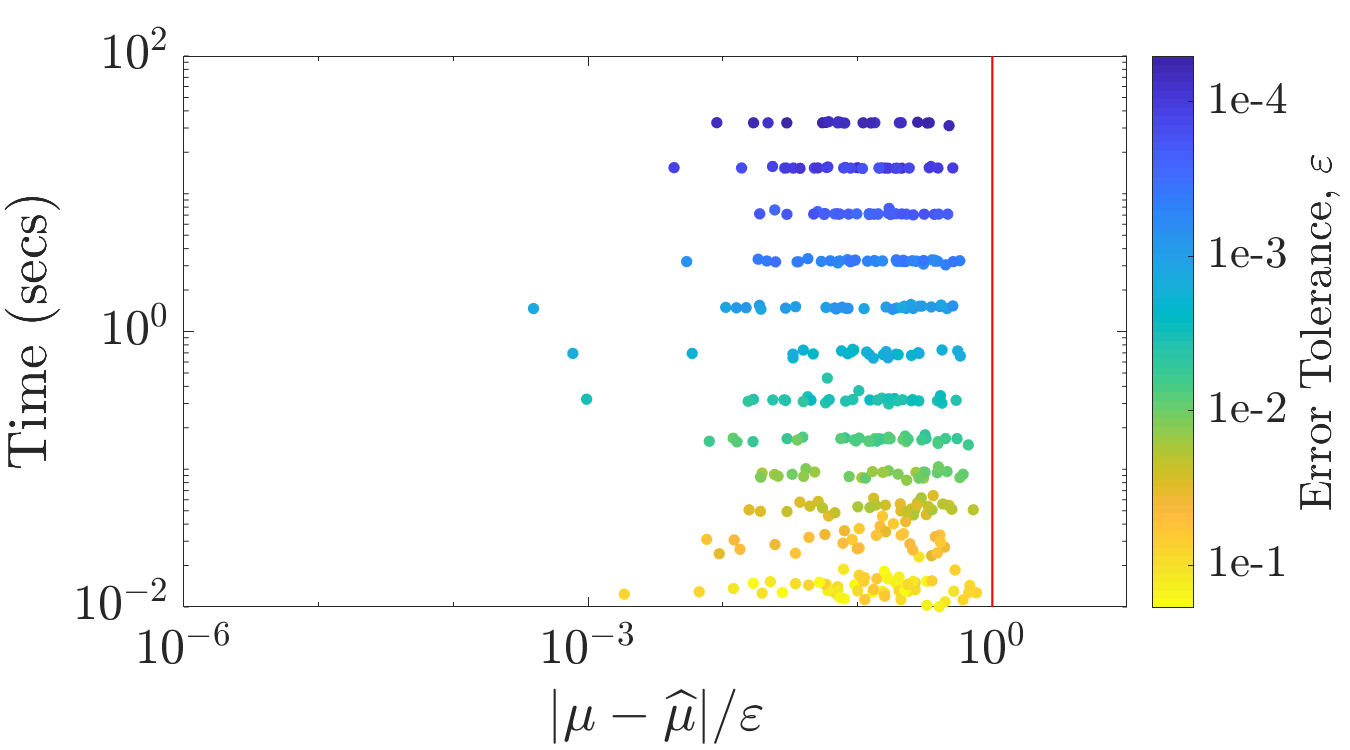}
	\caption[Sobol: Keister guaranteed: Full Bayes]{\code{cubBayesNet\_g}: Keister example using the full-Bayes stopping criterion.}
	\label{fig:Sobol-keister-guaranteed-FB}
\end{figure}
\begin{figure}
	\centering
	\includegraphics[width=0.95\linewidth]{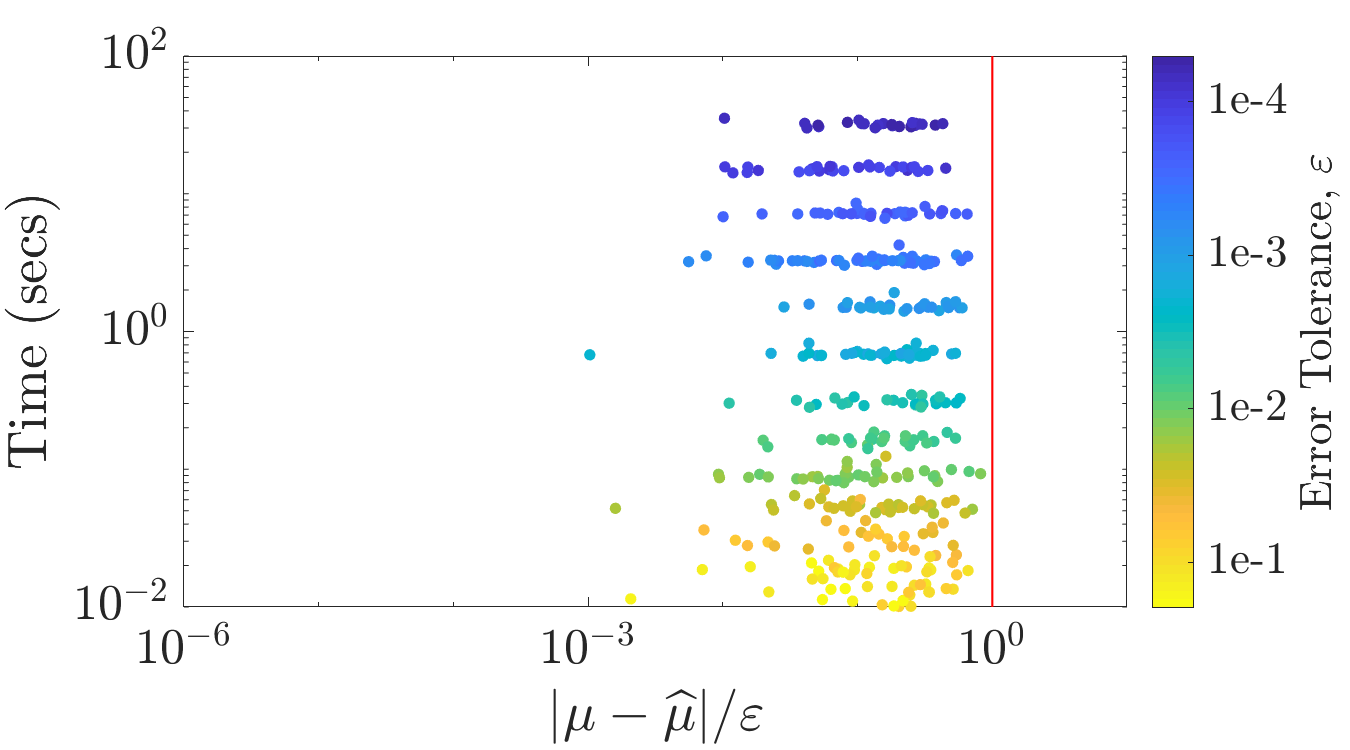}
	\caption[Sobol: Keister guaranteed: GCV]{\code{cubBayesNet\_g}: Keister example using the GCV stopping criterion.}
	\label{fig:Sobol-keister-guaranteed-GCV}
\end{figure}

\Section{Option Pricing}

The price of financial derivatives can often be modeled by high dimensional integrals. If the underlying asset is described in terms of a discretized geometric Brownian motion, then the fair price of the option is:
\begin{equation*}
\mu = \int_{\reals^d} \text{payoff}(\vz) \frac{\exp(\frac 12 \vz^T\mSigma^{-1}\vz)}{\sqrt{(2\pi)^d \det(\mSigma)}} \, \dvz = \int_{[0,1]^d} f(\vx) \, \dvx,
\end{equation*} 
where {payoff($\cdot$)} defines the discounted payoff of the option,
\begin{align*}
\mSigma &= (T/d) \bigl(\min(j,k) \bigr)_{j,k=1}^d = \mL \mL^T,\\
f(\vx) &= \text{payoff} \left(\mL 
\begin{pmatrix}
\Phi^{-1}(x_1) \\ \vdots \\ \Phi^{-1}(x_d)
\end{pmatrix} \right).
\end{align*}
The Asian arithmetic mean call option has a payoff of the form
\begin{align*}
\text{payoff}(\vz) &= \max\left( \frac 1d  \sum_{j=1}^d S_j(\vz) - K, 0 \right) \me^{-r T}, \\
S_j(\vz) &= S_0 \exp\bigl((r-\sigma^2/2)jT/d + \sigma \sqrt{T/d} z_j \bigr).
\end{align*}
Here, $T$ denotes the time to maturity of the option, $d$ the number of time steps, $S_0$ the initial price of the stock, $r$ the interest rate, $\sigma$ the volatility, and $K$ the strike price.  

\Subsection{Using \code{cubBayesLattice\_g}}
The Figures \ref{fig:optprice-guaranteed-MLE}, \ref{fig:optprice-guaranteed-FB} and 
\ref{fig:optprice-guaranteed-GCV} summarize the numerical results for this example using
$
T = 1/4, \ \ d = 13, \ \ S_0 = 100, \ \ r =  0.05, \ \ \sigma = 0.5, \ \ K = 100.
$
Moreover, $\mL$ is chosen to be the matrix of eigenvectors of $\mSigma$ times the square root of the diagonal matrix of eigenvalues of $\mSigma$.
Because the integrand has a kink caused by the $\max$ function, it does not help to use a periodizing transform that is very smooth.  We chose the baker's transform \eqref{eq:bakerTrans} and $r = 1$.

\label{sec:cubBayesLattice_option_pricing_example}
\begin{figure}
	\centering
	\includegraphics[width=0.95\linewidth]{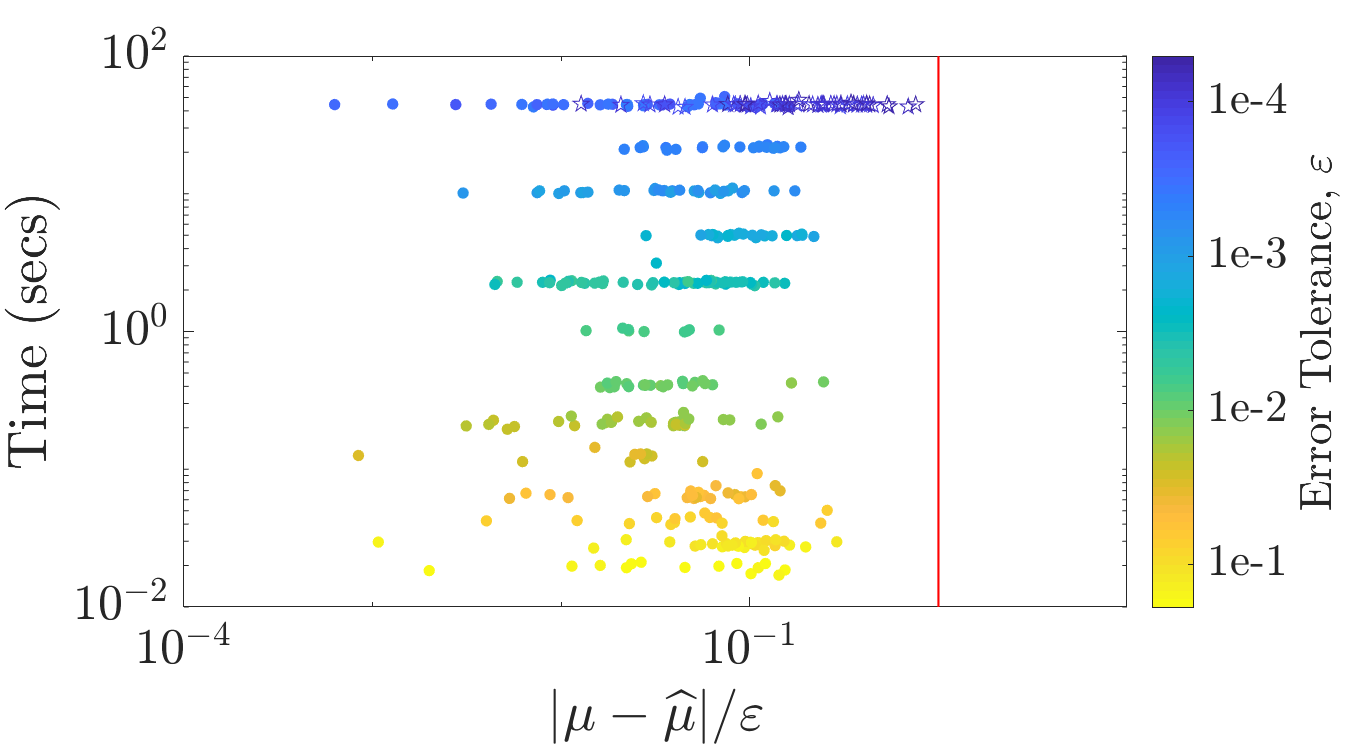}
	\caption[Lattice: Option pricing guaranteed: MLE]{\code{cubBayesLattice\_g}: Option pricing using the empirical Bayes stopping criterion. The hollow stars indicate the algorithm has not met the error threshold $\epsilon$ even with using maximum $n$. }
	\label{fig:optprice-guaranteed-MLE}
\end{figure}
\begin{figure}
	\centering
	\includegraphics[width=0.95\linewidth]{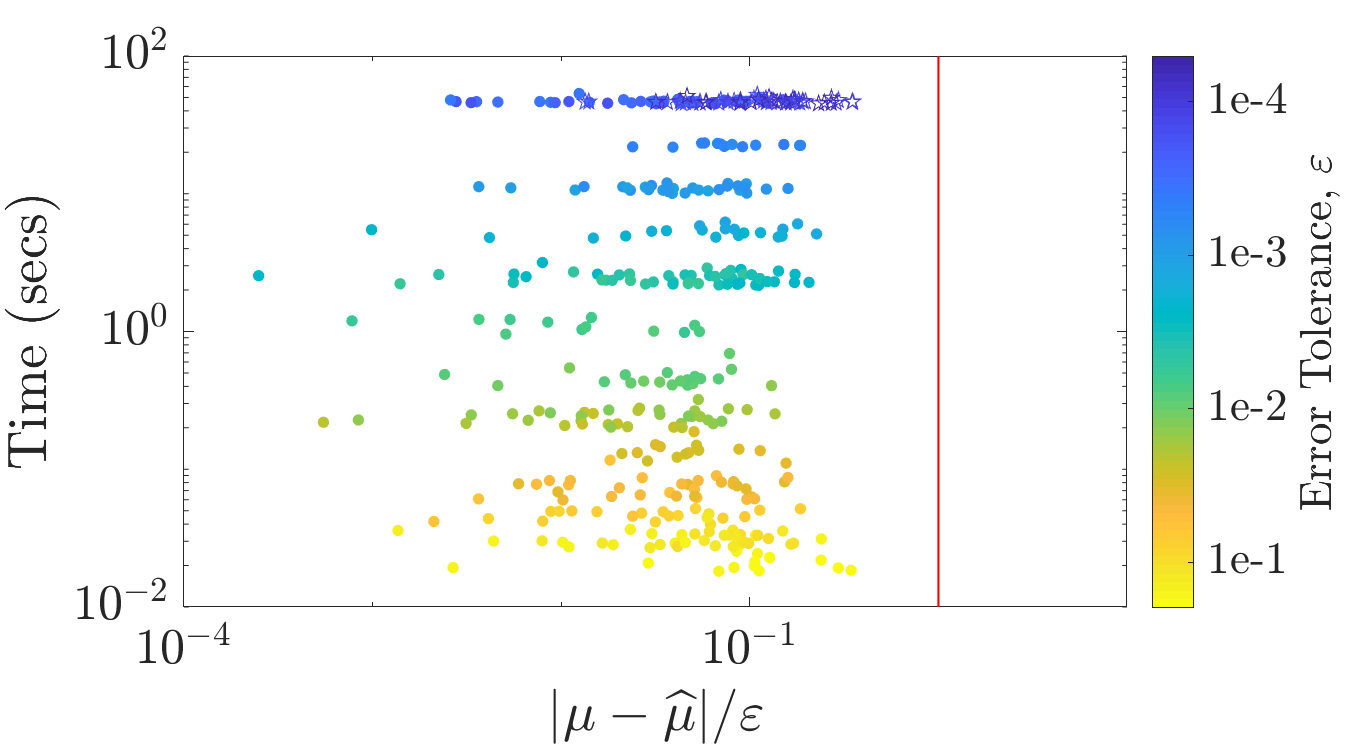}
	\caption[Lattice: Option pricing guaranteed: Full Bayes]{\code{cubBayesLattice\_g}: Option pricing using the full Bayes stopping criterion. The hollow stars indicate the algorithm has not met the error threshold $\epsilon$ even with using maximum $n$.}
	\label{fig:optprice-guaranteed-FB}
\end{figure}
\begin{figure}
	\centering
	\includegraphics[width=0.95\linewidth]{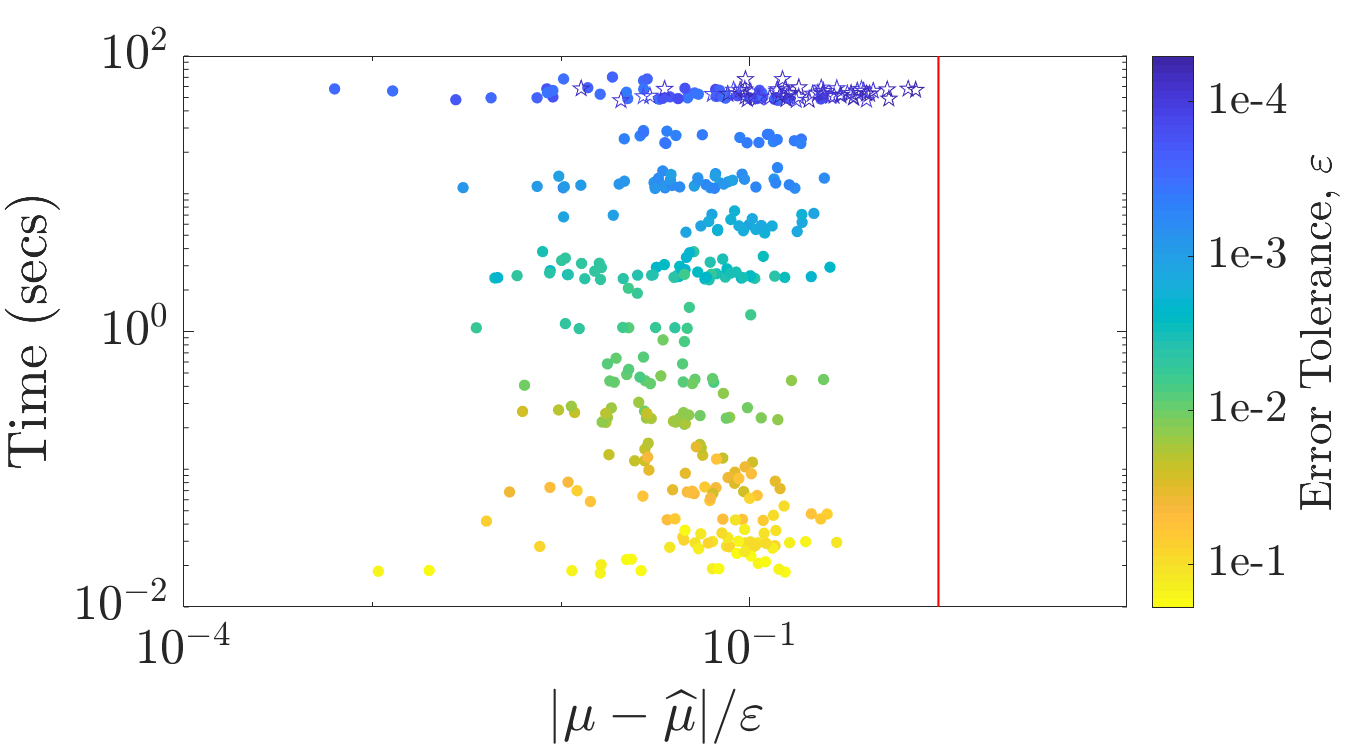}
	\caption[Lattice: Option pricing guaranteed: GCV]{\code{cubBayesLattice\_g}: Option pricing using the GCV stopping criterion. The hollow stars indicate the algorithm has not met the error threshold $\epsilon$ even with using maximum $n$.} 
	\label{fig:optprice-guaranteed-GCV}
\end{figure}

\Subsection{Using \code{cubBayesNet\_g}}
The Figures \ref{fig:Sobol-optprice-guaranteed-MLE}, \ref{fig:Sobol-optprice-guaranteed-FB} and 
\ref{fig:Sobol-optprice-guaranteed-GCV} summarize the numerical results for the option pricing example using the same values for,
$
T, \ \ d, \ \ S_0$, $\ \ r, \ \ \sigma, \ \ K
$, as in \secref{sec:cubBayesLattice_option_pricing_example}.
As mentioned before, this integrand has a kink caused by the $\max$ function, so, \code{cubBayesNet\_g} could be more efficient than \code{cubBayesLattice\_g}, as no periodization transform is required. This can be observed from the number of samples used for intgration to meet the same error threshold. For the error tolerance, $\varepsilon=10^{-3}$,  \code{cubBayesLattice\_g} used $n=2^{20}$ samples, whereas \code{cubBayesNet\_g} used $n=2^{17}$ samples.

\begin{figure}
	\centering
	\includegraphics[width=0.95\linewidth]{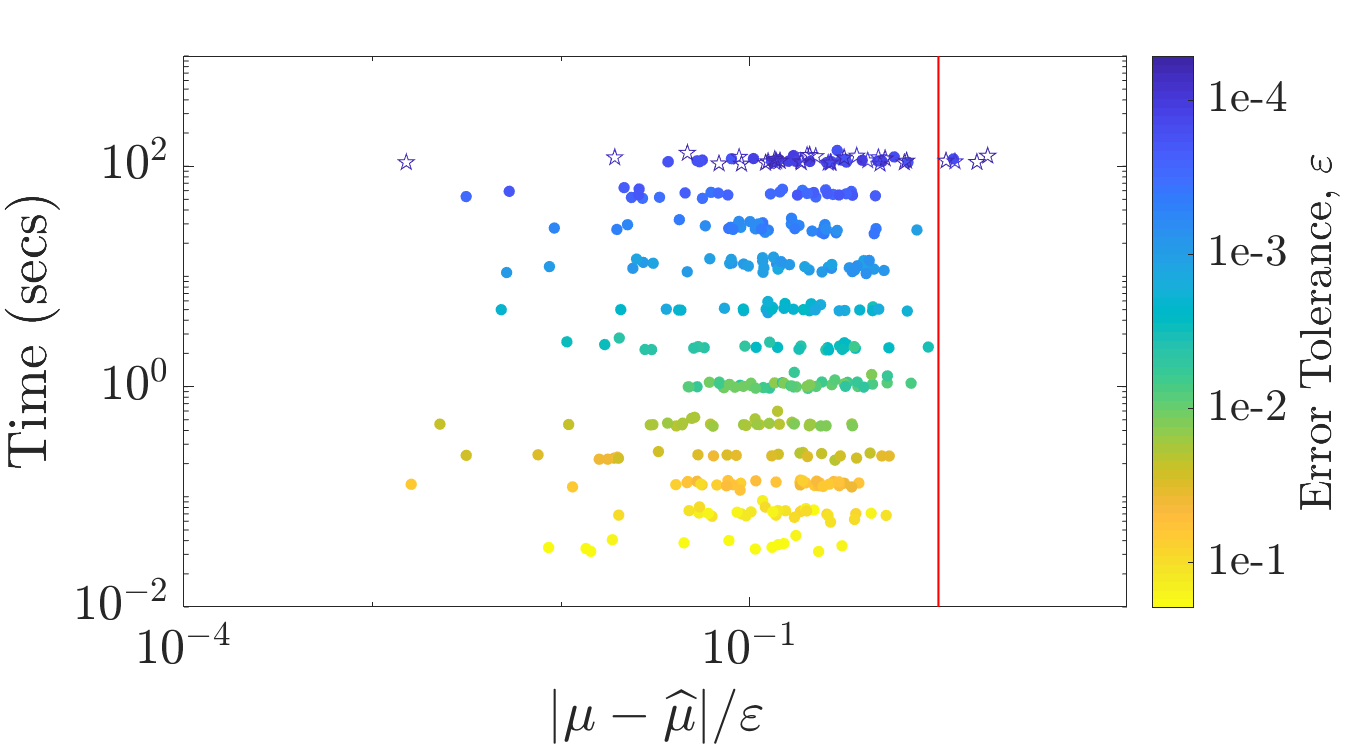}
	\caption[Sobol: Option pricing guaranteed: MLE]{\code{cubBayesNet\_g}: Option pricing using the empirical Bayes stopping criterion. The hollow stars indicate the algorithm has not met the error threshold $\epsilon$ even with using maximum $n$.}
	\label{fig:Sobol-optprice-guaranteed-MLE}
\end{figure}
\begin{figure}
	\centering
	\includegraphics[width=0.95\linewidth]{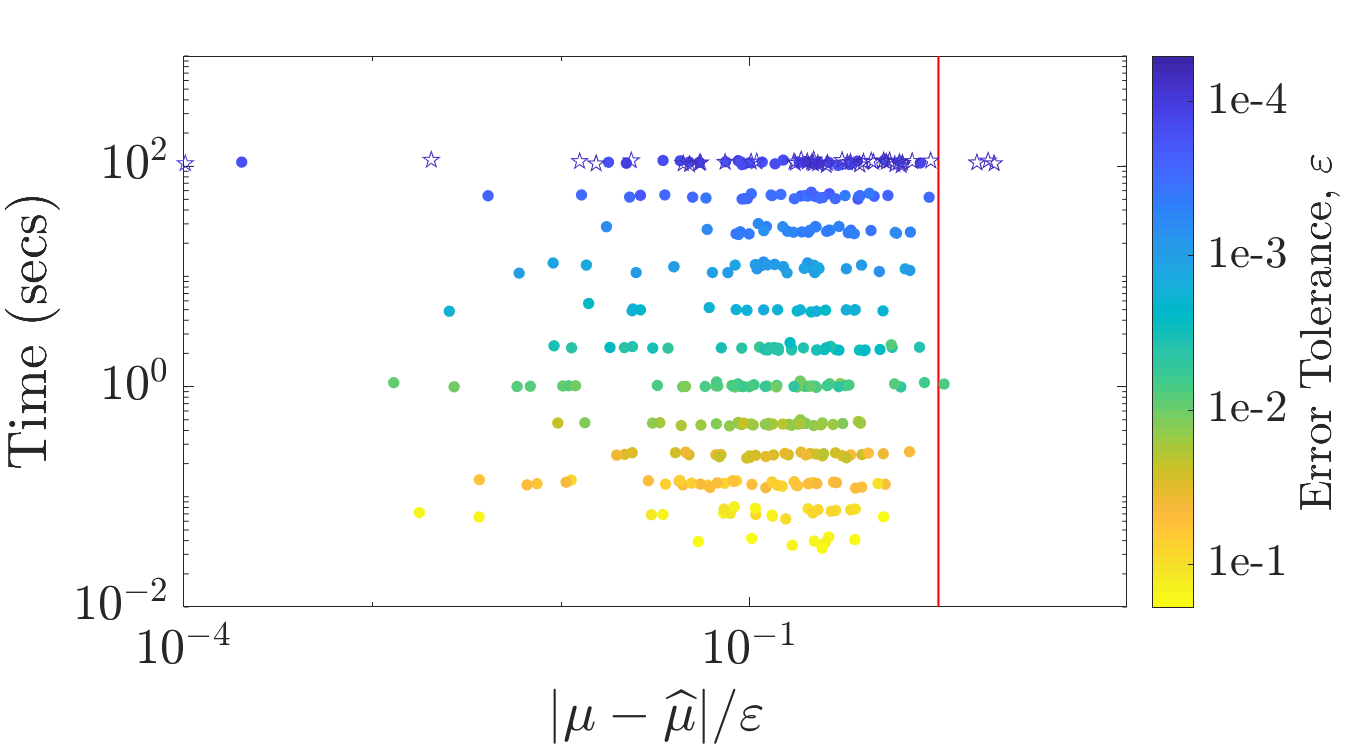}
	\caption[Sobol: Option pricing guaranteed: Full Bayes]{\code{cubBayesNet\_g}: Option pricing using the full-Bayes stopping criterion. The hollow stars indicate the algorithm has not met the error threshold $\epsilon$ even with using maximum $n$.}
	\label{fig:Sobol-optprice-guaranteed-FB}
\end{figure}
\begin{figure}
	\centering
	\includegraphics[width=0.95\linewidth]{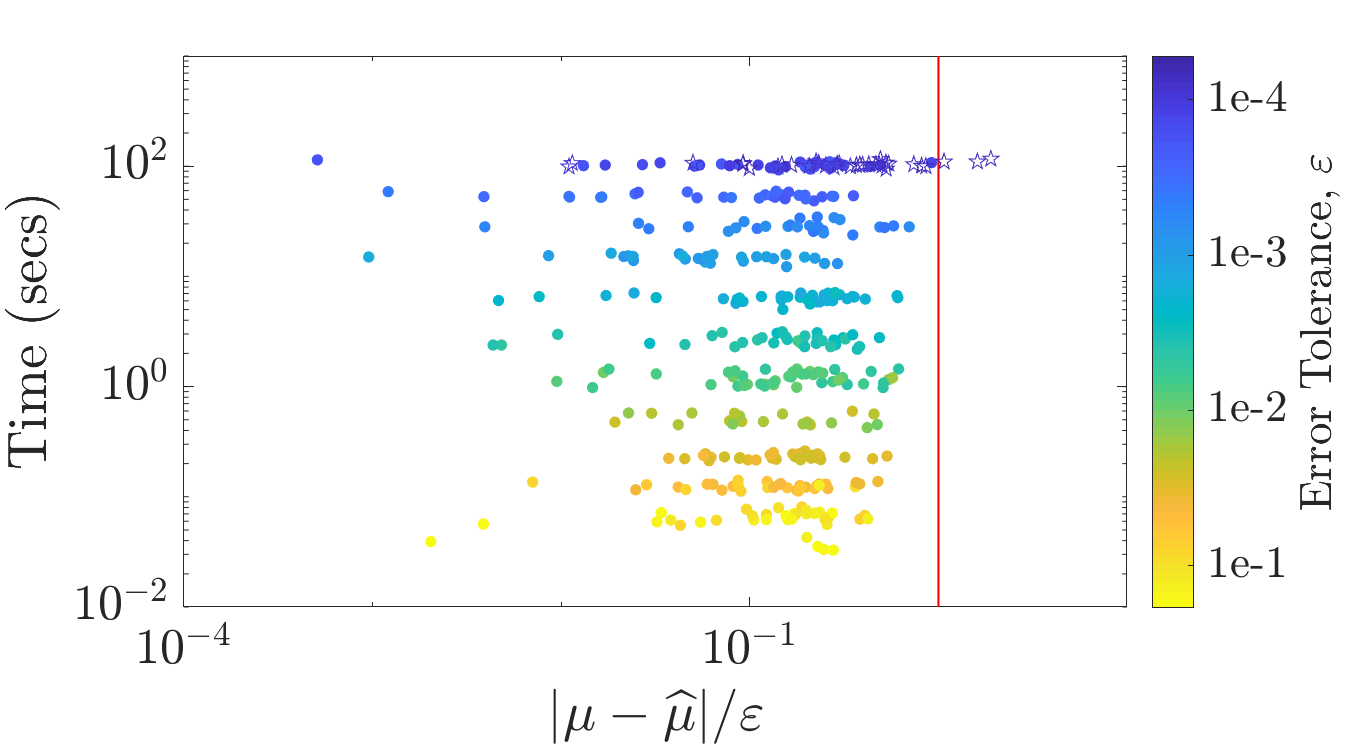}
	\caption[Sobol: Option pricing guaranteed: GCV]{\code{cubBayesNet\_g}: Option pricing using the GCV stopping criterion. The hollow stars indicate the algorithm has not met the error threshold $\epsilon$ even with using maximum $n$.}
	\label{fig:Sobol-optprice-guaranteed-GCV}
\end{figure}

\Section{Discussion}
\JRNote{Move this to end of chapter}

As shown in Figures \ref{fig:mvn-guaranteed-MLE} to \ref{fig:Sobol-optprice-guaranteed-GCV}, both the algorithms computed the integral within user specified threshold most of the time except on a few occasions. This is especially the case with option pricing example due to the complexity and high dimension of the integrand. 
Also notice that the \code{cubBayesLattice\_g} algorithm finished within 10 seconds for Keister and multivariate Gaussian. Option pricing took closer to 70 seconds due to the complexity of the integrand.

Another noticeable aspect from the plots of \code{cubBayesLattice\_g} is how much the error bounds differ from the true error. For option pricing example, the error bound is not as conservative as it is for the multivariate Gaussian and Keister examples. A possible reason is that the latter integrands are significantly smoother than the covariance kernel.  This is a matter for further investigation.


Most noticeable aspect from the plots of \code{cubBayesNet\_g} is how closer the error bounds are to the true error. 
This shows that the \code{cubBayesNet\_g}'s estimation of expected error in the stopping criterion is very accurate. 
Similar to \code{cubBayesLattice\_g}, it missed meeting the given error threshold for the option pricing example, as marked by the hollow stars, for $\varepsilon=10^{-4}$. The algorithm reached max allowed number of samples, $n=2^{20}$ due to the complexity of the integrand.

\Section{Comparison with \code{cubMC\_g}, \code{cubLattice\_g} and \code{cubSobol\_g}}

GAIL library provides variety of numerical integration algorithms based on different theoretical foundations, We would like to compare how our algorithms perform relatively to these. We consider three GAIL algorithms 1) \code{cubMC\_g}, a simple Monte-Carlo method for multi-dimensional integration, 2) \code{cubLattice\_g}, a quasi-Monte-Carlo method using Lattice points, and 3) \code{cubSobol\_g}, a quasi-Monte-Carlo method using Sobol points.

\Subsection{Keister integral}
The Table~\ref{tab2} summarizes the performance of the methods MC, Lattice, Sobol,
BayesLat, and BayesSob---which refer to the GAIL cubatures, \texttt{cubMC\_g},
\code{cubLattice\_g}, \code{cubSobol\_g},  \code{cubBayesLattice\_g},  \code{cubBayesNet\_g},
respectively for estimating Keister integral defined in \eqref{eqn:keister_integral}.
We conducted two simulations with $d=3$ and $8$. In the case of $d=3$, all five methods succeeded completely, meaning, the absolute error is less than given tolerance, i.e., $|\mu - \hat{\mu}| \le
\varepsilon$, where $\hat{\mu}$ is a cubature's approximated value. The
fastest method was \code{cubBayesLattice\_g}.
In the case of $d=8$,   \code{cubSobol\_g} achieved 100\% success rate
and was the fastest. But \code{cubBayesLattice\_g}  was competitive and
had the smallest average absolute error. \code{cubBayesNet\_g} used lowest number of samples but was slower than \code{cubSobol\_g}.

\JRNote{avoid Exp notation for 2 decimals}
\begin{table}[ht] 
\centering
\caption{Comparison of average performance of cubatures for estimating the Keister integral \eqref{eqn:keister_integral} for $1000$ independent runs. 
These results can be conditionally reproduced with the
script, \allowdisplaybreaks \code{KeisterCubatureExampleBayes.m}, in GAIL. 
\label{tab2}}	   
\begin{tabular}{c}
$
\arraycolsep=1.4pt\def\arraystretch{0.9}
\begin{array}{l@{\quad}r@{\quad}r@{\quad}r@{\quad}r@{\quad}r@{\quad}r}
\hhline{======}
\input{figures/KeisterBayesOut.txt} 
\end{array}
$
\end{tabular}
\end{table}

\Subsection{Multivariate Gaussian}
The Table~\ref{tab3} summarizes the performance of the methods MC, Lattice, Sobol,
BayesLat, and BayesSob for estimating the multi-dimensional Gaussian probability $\bf{X}\sim N(\bf{\mu},\Sigma)$. This experiment demonstrates our algorithm's ability to handle  high-dimensional integral.

We conducted two simulations with different $\Sigma$ and estimation intervals $(\bf{a}, \bf{b})$ but fixed $\mu=0$ and required error threshold, $\varepsilon=10^{-3}$. In the first case, all five methods succeeded completely. 
The fastest method was \code{cubBayesLattice\_g} but \code{cubBayesNet\_g} used the lowest number of samples.
In the second case also, all five methods succeeded,  but \code{cubLattice\_g} was the fastest. 
The \code{cubBayesNet\_g}  was competitive and had the smallest average absolute error using lowest number of samples. The \code{cubBayesLattice\_g} achieved the next lowest average error but was slower than \code{cubSobol\_g}.

\begin{table}[ht] 
\centering
\caption{Comparison of average performance of cubatures for estimating the $d=20$ Multivariate Normal \eqref{eqn:fGenzdef} for $1000$ independent runs with $\varepsilon=10^{-3}$. These results can be conditionally reproduced with the script, \code{MVNCubatureExampleBayes.m}, in GAIL. 
\label{tab3}}	   
$
\arraycolsep=1.4pt\def\arraystretch{0.9}
\begin{array}{l@{\quad}r@{\quad}r@{\quad}r@{\quad}r@{\quad}r@{\quad}r}
\hhline{======}
\input{figures/MVNBayesOut.txt} 
\end{array}
$
\end{table}

\clearpage

\Section{Shape Parameter Fine-tuning}

\JRNote{ Numerical examples for the case of shape parameter per dimension }

Allowing the kernel shape parameter to vary for each dimension could improve the accuracy of numerical integration when the integrand under consideration has only very low effective dimension  as in the Option Pricing example we demonstrated. We demonstrate this advantage by integrating a function that is not symmetric across dimensions,
\begin{align}
\label{eqn:fresnels}
f(\vx) = \sum_{j=1}^d \upsilon_j \sin(2 \pi x_j^2)
\end{align} 
which has known integral
\begin{align*}
\int_{[0,1)^d} f(\vx)  = \frac{1}{2} \; \code{fresnels}(d) \sum_{j=1}^{d} \upsilon_j
\end{align*}
where \code{fresnels} is the Fresnel Sine integral,
\begin{align*}
\code{fresnels}(z) = \int_{0}^{z} \sin \left( \frac{\pi t^2}{2} \right) \dt.
\end{align*}

\begin{table}[ht] 
	\centering
	\caption{Comparison of average performance of Bayesian Cubature with common shape parameter vs dimension specific shape parameter
		 for estimating the $d=3$ Fresnel Sine integral. These results can be conditionally reproduced with the script, \code{demoMultiTheta.m}, in GAIL. 
		\label{tabMultiTheta}}	   
	$
	\arraycolsep=1.4pt\def\arraystretch{0.9}
\begin{array}{l@{\quad}r@{\quad}r@{\quad}r@{\quad}r@{\quad}r@{\quad}r}
\hhline{======}
\input{figures/MultiThetaOut.txt} 
\end{array}	
	$
\end{table}

The results are summarized from the two different approaches in Table \ref{tabMultiTheta}. The first method, called \code{OneTheta}, uses common shape parameter across all the dimensions, whereas the second method, called \code{MultiTheta}, allows the shape parameters to vary across the dimensions. In the \code{MultiTheta} method, the shape parameter search is multivariate, so the magnitude of shape parameter depends on the integrand's magnitude in each dimension. We have chosen an integrand particularly to demonstrate this aspect \eqref{eqn:fresnels} where we used $d=3$ and the constants $\bm{\upsilon}= (10^{-4}, 1, 10^4)$. The choice of magnitude variations in constants $\bm{\upsilon}$ allows to make the integrand varies significantly across dimensions.

We ran this test for 1000 times. In comparison, both the methods successfully computed the integral all the time but \code{MultiTheta} was slightly faster. The \code{MultiTheta} method used less number of samples but the integration error was bigger than the \code{OneTheta}. For the same number of samples, the \code{OneTheta} method will be much faster since the shape parameter search is faster. The \code{MultiTheta} method is useful in scenarios where we want to use smaller size, $n$, and the integrand varies significantly across dimensions.

\Chapter{Conclusion and Future Work}
\label{sec:conclusion-future-work}

\Section{Conclusion}

We have developed a fast, automatic Bayesian cubature that estimates the high dimensional integral within a user defined error tolerance  that occur in many scientific computing such as finance, machine learning, imaging, etc.  The stopping criteria arise from assuming the integrand to be a Gaussian process.  In \secref{sec:stopping_criteria}, we developed three criteria:  empirical Bayes, full Bayes, and generalized cross-validation.  Empirical-Bayes uses maximum-likelihood to optimally choose the parameters, where posterior of the parameters given the integrand values is maximized. Alternatively, full-Bayes assumes non-informative prior on the parameters and then computes posterior distribution of the integral $\mu$, which leads to a $t$-distribution to obtain the parameters. Generalized cross-validation extends the concept of cross-validation to construct an objective which in turn is maximized.

The computational cost of the automatic Bayesian cubature can be dramatically reduced if the covariance kernel matches the nodes.  We have demonstrated two such matches in practice. The first algorithm was based on rank-1 lattice nodes and shift-invariant kernels where the matrix-vector multiplications can be accomplished using the fast Fourier Transform.  The second algorithm was based on Sobol' points with first order Walsh kernel where the matrix-vector multiplications can be accomplished using the fast Walsh transform. Three integration problems illustrate the performance of our automatic Bayesian cubature algorithms.  

For faster computations one could use fixed order kernels in \code{cubBayesLattice\_g}, but for more advanced usage, we have added a kernel variation in \secref{sec:non_integer_kernel_order} that allows one to optimally choose the kernel order without the constraint of being an even integer.

During the numerical experiments, we noticed a computation step that causes inaccuracy due to a cancellation error in the estimation of stopping criterion.
We have developed a novel technique in \secref{sec:overcome_cancel_error} to overcome this cancellation error using the inherent structure of the shift-invariant kernel used in our algorithm.

In \secref{sec:deriv_of_kernel}, we have analytically computed the gradient of the objective function and the shift invariant kernel to use with steepest descent in kernel parameters search. 
Quasi-Monte Carlo cubature methods are efficient \cite{SloWoz98} even if the dimension is high given that the effective dimension is low.
To take advantage of low effective dimension, one should not fix the kernel shape parameter across all the dimensions. In this situation, steepest descent methods come in handy as one searches for parameters in multi-dimensions. 

%

\Section{Future Work}

We demonstrated the capability of our new Bayesian cubature algorithms to successfully compute the integrals faster within the user defined error tolerances. But there are possibilities for improvements and new areas of applications.
Some of the improvement ideas are listed here:
\begin{itemize}
\item Higher order digital sequences and digital shift invariant kernels \cite{Nuyens2013} \cite{Bald12a}:
We could improve the computation speed of \code{cubBayesNet\_g} for smoother integrands using higher order digital sequences and matching kernels, which have the potential of being another match that satisfies the conditions in Section~\ref{sec:fast_BC}.  The fast Bayesian transform would correspond to a fast Walsh transform similar to the second algorithm we demonstrated.  For such kernels and the first order Walsh kernel we demonstrated, periodicity is not assumed, however, special structure of both the sequences and the kernels are required to take advantage of integrand smoothness.

\item Control variates:
Hickernell et.al \cite{HicEtal17a} \cite{Li16a} adapted control variates for Quasi-Monte Carlo. 
Control variates are commonly used to improve the efficiency of IID Monte Carlo integration.
One should be able to adapt our Bayesian cubature to control variates, i.e., assuming  
\begin{equation*}
f = \mathcal{GP} \left( \beta_0 + \beta_1 \, g_1 + \cdots + \beta_p \, g_p, \;s^2 C \right),
\end{equation*}
for some choice of vector of functions $\vg = \{g_1, \ldots, g_p\}$, where $\vg : [0,1)^d \to \reals^p$ whose integrals are known $\mu_{\vg} := \int_{[0,1)^d} \vg(\vx)\dvx$, and some parameters $\beta_0, \ldots, \beta_p$ in addition to the $s$ and $C$, then
\begin{align*}
\mu :=
\int_{[0,1)^d} f(\vx) \dx =
\int_{[0,1)^d} h_{\beta} (\vx) \dvx,  \; \text{where} \;
h_{\beta}(\vx) := f(\vx) + \vbeta^T
(\mu_{\vg} - \vg(\vx)).
\end{align*}
Here $\vg$ are the functions on which the QMC method does a good job of integrating it without error. 
The goal is to choose an optimal $\vbeta$ to make
\begin{align*}
\widehat{\mu}_{\vbeta,n} :=
\frac 1n \sum_{i=0}^{n-1} h_{\vbeta} (\vx_i)
\end{align*}
sufficiently close to $\mu$ with the least expense, $n$, possible.
The efficacy of this approach has not yet been explored.

\item Steepest descent: 
The kernels's optimal shape parameter searched using steepest descent with kernels gradient could sometime get into local minima. This needs more understanding and enhancements.
\JRNote{explain why, any suggestions?}

\item Gaussian diagnosis: We assumed the integrand to be an instance of a Gaussian process. One could attempt to prove if this is a good assumption using statistical diagnosis for goodness of fit.

\item Parallel Algorithm: 
For more demanding high performance computing applications, where the precision requirements are high, our algorithms will try to use large number samples leading to longer computation time. One approach to overcome this constraint is to use Parallel computing techniques to speed up the algorithm. Most time consuming parts of our algorithm are shape parameter search and fast Bayesian transform computation. Fast Fourier transform (FFT) and Fast Walsh transform are easily amenable to parallelization. There exist plenty of prior work that can be adapted to work with our algorithms. We use radix-2 FFT. One could use a higher radix FFT to make the computations faster. 

Another area of improvement is the parameter search. We explored the steepest descent algorithm but the speedup was not significant. One could explore higher order algorithms such as Newton method,  which could find the minima faster. 
Fast Bayesian transforms are repeatedly computed in every step of the parameter search if it can be avoided by interpolation or other techniques, this could significantly speedup the algorithm.

One could also use GPU to run the whole code of our Bayesian Cubature algorithms or just the FFT/FWHT part to get a easier speedup.

\end{itemize}

\clearpage

\appendix

\bibliographystyle{IEEEtran}  
\bibliography{FJHown23,FJH23}

\end{document}

%% file: figures/KeisterBayesOut.txt
& \multicolumn{4}{c}{d =   3,\ \varepsilon =  0.005} \\ 
 \hline 
 \text{Method} & \text{MC} & \text{Lattice} & \text{Sobol} & \text{BayesLat} & \text{BayesSobol}  \\ 
 \text{Absolute Error} & \num{0.001100} & \num{0.000510} & \num{0.000520}  & \num{0.000430}  & \num{0.000560}  \\ 
 \text{Tolerance Met} & \num{  100} \% & \num{  100} \% & \num{  100} \% & \num{  100} \% & \num{  100} \%  \\
 n & \num{ 2500000} & \num{    4100} & \num{    3900} & \num{    1000} & \num{    1900} \\
 \text{Time (seconds)} & \num{ 0.1800} & \num{ 0.0069} & \num{ 0.0054} & \num{ 0.0029} & \num{ 0.0700} \\ 
 \\ 
& \multicolumn{4}{c}{d =   8,\ \varepsilon =  0.050} \\ 
 \hline 
 \text{Method} & \text{MC} & \text{Lattice} & \text{Sobol} & \text{BayesLat} & \text{BayesSobol}  \\ 
 \text{Absolute Error} & \num{0.012000} & \num{0.015000} & \num{0.007300}  & \num{0.001800}  & \num{0.008300}  \\ 
 \text{Tolerance Met} & \num{  100} \% & \num{   99} \% & \num{  100} \% & \num{  100} \% & \num{  100} \%  \\
 n & \num{ 7400000} & \num{   15000} & \num{   16000} & \num{   66000} & \num{    8200} \\
 \text{Time (seconds)} & \num{ 1.2000} & \num{ 0.0220} & \num{ 0.0160} & \num{ 0.2100} & \num{ 0.3500} \\ 
 \\ 
 \hline 

%% file: figures/MVNBayesOut.txt
& \multicolumn{6}{c}{\Sigma = \mathsf{I}_d, \ \bm{b}=-\bm{a}=(3.5,\cdots,3.5) } \\ 
\hline 
 \text{Method} & \text{MC} & \text{Lattice} & \text{Sobol} & \text{BayesLat} & \text{BayesSobol}  \\
 \text{Absolute Error} & \num{2.20E-16} & \num{2.70e-14} & \num{2.70e-14}  & \num{2.20e-16}  & \num{2.20e-16}  \\
 \text{Tolerance Met} & \num{  100} \% & \num{  100} \% & \num{  100} \% & \num{  100} \% & \num{  100} \%  \\
 n & \num{   10000} & \num{    1000} & \num{    1000} & \num{    1000} & \num{     260} \\
 \text{Time (seconds)} & \num{ 0.0410} & \num{ 0.0820} & \num{ 0.0710} & \num{ 0.0650} & \num{ 0.0790}  
 \\ \\
& \multicolumn{6}{c}{\Sigma = 0.4 \ \mathsf{I}_{d} + \text{0.6} \ \bm{1} \bm{1}^T , \ \bm{a}=(-\infty,\cdots,-\infty), \ \bm{b} = \sqrt{d} (U_1,\cdots,U_d) } \\ \hline 
 \text{Method} & \text{MC} & \text{Lattice} & \text{Sobol} & \text{BayesLat} & \text{BayesSobol}  \\
 \text{Absolute Error} & \num{2.30e-04} & \num{2.10e-04} & \num{4.40e-04}  & \num{1.00e-04}  & \num{4.80e-05}  \\
 \text{Tolerance Met} & \num{  100} \% & \num{  100} \% & \num{  100} \% & \num{  100} \% & \num{  100} \%  \\
 n & \num{   10000} & \num{    1000} & \num{    1000} & \num{    1000} & \num{     260} \\
 \text{Time (seconds)} & \num{ 0.0350} & \num{ 0.0120} & \num{ 0.0140} & \num{ 0.0150} & \num{ 0.0300}  \\
\hline 

%% file: figures/MultiThetaOut.txt
& \multicolumn{2}{c}{\text{Fresnel Sine Integral in} \; d=3 } \\ 
 \hline 
 \text{Method} & \code{OneTheta} & \code{MultiTheta} \\ 
 \text{Absolute Error} & \num{0.00023} & \num{0.06300}   \\ 
 n & \num{    4100} & \num{     260} \\ 
 \text{Time (seconds)} & \num{ 0.0270} & \num{ 0.0230} \\ 
 \hline

%% file: finalThesis2018.bbl
\def\Ignore#1{}\def\notesupp#1{}\providecommand{\HickernellFJ}{Hickernell\xspace}\def\Ignore#1{}\def\notesupp#1{}\providecommand{\HickernellFJ}{Hickernell\xspace}
\begin{thebibliography}{10}
\providecommand{\url}[1]{#1}
\csname url@samestyle\endcsname
\providecommand{\newblock}{\relax}
\providecommand{\bibinfo}[2]{#2}
\providecommand{\BIBentrySTDinterwordspacing}{\spaceskip=0pt\relax}
\providecommand{\BIBentryALTinterwordstretchfactor}{4}
\providecommand{\BIBentryALTinterwordspacing}{\spaceskip=\fontdimen2\font plus
\BIBentryALTinterwordstretchfactor\fontdimen3\font minus
  \fontdimen4\font\relax}
\providecommand{\BIBforeignlanguage}[2]{{%
\expandafter\ifx\csname l@#1\endcsname\relax
\typeout{** WARNING: IEEEtran.bst: No hyphenation pattern has been}%
\typeout{** loaded for the language `#1'. Using the pattern for}%
\typeout{** the default language instead.}%
\else
\language=\csname l@#1\endcsname
\fi
#2}}
\providecommand{\BIBdecl}{\relax}
\BIBdecl

\bibitem{JagHic09a}
R.~Jagadeeswaran and F.~J. {\HickernellFJ}, ``Fast automatic {B}ayesian
  cubature using lattice sampling,'' \emph{Statist.\ Comp.}, vol.~29, pp.
  1215--1229.

\bibitem{HicJag09a}
F.~J. {\HickernellFJ} and R.~Jagadeeswaran, ``{C}omment on {``P}robabilistic
  integration: A role in statistical computation{?''},'' \emph{Statist.\ Sci},
  vol.~34, pp. 23--28, 2019.

\bibitem{Gla03}
P.~Glasserman, \emph{{M}onte {C}arlo Methods in Financial Engineering}, ser.
  Applications of Mathematics.\hskip 1em plus 0.5em minus 0.4em\relax New York:
  Springer-Verlag, 2004, vol.~53.

\bibitem{Keller2013}
A.~Keller, ``{Q}uasi-{M}onte {C}arlo image synthesis in a nutshell,'' in
  \emph{{M}onte {C}arlo and Quasi-{M}onte {C}arlo Methods {2012}}, ser.
  Springer Proceedings in Mathematics and Statistics, J.~Dick, F.~Y. Kuo, G.~W.
  Peters, and I.~H. Sloan, Eds., vol.~65.\hskip 1em plus 0.5em minus
  0.4em\relax Springer Berlin Heidelberg, 2013, pp. 213--–249.

\bibitem{Goodfellow-et-al-2016}
I.~Goodfellow, Y.~Bengio, and A.~Courville, \emph{Deep Learning}.\hskip 1em
  plus 0.5em minus 0.4em\relax MIT Press, 2016,
  \url{http://www.deeplearningbook.org}.

\bibitem{BecHae92b}
M.~Beckers and A.~Haegemans, ``Transformation of integrands for lattice
  rules,'' in \emph{Numerical Integration: Recent Developments, Software and
  Applications}, T.~O. Espelid and A.~C. Genz, Eds.\hskip 1em plus 0.5em minus
  0.4em\relax Kluwer Academic Publishers, Dordrecht, 1992, pp. 329--340.

\bibitem{Sid08a}
A.~Sidi, ``Further extension of a class of periodizing variable transformations
  for numerical integration,'' \emph{J. Comput.\ Appl.\ Math.}, vol. 221, pp.
  132--149, 2008.

\bibitem{Sid93}
------, ``A new variable transformation for numerical integration,'' in
  \emph{Numerical Integration IV}, ser. International Series of Numerical
  Mathematics, H.~Brass and G.~H\"ammerlin, Eds., no. 112.\hskip 1em plus 0.5em
  minus 0.4em\relax Birkh\"auser, Basel, 1993, pp. 359--373.

\bibitem{Lau96a}
D.~Laurie, ``Periodizing transformations for numerical integration,'' \emph{J.
  Comput.\ Appl.\ Math.}, vol.~66, pp. 337---344, 1996.

\bibitem{CriEtal07}
L.~L. Cristea, J.~Dick, G.~Leobacher, and F.~Pillichshammer, ``The tent
  transformation can improve the convergence rate of quasi-{M}onte {C}arlo
  algorithms using digital nets,'' \emph{Numer. Math.}, vol. 105, pp. 413--455,
  2007.

\bibitem{BriEtal18a}
F.-X. Briol, C.~J. Oates, M.~Girolami, M.~A. Osborne, and D.~Sejdinovic,
  ``Probabilistic integration: A role in statistical computation{?}''
  \emph{Statist.\ Sci.}, 2019, to appear.

\bibitem{Dia88a}
P.~Diaconis, ``{B}ayesian numerical analysis,'' in \emph{Statistical Decision
  Theory and Related Topics {IV}, {P}apers from the 4th {P}urdue Symp., {W}est
  {L}afayette, {I}ndiana 1986}, S.~S. Gupta and J.~O. Berger, Eds.\hskip 1em
  plus 0.5em minus 0.4em\relax Springer-Verlag, New York, 1988, vol.~1, pp.
  163--175.

\bibitem{OHa91a}
A.~{O'H}agan, ``{B}ayes-{H}ermite quadrature,'' \emph{J. Statist.\ Plann.\
  Inference}, vol.~29, pp. 245--260, 1991.

\bibitem{Rit00a}
K.~Ritter, \emph{Average-Case Analysis of Numerical Problems}, ser. Lecture
  Notes in Mathematics.\hskip 1em plus 0.5em minus 0.4em\relax Berlin:
  Springer-Verlag, 2000, vol. 1733.

\bibitem{RasGha03a}
C.~E. Rasmussen and C.~Williams, ``Bayesian {M}onte {C}arlo,'' in
  \emph{Advances in {N}eural {I}nformation {P}rocessing {S}ystems}, S.~Thrun,
  L.~K. Saul, and K.~Obermayer, Eds.\hskip 1em plus 0.5em minus 0.4em\relax MIT
  Press, vol.~15, pp. 489 -- 496.

\bibitem{Hic17a}
F.~J. {\HickernellFJ}, ``The trio identity for quasi-{M}onte {C}arlo error
  analysis,'' in \emph{{M}onte {C}arlo and Quasi-{M}onte {C}arlo Methods:
  {MCQMC}, {S}tanford, USA, {A}ugust 2016}, ser. Springer Proceedings in
  Mathematics and Statistics, P.~Glynn and A.~Owen, Eds.\hskip 1em plus 0.5em
  minus 0.4em\relax Springer-Verlag, Berlin, 2018, pp. 13--37,
  arXiv:1702.01487.

\bibitem{HicJim16a}
F.~J. {\HickernellFJ} and {\relax Ll}.~A. {Jim\'enez Rugama}, ``Reliable
  adaptive cubature using digital sequences,'' in \emph{{M}onte {C}arlo and
  Quasi-{M}onte {C}arlo Methods: {MCQMC}, {L}euven, {B}elgium, {A}pril 2014},
  ser. Springer Proceedings in Mathematics and Statistics, R.~Cools and
  D.~Nuyens, Eds., vol. 163.\hskip 1em plus 0.5em minus 0.4em\relax
  Springer-Verlag, Berlin, 2016, pp. 367--383, arXiv:1410.8615 [math.NA].

\bibitem{JimHic16a}
{\relax Ll}.~A. {Jim\'enez Rugama} and F.~J. {\HickernellFJ}, ``Adaptive
  multidimensional integration based on rank-1 lattices,'' in \emph{{M}onte
  {C}arlo and Quasi-{M}onte {C}arlo Methods: {MCQMC}, {L}euven, {B}elgium,
  {A}pril 2014}, ser. Springer Proceedings in Mathematics and Statistics,
  R.~Cools and D.~Nuyens, Eds., vol. 163.\hskip 1em plus 0.5em minus
  0.4em\relax Springer-Verlag, Berlin, 2016, pp. 407--422, arXiv:1411.1966.

\bibitem{HicEtal17a}
F.~J. {\HickernellFJ}, {\relax Ll}.~A. {Jim\'enez Rugama}, and D.~Li,
  ``Adaptive quasi-{M}onte {C}arlo methods for cubature,'' in
  \emph{Contemporary Computational Mathematics --- a celebration of the 80th
  birthday of {I}an {S}loan}, J.~Dick, F.~Y. Kuo, and H.~Wo\'zniakowski,
  Eds.\hskip 1em plus 0.5em minus 0.4em\relax Springer-Verlag, 2018, pp.
  597--619.

\bibitem{OHagen1991}
A.~O’Hagan, ``Bayes-hermite quadrature,'' \emph{Journal of Statistical
  Planning and Inference}, vol. 29(3), p. 245–260, 1991.

\bibitem{RasWil06a}
C.~E. Rasmussen and C.~Williams, \emph{Gaussian Processes for Machine
  Learning}.\hskip 1em plus 0.5em minus 0.4em\relax Cambridge, Massachusetts:
  MIT Press, 2006, (online version available at {\tt
  http://www.gaussianprocess.org/gpml/}).

\bibitem{Bre73}
R.~Brent, \emph{Algorithms for Minimization Without Derivatives}.\hskip 1em
  plus 0.5em minus 0.4em\relax Prentice-Hall, 1973.

\bibitem{For77}
G.~Forsythe, M.~Malcolm, and C.~Moler, \emph{Computer methods for mathematical
  computations}.\hskip 1em plus 0.5em minus 0.4em\relax Prentice-Hall, 1976.

\bibitem{Dong2017a}
K.~Dong, D.~Eriksson, H.~Nickisch, D.~Bindel, and A.~G. Wilson, ``Scalable log
  determinants for gaussian process kernel learning,'' \emph{NIPS}, 2017, in
  press.

\bibitem{CraWah79a}
P.~Craven and G.~Wahba, ``Smoothing noisy data with spline functions:
  Estimating the correct degree of smoothing by the method of generalized
  cross-validation,'' \emph{Numer.\ Math.}, vol.~31, pp. 307--403, 1979.

\bibitem{GolHeaWah79a}
G.~H. Golub, M.~Heath, and G.~Wahba, ``Generalized cross-validation as a method
  for choosing a good ridge parameter,'' \emph{Technometrics}, vol.~21, pp.
  215--223, 1979.

\bibitem{Wah90}
G.~Wahba, \emph{Spline Models for Observational Data}, ser. CBMS-NSF Regional
  Conference Series in Applied Mathematics.\hskip 1em plus 0.5em minus
  0.4em\relax Philadelphia: SIAM, 1990, vol.~59.

\bibitem{Gen93}
A.~Genz, ``Comparison of methods for the computation of multivariate normal
  probabilities,'' \emph{Computing Science and Statistics}, vol.~25, pp.
  400--405, 1993.

\bibitem{DicEtal14a}
J.~Dick, F.~Kuo, and I.~H. Sloan, ``High dimensional integration --- the
  {Q}uasi-{M}onte {C}arlo way,'' \emph{Acta Numer.}, vol.~22, pp. 133--288,
  2013.

\bibitem{DicPil10a}
J.~Dick and F.~Pillichshammer, \emph{Digital Nets and Sequences: Discrepancy
  Theory and Quasi-{M}onte {C}arlo Integration}.\hskip 1em plus 0.5em minus
  0.4em\relax Cambridge: Cambridge University Press, 2010.

\bibitem{Hig08}
N.~J. Higham, \emph{Functions of matrices: theory and computation}.\hskip 1em
  plus 0.5em minus 0.4em\relax SIAM, 2008.

\bibitem{HicNie03a}
F.~J. {\HickernellFJ} and H.~Niederreiter, ``The existence of good extensible
  rank-1 lattices,'' \emph{J. Complexity}, vol.~19, pp. 286--300, 2003.

\bibitem{Hic96a}
F.~J. {\HickernellFJ}, ``Quadrature error bounds with applications to lattice
  rules,'' \emph{SIAM J. Numer.\ Anal.}, vol.~33, pp. 1995--2016, 1996,
  corrected printing of Sections 3-6 in ibid., {\bf 34} (1997), 853--866.

\bibitem{OlvEtal10a}
\BIBentryALTinterwordspacing
F.~W.~J. Olver, D.~W. Lozier, R.~F. Boisvert, C.~W. Clark, and A.~B.~O.
  Dalhuis, ``Digital library of mathematical functions,'' 2018. [Online].
  Available: \url{http://dlmf.nist.gov/}
\BIBentrySTDinterwordspacing

\bibitem{ChoEtal17b}
\BIBentryALTinterwordspacing
S.-C.~T. Choi, Y.~Ding, F.~J. {\HickernellFJ}, L.~Jiang, {\relax Ll}.~A.
  {Jim\'enez Rugama}, D.~Li, R.~Jagadeeswaran, X.~Tong, K.~Zhang, Y.~Zhang, and
  X.~Zhou, ``{GAIL}: {G}uaranteed {A}utomatic {I}ntegration {L}ibrary (versions
  1.0--2.3),'' MATLAB software, {2013--2019}. [Online]. Available:
  \url{http://gailgithub.github.io/GAIL_Dev/}
\BIBentrySTDinterwordspacing

\bibitem{NuyMagic}
\BIBentryALTinterwordspacing
D.~Nuyens. [Online]. Available:
  \url{https://people.cs.kuleuven.be/~dirk.nuyens/qmc-generators/}
\BIBentrySTDinterwordspacing

\bibitem{Sob67}
I.~M. Sobol', ``The distribution of points in a cube and the approximate
  evaluation of integrals,'' \emph{U.S.S.R. Comput. Math. and Math. Phys.},
  vol.~7, pp. 86--112, 1967.

\bibitem{Nie05a}
H.~Niederreiter, ``Constructions of $(t,m, s)$-nets and $(t, s)$-sequences,''
  \emph{Finite Fields Appl.}, vol.~11, pp. 578--600, 2005.

\bibitem{Bald10a}
J.~F. Baldeaux, ``{H}igher order nets and sequences,'' Ph.D. dissertation, The
  School of Mathematics and Statistics at The University of New South Wales,
  June 2010.

\bibitem{HicYue00}
F.~J. {\HickernellFJ} and R.~X. Yue, ``The mean square discrepancy of scrambled
  $(t,s)$-sequences,'' \emph{SIAM J. Numer.\ Anal.}, vol.~38, pp. 1089--1112,
  2000.

\bibitem{Owe95}
A.~B. Owen, ``Randomly permuted $(t,m,s)$-nets and $(t,s)$-sequences,'' pp.
  299--317.

\bibitem{Mat98}
J.~Matou\v{s}ek, ``On the {$L_2$}-discrepancy for anchored boxes,'' \emph{J.
  Complexity}, vol.~14, pp. 527--556, 1998.

\bibitem{Sob76}
I.~M. Sobol', ``\BIBforeignlanguage{Russian}{Uniformly distributed sequences
  with an additional uniformity property},''
  \emph{\BIBforeignlanguage{Russian}{Zh.\ Vychisl.\ Mat.\ i Mat.\ Fiz.}},
  vol.~16, pp. 1332--1337, 1976.

\bibitem{KuoNuyens2016}
F.~Y. Kuo and D.~Nuyens, ``Application of quasi-{M}onte {C}arlo methods to
  elliptic pdes with random diffusion coefficients — a survey of analysis and
  implementation,'' \emph{Foundations of Computational Mathematics}, vol.
  16(6), pp. 1631--1696, 2016.

\bibitem{NuySoft}
\BIBentryALTinterwordspacing
D.~Nuyens. [Online]. Available:
  \url{https://people.cs.kuleuven.be/~dirk.nuyens/}
\BIBentrySTDinterwordspacing

\bibitem{HonHic00a}
H.~S. Hong and F.~J. {\HickernellFJ}, ``Algorithm 823: Implementing scrambled
  digital nets,'' \emph{ACM Trans.\ Math.\ Software}, vol.~29, pp. 95--109,
  2003.

\bibitem{Nuyens2013}
D.~Nuyens, ``The construction of good lattice rules and polynomial lattice
  rules,'' Aug 2013.

\bibitem{BraFox88}
P.~Bratley and B.~L. Fox, ``Algorithm 659: Implementing {Sobol's} quasirandom
  sequence generator,'' \emph{ACM Trans.\ Math.\ Software}, vol.~14, pp.
  88--100, 1988.

\bibitem{Dic08a}
J.~Dick, ``{W}alsh spaces containing smooth functions an quasi-{M}onte {C}arlo
  rules of arbitrary high order,'' \emph{SIAM J. Numer.\ Anal.}, vol.~46, no.
  1519--1553, 2008.

\bibitem{Kei96}
B.~D. Keister, ``Multidimensional quadrature algorithms,'' \emph{Computers in
  Physics}, vol.~10, pp. 119--122, 1996.

\bibitem{SloWoz98}
I.~H. Sloan and H.~Wo\'zniakowski, ``When are quasi-{M}onte {C}arlo algorithms
  efficient for high dimensional integrals?'' \emph{J. Complexity}, vol.~14,
  pp. 1--33, 1998.

\bibitem{Bald12a}
J.~Baldeaux, J.~Dick, G.~Leobacher, D.~Nuyens, and F.~Pillichshammer,
  ``Efficient calculation of the worst-case error and (fast)
  component-by-component construction of higher order polynomial lattice
  rules,'' \emph{Numerical Algorithms}, vol.~59, pp. 403--431, Mar. 2012.

\bibitem{Li16a}
D.~Li, ``Reliable quasi-{M}onte {C}arlo with control variates,'' Master's
  thesis, Illinois Institute of Technology, 2016.

\bibitem{CooNuy16a}
R.~Cools and D.~Nuyens, Eds., \emph{{M}onte {C}arlo and Quasi-{M}onte {C}arlo
  Methods: {MCQMC}, {L}euven, {B}elgium, {A}pril 2014}, ser. Springer
  Proceedings in Mathematics and Statistics, vol. 163.\hskip 1em plus 0.5em
  minus 0.4em\relax Springer-Verlag, Berlin, 2016.

\end{thebibliography}
